\documentclass[review,onefignum,onetabnum]{siamart171218}

\usepackage{a4wide}
\usepackage{amsfonts}
\usepackage{amscd}
\usepackage{amssymb}
\usepackage{amsmath}
\usepackage{latexsym}
\usepackage{dsfont}
\usepackage{subfigure}
\usepackage{placeins}
\usepackage{mathrsfs}

\usepackage{floatflt}
\usepackage{nicefrac}
\usepackage{MnSymbol}
\usepackage[normalem]{ulem}
\usepackage{pgfplots}
\pgfplotsset{compat=newest}
\usepackage{booktabs}
\usepackage{epstopdf}

\usepackage{mathtools}
\usepackage{esvect} %Vector arrows
\usepackage{blindtext}
\usepackage[font=small,labelfont=bf]{caption}

\definecolor{r}{rgb}{1, 0, 0} 
\definecolor{b}{rgb}{0, 0, 1} 
\definecolor{m}{rgb}{1, 0, 1} 
\definecolor{dg}{rgb}{0.4660, 0.6740, 0.1880} 
\definecolor{dy}{rgb}{0.9290, 0.6940, 0.1250}

\theoremstyle{plain}
\newtheorem{rem}[theorem]{Remark}
\newtheorem{assumption}[theorem]{Assumption}
\newtheorem{example}[theorem]{Example}

{%
\setcounter{enumi}{0}
\renewcommand{\theenumi}{(\textbf{A}.\arabic{enumi})}

\begin{enumerate}}%
{\end{enumerate} }
{\end{enumerate} }

\newcommand{\apprlevy}{{(\varepsilon_l)}}
\newcommand{\apprgrf}{{(\varepsilon_W)}}

\newcommand{\cD}{\mathcal{D}}

\makeatletter
\usepackage{color}

\makeatother

\numberwithin{equation}{section} \allowdisplaybreaks[1]

\definecolor{darkgreen}{rgb}{0,.6,0}

\nolinenumbers
\title{On properties and applications of Gaussian subordinated Lévy fields}

\author{Robin Merkle \thanks{IANS\textbackslash SimTech, University of Stuttgart
(\email{robin.merkle@mathematik.uni-stuttgart.de}).},
\and Andrea Barth \thanks{IANS\textbackslash SimTech, University of Stuttgart  
	(\email{andrea.barth@mathematik.uni-stuttgart.de}).}}

\headers{On properties and applications of Gaussian subordinated Lévy fields}{R. Merkle and A. Barth}

\begin{document}
		
	\maketitle
		
	\begin{abstract}
We consider Gaussian subordinated Lévy fields (GSLFs) that arise by subordinating Lévy processes with positive transformations of Gaussian random fields on some spatial domain $\mathcal{D}\subset \mathbb{R}^d$, $d\geq 1$. The resulting random fields are distributionally flexible and have in general discontinuous sample paths. Theoretical investigations of the random fields include pointwise distributions, possible approximations and their covariance function. As an application, a random elliptic PDE is considered, where the constructed random fields occur in the diffusion coefficient. Further, we present various numerical examples to illustrate our theoretical findings.
	\end{abstract}
	


    %%%%%%%%%%%%%%%%%%%%%%%%%%%%%%%%%%%%%%%%%%%%%%%%%%%%%%%%%%%%%%%%%%%%%%%%%%%%%%%%%
	%%%%%%%%%%%%%%%%%%%%%%%%%%%%%%%%%%%%%%%%%%%%%%%%%%%%%%%%%%%%%%%%%%%%%%%%%%%%%%%%%
	%% Introduction
	%%%%%%%%%%%%%%%%%%%%%%%%%%%%%%%%%%%%%%%%%%%%%%%%%%%%%%%%%%%%%%%%%%%%%%%%%%%%%%%%%
	%%%%%%%%%%%%%%%%%%%%%%%%%%%%%%%%%%%%%%%%%%%%%%%%%%%%%%%%%%%%%%%%%%%%%%%%%%%%%%%%%	
	
	\section{Introduction}\label{sec:Intro}
Ample applications that are modelled stochastically require stochastic processes or random fields, which allow for discontinuities or possess higher distributional flexibility than the standard Gaussian model (see for example \cite{HosseiniWellPosedBayesianInverseProblems}, \cite{LevyProcessesInFinance} and \cite{Zhang2004-dt}). In case of a one-dimensional parameter domain, where the parameter often represents time, Lévy processes are widely used (see e.g. \cite{LevyProcessesInFinance} for some applications in finance). For higher-dimensional parameter domains, some extensions of the Gaussian model have been proposed in the literature: The authors of \cite{Gottschalk2021} consider (smoothed) Lévy noise fields with high distributional flexibility and continuous realizations in the context of random PDEs. In \cite{AStudyOfElliptic}, the authors propose a general random field model which allows for spatial discontinuities with flexible jump geometries. However, in its general form, it not easy to investigate theoretical properties of the random field itself. In the context of Bayesian inversion, the level set approach combined with Gaussian random fields is often used as a discontinuous random field model (see, e.g. \cite{MR3549006}, \cite{MR4119332} and \cite{MR3687326}). The distributional flexibility of the resulting model is, however, again restricted since the stochasticity is governed by the Gaussian field.  Another extension of the Gaussian model has been investigated in the recent paper \cite{SGRF}. The construction is motivated by the subordinated Brownian motion, which is a Brownian motion time-changed by a Lévy subordinator (i.e. a non-decreasing Lévy process). As a generalization, the authors consider Gaussian random fields on a higher-dimensional parameter domain subordinated by several independent Lévy subordinators. The constructed random fields allow for spatial jumps and display great distributional flexibility. However, the jump geometry is restricted in the sense that the jump interfaces are always rectangular.

In this paper, we investigate another specific class of discontinuous random fields: the Gaussian subordinated Lévy fields (GSLF). Motivated by the subordination of standard Lévy processes, the GSLF is constructed by subordination of a general real-valued Lévy process by a transformation of a Gaussian random field. It turns out that the resulting fields allow for spatial discontinuities with flexible jump geometries, are distributionally flexible and relatively easy to simulate. Besides a theoretical investigation of the constructed random fields, we present numerical examples and introduce a possible application of the fields in the diffusion coefficient of an elliptic model PDE. Such a problem arises, for instance, in models for subsurface flow in heterogeneous/porous media (see \cite{AStudyOfElliptic}, \cite{Dagan1991}, \cite{Naff1998} and the references therein).

The rest of the paper is structured as follows: In Section \ref{sec:prelim} we shortly introduce Lévy processes and Gaussian random fields, which are crucial for the definition of the GSLF in Section \ref{sec:GaussianSubordLP}. The pointwise distribution of the random fields are investigated in Section \ref{sec:Chracteristicfunction}, where we derive a formula for its (pointwise) characteristic function. Section \ref{sec:approximation} deals with numerical approximations of the GSLF, including an investigation of the approximation error and the pointwise distribution of the approximated fields. Our theoretical investigations are concluded in Section \ref{sec:Covariance}, where we derive a formula for the covariance function. In Section \ref{sec:numerics}, we present various numerical examples which validate and illustrate the theoretical findings of the previous sections. Section \ref{sec:GSLPellPDE} concludes the paper and highlights the application of the GSLF in the context of random elliptic PDEs, exploring theoretical as well as numerical aspects.

    %%%%%%%%%%%%%%%%%%%%%%%%%%%%%%%%%%%%%%%%%%%%%%%%%%%%%%%%%%%%%%%%%%%%%%%%%%%%%%%%%
	%%%%%%%%%%%%%%%%%%%%%%%%%%%%%%%%%%%%%%%%%%%%%%%%%%%%%%%%%%%%%%%%%%%%%%%%%%%%%%%%%
	%% Preliminaries
	%%%%%%%%%%%%%%%%%%%%%%%%%%%%%%%%%%%%%%%%%%%%%%%%%%%%%%%%%%%%%%%%%%%%%%%%%%%%%%%%%
	%%%%%%%%%%%%%%%%%%%%%%%%%%%%%%%%%%%%%%%%%%%%%%%%%%%%%%%%%%%%%%%%%%%%%%%%%%%%%%%%%	
	
	\section{Preliminaries}\label{sec:prelim}
In the following section, we give a short introduction to Lévy processes and Gaussian random fields following \cite{SGRF} since they are crucial elements for the construction of the Gaussian subordinated Lévy field. For more details we refer the reader to \cite{RandomFieldsAndGeometry, LevyProcessesAndInfinitelyDivisibleDistributions, LevyProcessesAndStochasticCalculus}. Throughout the paper, we assume that $(\Omega,\mathcal{F},\mathbb{P})$ is a complete probability space. 

\subsection{Lévy processes}
We consider an Borel-measurable index set $\mathcal{T}\subseteq \mathbb{R}_+:=[0,+\infty)$. A \textit{stochastic process} $X=(X(t),~t\in \mathcal{T})$ on $\mathcal{T}$ is a family of random variables on the probability space $(\Omega,\mathcal{F},\mathbb{P})$.  
\begin{definition}%(see \cite[Sc. 1.3]{LevyProcessesAndStochasticCalculus})
A stochastic process $l$ on $\mathcal{T}=[0,+\infty)$ is said to be a \textit{Lévy process} if 
\begin{enumerate}
\item $l(0)=0~\mathbb{P}-a.s.$,
\item $l$ has independent increments, i.e. for each $0\leq t_1\leq t_2\leq \dots \leq t_{n+1}$ the random variables $(l(t_{j+1}) -l(t_j),~1\leq j\leq n)$ are mutually independent,
\item $l$ has stationary increments, i.e. for each $0\leq t_1\leq t_2\leq \dots \leq t_{n+1}$ it holds $l(t_{j+1}) - l(t_j) \stackrel{\cD}{=} l(t_{j+1}-t_j) - l(0)\stackrel{\cD}{=} l(t_{j+1}-t_j)$, where $\stackrel{\cD}{=}$ denotes equivalence in distribution,
\item $l$ is stochastically continuous, i.e. for all $a>0$ and for all $s\geq 0$ it holds 
\begin{align*}
\underset{t\rightarrow s}{\lim}\, \mathbb{P}(|l(t)-l(s)|>a)=0.
\end{align*}
\end{enumerate}
\end{definition} 	
 The well known \textit{Lévy-Khinchin formula} yields a parametrization of the class of Lévy processes by the so called \textit{Lévy triplet} $(\gamma,b,\nu)$.
\begin{theorem}(Lévy-Khinchin formula, see \cite[Th. 1.3.3 and p. 29]{LevyProcessesAndStochasticCalculus})\label{TH:LevyKhinchinFormula1d}
Let $l$ be a real-valued Lévy process on $\mathcal{T}\subseteq \mathbb{R}_+:=[0,+\infty)$. There exist parameters $\gamma\in \mathbb{R}$, $b\in \mathbb{R}_+$ and a measure $\nu$ on $\mathbb{R}$ such that the characteristic function $\phi_{l(t)}$ of the Lévy process $l$ admits the representation
\begin{align*}
\phi_{l(t)}(\xi) := \mathbb{E}(\exp(i\xi l(t))) = \exp(t\psi(\xi)),~\xi\in \mathbb{R},
\end{align*}
for $t\in \mathcal{T}$. Here, $\psi$ denotes the characteristic exponent of $l$ which is given by 
\begin{align*}
\psi(\xi) = i\gamma \xi - \frac{b}{2}\xi^2 + \int_{\mathbb{R}\setminus\{0\}}e^{i\xi y}-1-i\xi y\mathds{1}_{\{|y|\leq 1\}}(y)\nu(dy),~\xi\in\mathbb{R}.
\end{align*}
Further, the measure $\nu$ satisfies
\begin{align*}
\int_\mathbb{R}	 \min(y^2,1) \,\nu(dy)<\infty,
\end{align*}
and is called \textit{Lévy measure} and $(\gamma,b,\nu)$ \textit{L\'evy triplet}.
\end{theorem}

\subsection{Gaussian random fields}\label{Subsec:GRF}
A random field defined over the Borel set $\mathcal{D}\subset \mathbb{R}^d$ is a family of real-valued random variables on the probability space $(\Omega,\mathcal{F},\mathbb{P})$ parametrized by $\underline{x}\in\mathcal{D}$. The Gaussian random field  defines an important example.	
	\begin{definition}(see \cite[Sc. 1.2]{RandomFieldsAndGeometry})
	A random field $(W(\underline{x}),~\underline{x}\in\mathcal{D})$ on a $d$-dimensional domain $\mathcal{D}\subset \mathbb{R}^d$ is said to be a \textit{Gaussian random field (GRF)} if, for any $\underline{x}^{(1)},\dots,\underline{x}^{(n)} \in \mathcal{D}$ with $n\in \mathbb{N}$, the $n$-dimensional random variable $(W(\underline{x}^{(1)}),\dots,W(\underline{x}^{(n)}))$ follows a multivariate Gaussian distribution. In this case, we define the \textit{mean function} $\mu_W(\underline{x}^{(i)}):=\mathbb{E}(W(\underline{x}^{(i)}))$ and the covariance function $q_W(\underline{x}^{(i)},\underline{x}^{(j)}):=\mathbb{E}((W(\underline{x}^{(i)})-\mu_W(\underline{x}^{(i)}))(W(\underline{x}^{(j)})-\mu_W(\underline{x}^{(j)})))$, for $\underline{x}^{(i)},~\underline{x}^{(j)}\in \mathcal{D}$. The GRF $W$ is called centered, if $\mu_W(\underline{x})=0$ for all $\underline{x}\in \mathcal{D}$. 
	\end{definition}
	Note that every GRF is determined uniquely by its mean and covariance function. The GRFs considered in this paper are assumed to be mean-square continuous, which is a common assumption (cf. \cite{RandomFieldsAndGeometry}). We denote by $Q:L^2(\mathcal{D})\rightarrow L^2(\mathcal{D})$ the \textit{covariance operator} of $W$ which is defined by 
\begin{align*}
Q(\psi)(\underline{x})=\int_{\mathcal{D}}q_W(\underline{x},\underline{y})\psi(\underline{y})d\underline{y} \text{, for } \underline{x}\in \mathcal{D},
\end{align*}
for $\psi \in L^2(\mathcal{D})$. Here, $L^2(\mathcal{D})$ denotes the Lebesgue space of all square integrable functions over $\mathcal{D}$ (see, for example, \cite{SobolevSpaces}). If $\mathcal{D}$ is compact, it is well known that there exists a decreasing sequence $(\lambda_i,~i\in \mathbb{N})$ of real eigenvalues of $Q$ with corresponding eigenfunctions $(e_i,~i\in\mathbb{N})\subset L^2(\mathcal{D})$ which form an orthonormal basis of $L^2(\mathcal{D})$ (see \cite[Section 3.2]{RandomFieldsAndGeometry} and \cite[Theorem VI.3.2 and Chapter II.3]{Funktionalanalysis}). A GRF $W$ is called \textit{stationary} if the mean function $\mu$ is constant and the covariance function $q_W(\underline{x}^{(1)},\underline{x}^{(2)})$ only depends on the difference $\underline{x}^{(1)}-\underline{x}^{(2)}$ of the values $\underline{x}^{(1)},~\underline{x}^{(2)}\in \mathcal{D}$ and a stationary GRF $W$ is called \textit{isotropic} if the covariance function $q_W(\underline{x}^{(1)},\underline{x}^{(2)})$ only depends on the Euclidean length $\|\underline{x}^{(1)}-\underline{x}^{(2)}\|_2$ of the difference of the values $\underline{x}^{(1)},~\underline{x}^{(2)}\in \mathcal{D}$ (see \cite{RandomFieldsAndGeometry}, p. 102 and p. 115).

\begin{example}
The Matérn-GRFs are a class of continuous GRFs which are commonly used in applications. For a certain smoothness parameter $\nu > 1/2$, correlation parameter $r>0$ and variance $\sigma^2>0$, the Matérn-$\nu$ covariance function is defined by $q_M(\underline{x},\underline{y})=\rho_M(\|\underline{x}-\underline{y}\|_2)$, for $(\underline{x},\underline{y})\in\mathbb{R}_+^d\times \mathbb{R}_+^d$, with 
\begin{align*}
\rho_M(s) = \sigma^2 \frac{2^{1-\nu}}{\Gamma(\nu)}\Big(\frac{2s\sqrt{\nu}}{r}\Big)^\nu K_\nu\Big(\frac{2s\sqrt{\nu}}{r}\Big), \text{ for }s\geq 0.
\end{align*}
Here, $\Gamma(\cdot)$ is the Gamma function and $K_\nu(\cdot)$ is the modified Bessel function of the second kind (see \cite[Section 2.2 and Proposition 1]{QuasiMonteCarloFEMethodsForEllipticPDEsWithLognormalRandomCoefficients}). A Matérn-$\nu$ GRF is a centered GRF with covariance function $q_M$.
\end{example}

    %%%%%%%%%%%%%%%%%%%%%%%%%%%%%%%%%%%%%%%%%%%%%%%%%%%%%%%%%%%%%%%%%%%%%%%%%%%%%%%%%
	%%%%%%%%%%%%%%%%%%%%%%%%%%%%%%%%%%%%%%%%%%%%%%%%%%%%%%%%%%%%%%%%%%%%%%%%%%%%%%%%%
	%% The Gaussian subordinated LP
	%%%%%%%%%%%%%%%%%%%%%%%%%%%%%%%%%%%%%%%%%%%%%%%%%%%%%%%%%%%%%%%%%%%%%%%%%%%%%%%%%
	%%%%%%%%%%%%%%%%%%%%%%%%%%%%%%%%%%%%%%%%%%%%%%%%%%%%%%%%%%%%%%%%%%%%%%%%%%%%%%%%%

\section{The Gaussian subordinated Lévy field}\label{sec:GaussianSubordLP} In this section we define Gaussian subordinated Lévy fields. Since their construction is motivated by the subordination of standard Lévy processes we shortly repeat this procedure: If $l$ denotes a Lévy process and $S$ denotes a Lévy subordinator (i.e. a non-decreasing Lévy process), which is independent of $l$, the time-changed process
\begin{align*}
t\mapsto l(S(t)),~t\geq 0,
\end{align*}
is called subordinated Lévy process. It can be shown that this process is again a Lévy process (cf. \cite[Theorem 1.3.25]{LevyProcessesAndStochasticCalculus}).

In order to construct the GSLF we consider a domain $\mathcal{D}\subset \mathbb{R}^d$ with $1\leq d\in\mathbb{N}$. Let $l=(l(t),~t\geq 0)$ be a Lévy process, $W:\Omega\times\mathcal{D}\rightarrow\mathbb{R}$ be an $\mathcal{F}\otimes\mathcal{B}(\mathcal{D})$-measurable GRF  which is independent of $l$ and $F:\mathbb{R}\rightarrow \mathbb{R_+}$ be a measurable, non-negative function. The \textit{Gaussian subordinated Lévy field} is defined by 
\begin{align*}
L(\underline{x}) = l(F(W(\underline{x}))), \text{ for } \underline{x}\in \mathcal{D}.
\end{align*}
Note that assuming the GRF $W$ to have continuous paths is sufficient to ensure joint measurability (see \cite[Lemma 4.51]{InfiniteDimensionalAnalysis}). Since the Lévy process $l$ is in general discontinuous, the GSLF $L$ has in general discontinuous paths. This is demonstrated in Figure \ref{FIG:SamplesGSLP}, which shows samples of the GSLF. 
\begin{figure}[ht]
	\centering	\subfigure{\includegraphics[scale=0.19]{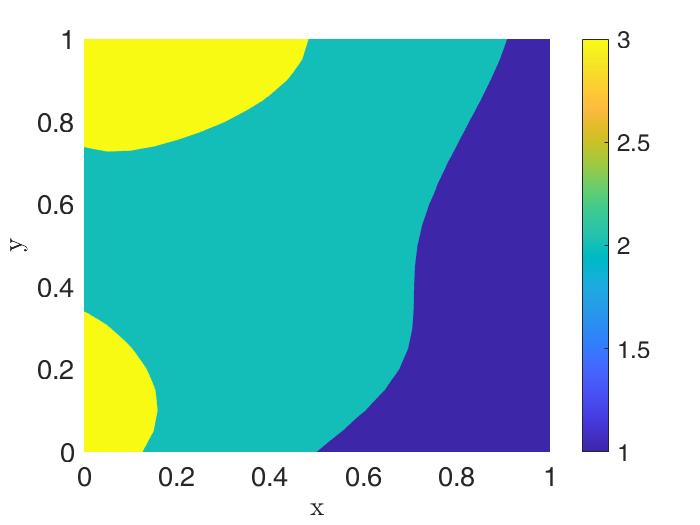}}	\subfigure{\includegraphics[scale=0.19]{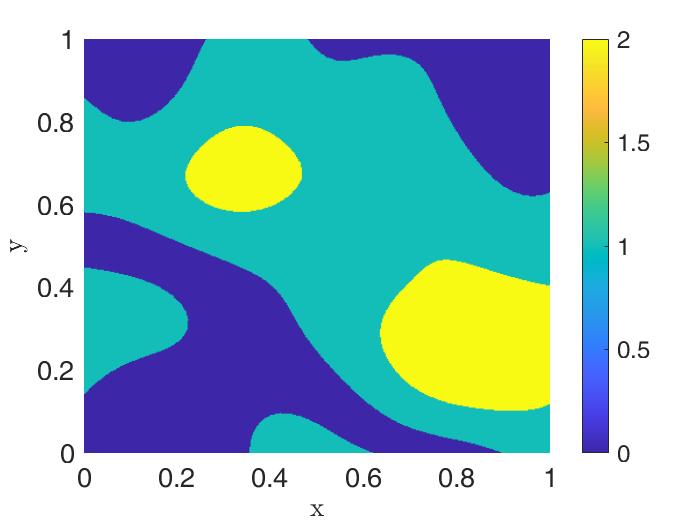}}
	\caption{Sample of a Matérn-1.5-GRF-subordinated Poisson process.}\label{FIG:SamplesGSLP}
\end{figure}
\begin{rem}
Discontinuous random fields often serve as prior model to solve inverse problems with a Bayesian approach as an alternative to the standard Gaussian prior (see e.g. \cite{MR3549006, HosseiniWellPosedBayesianInverseProblems}). In these situations, the prior model is often set to be a Gaussian(-related) level-set function. One way to construct such a prior model is as follow:
\begin{align*}
u(\underline{x}) = \sum_{i=1}^n u_i\mathds{1}_{D_i}(\underline{x}),~\underline{x}\in\mathcal{D},
\end{align*} 
where $n\in\mathbb{N}$, $(u_i,~i=1,\dots,n)\subset{\mathbb{R}}$ are fixed and 
\begin{align*}
D_i=\{\underline{x}\in\mathcal{D}~|~c_{i-1}\leq W(\underline{x})< c_i\},
\end{align*}
with fixed levels $(c_i,~i=1,\dots,n)\subset\mathbb{R}$ and a GRF $W$ (see e.g. \cite{MR3687326, MR4119332}). The GSLF, as defined above, may be interpreted as a generalization of the Gaussian level-set function.
\end{rem}

\begin{rem}\label{rem:GSLPmeasurable}
It is easy to see that the GSLF is measurable: A Lévy process $l:\Omega\times \mathbb{R}_+\rightarrow \mathbb{R}$ has càdlàg paths and, hence, is $\mathcal{F}\otimes \mathcal{B}(\mathbb{R}_+)-\mathcal{B}(\mathbb{R})$-measurable (see \cite[Chapter 1, Theorem 30]{StochasticIntegrationAndDifferentialEquations} and \cite[Chapter 6]{LevyProcessesAndInfinitelyDivisibleDistributions}). Further, since $F$ and $W$ are measurable by assumption, the mapping 
\begin{align*}
(\omega,\underline{x})\mapsto F(W(\omega,\underline{x})), ~(\omega,\underline{x})\in\Omega\times\mathcal{D},
\end{align*}
is $\mathcal{F}\otimes\mathcal{B}(\mathcal{D})-\mathcal{B}(\mathbb{R}_+)$-measurable. It follows now by \cite[Lemma 4.49]{InfiniteDimensionalAnalysis}, that the mapping
\begin{align*}
(\omega,\underline{x})\mapsto (\omega,F(W(\omega,\underline{x}))),~(\omega,\underline{x})\in\Omega\times\mathcal{D},
\end{align*}
 is $\mathcal{F}\otimes\mathcal{B}(\mathcal{D})-\mathcal{F}\otimes \mathcal{B}(\mathbb{R}_+)$-measurable. Therefore, the GSLF
 \begin{align*}
  (\omega,\underline{x})\mapsto  l(\omega,F(W(\omega,\underline{x}))),~(\omega,\underline{x})\in\Omega\times\mathcal{D},
 \end{align*}
is $\mathcal{F}\otimes \mathcal{B}(\mathcal{D})-\mathcal{B}(\mathbb{R})$-measurable (cf. \cite[Lemma 4.22]{InfiniteDimensionalAnalysis}).
\end{rem}

    %%%%%%%%%%%%%%%%%%%%%%%%%%%%%%%%%%%%%%%%%%%%%%%%%%%%%%%%%%%%%%%%%%%%%%%%%%%%%%%%%
	%%%%%%%%%%%%%%%%%%%%%%%%%%%%%%%%%%%%%%%%%%%%%%%%%%%%%%%%%%%%%%%%%%%%%%%%%%%%%%%%%
	%% The pointwise characteristic function of the field
	%%%%%%%%%%%%%%%%%%%%%%%%%%%%%%%%%%%%%%%%%%%%%%%%%%%%%%%%%%%%%%%%%%%%%%%%%%%%%%%%%
	%%%%%%%%%%%%%%%%%%%%%%%%%%%%%%%%%%%%%%%%%%%%%%%%%%%%%%%%%%%%%%%%%%%%%%%%%%%%%%%%%

\section{The pointwise characteristic function of a GSLF}\label{sec:Chracteristicfunction}
In the following section we derive a formula for the pointwise characteristic function of the GSLF, which determines the pointwise distribution entirely. Such a formula is especially valuable in applications, where distributions of a random field have to be fitted to data observed from real world phenomena. We start with a technical Lemma on the computation of expectations of functionals of the GSLF.
\begin{lemma}\label{LE:IteratedExpectationGSLP}
Let $l=(l(t),~t\geq 0)$ be a stochastic process with a.s. càdlàg paths and $W_+$ be a real-valued, non-negative random variable which is stochastically independent of $l$. Further, let $g:\mathbb{R}\rightarrow\mathbb{R}$ be a continuous function. It holds
\begin{align*}
\mathbb{E}(g(l(W_+))) = \mathbb{E}(m(W_+)),
\end{align*}
with $m(z) := \mathbb{E}(g(l(z)))$ for $z\in\mathbb{R}_+$.
\end{lemma}
\begin{proof}
The proof follows by the same arguments as in the proof of \cite[Lemma 4.1]{SubordGRFTheory}. 
\end{proof}

\begin{rem}\label{rem:ExtensionIteratedExpectation}
Note that Lemma \ref{LE:IteratedExpectationGSLP} also holds for complex-valued, continuous and bounded functions $g:\mathbb{R}\rightarrow \mathbb{C}$ (cf. \cite[Remark 4.2]{SubordGRFTheory}). Further, we emphasize that Lemma \ref{LE:IteratedExpectationGSLP} also holds for an $\mathbb{R}_+^d$-valued random variable $W_+$ and a random field $(l(\underline{t}),~\underline{t}\in\mathbb{R}_+^d)$, independent of $W_+$, which is a.s. càdlàg in each variable, i.e. for $\mathbb{P}$-almost all $\omega \in\Omega$, it holds
\begin{align*}
\underset{n\rightarrow \infty}{lim} l(t_1^{(n)},\dots,t_d^{(n)}) = l(t_1,\dots,t_d),
\end{align*}
for $t_j^{(n)} \searrow t_j$, for $n\rightarrow \infty$, $j=1,\dots,d$, and any $\underline{t}=(t_1,\dots,t_d)\in \mathbb{R}_+^d$.
\end{rem}
With Lemma \ref{LE:IteratedExpectationGSLP} at hand, we are able to derive a formula for the pointwise characteristic funciton of the GSLF.
\begin{corollary}\label{COR:CharFctGSLP}
Let $l=(l(t),~t\geq 0)$ be a Lévy process with Lévy triplet $(\gamma,b,\nu)$ and $W=(W(\underline{x}),~\underline{x} \in \mathcal{D})$ be an independent GRF with pointwise mean $\mu_W(\underline{x})=\mathbb{E}(W(\underline{x}))$ and variance $\sigma_W(\underline{x})^2:=Var(W(\underline{x}))$ for $\underline{x} \in \mathcal{D}$. Further, let $F:\mathbb{R}\rightarrow \mathbb{R}_+$ be measurable. It holds
\begin{align*}
\phi_{l(F(W(\underline{x})))}(\xi)=\mathbb{E}\Big(\exp(i\xi \,l(F(W(\underline{x}))))\Big)  
&=\frac{1}{\sqrt{2\pi \sigma_W(\underline{x})^2}}\int_\mathbb{R}exp\Big( F(y)\psi(\xi)- \frac{(y-\mu_W(\underline{x}))^2}{2\sigma_W(\underline{x})^2}\Big) dy,
\end{align*}
for $\underline{x}\in\mathcal{D}$, where $\psi$ denotes the characteristic exponent of $l$ defined by 
\begin{align*}
\psi(\xi) = i\gamma \xi - \frac{b}{2}\xi^2 + \int_{\mathbb{R}\setminus\{0\}}e^{i\xi y}-1-i\xi y\mathds{1}_{\{|y|\leq 1\}}(y)\nu(dy),~\xi\in\mathbb{R}.
\end{align*}
\end{corollary}
\begin{proof}
We consider a fixed $\underline{x}\in\mathcal{D}$ and use Lemma \ref{LE:IteratedExpectationGSLP} with $g(\cdot) := \exp(i\xi\cdot)$ and $W_+:=F(W(\underline{x}))$ to calculate
\begin{align*}
\mathbb{E}\Big(\exp(i\xi  \,l(F(W(\underline{x}))))\Big) = \mathbb{E}(g(l(W_+)) = \mathbb{E}(m(W_+)) = \mathbb{E}(m(F(W(\underline{x})))),
\end{align*}
where $m$ is defined through
\begin{align*}
m(z) = \mathbb{E}(\exp(i\xi l(z)) = \exp(z\,\psi(\xi)),~z\in\mathbb{R}_+,
\end{align*}
by the Lévy-Khinchin formula (see Theorem \ref{TH:LevyKhinchinFormula1d}).
Hence, we obtain
\begin{align*}
\mathbb{E}\Big(\exp(i\xi  \,l(F(W(\underline{x}))))\Big)&= \mathbb{E}\Big(\exp(F(W(\underline{x}))\psi(\xi))\Big)\\
&=\frac{1}{\sqrt{2\pi \sigma_W(\underline{x})^2}}\int_\mathbb{R}exp\Big( F(y)\psi(\xi)- \frac{(y-\mu_W(\underline{x}))^2}{2\sigma_W(\underline{x})^2}\Big) dy.
\end{align*}
\end{proof}

    %%%%%%%%%%%%%%%%%%%%%%%%%%%%%%%%%%%%%%%%%%%%%%%%%%%%%%%%%%%%%%%%%%%%%%%%%%%%%%%%%
	%%%%%%%%%%%%%%%%%%%%%%%%%%%%%%%%%%%%%%%%%%%%%%%%%%%%%%%%%%%%%%%%%%%%%%%%%%%%%%%%%
	%% Approximation
	%%%%%%%%%%%%%%%%%%%%%%%%%%%%%%%%%%%%%%%%%%%%%%%%%%%%%%%%%%%%%%%%%%%%%%%%%%%%%%%%%
	%%%%%%%%%%%%%%%%%%%%%%%%%%%%%%%%%%%%%%%%%%%%%%%%%%%%%%%%%%%%%%%%%%%%%%%%%%%%%%%%%

\section{Approximation of the fields}\label{sec:approximation}

The GSLF may in general not be simulated exactly since in most situations it is not possible to draw exact samples of the corresponding GRF and the Lévy process. The question arises how the GSLF may be approximated and if the corresponding approximation error may be quantified. In this section we answer both questions. We prove an approximation result for the GSLF where we approximate the GRF and the Lévy process separately.
To be more precise, we approximate the Lévy processes using a piecewise constant càdlàg approximation on a discrete grid (see e.g. \cite{ApproximationAndSimulation} and the remainder of the current section). The GRF may be approximated by a truncated Karhuen-Loève-expansion or using values of the GRF simulated on a discrete grid, e.g. via Circulant Embedding (see, e.g. \cite{AStudyOfElliptic, MLMCMethodsForStochEllMultiscalePDEs, FurtherAnalysisOfMultilevelMonteCarloMethodsForEllipticPDEsWithRandomCoefficients} resp. \cite{AnalysisOfCirculantEmbeddingMethodsForSamplingStationaryRandomFields, CirculantEmbeddingWithWMCAnalysisForEllipicPDEWithLognormalCoefficients}). Naturally, we have to start with some assumptions on the regularity of the GRF and the approximability of the Lévy process. Assumptions of this type are natural and well known in different situations (see e.g. \cite{AStudyOfElliptic, SGRFPDE, StrongAndWeakErrorEstimatesForTheSolutionsOfEllipticPDEsWithRandomCoefficient}). 
For simplicity, we consider centered GRFs in this subsection.

\begin{assumption}\label{ASS:CutProblemEigenvalues}
Let $W$ be a zero-mean GRF on the compact domain $\mathcal{D}$. We denote by $q_W:\mathcal{D}\times \mathcal{D}\rightarrow\mathbb{R}$ the corresponding covariance function and by $((\lambda_i,e_i),~i\in \mathbb{N})$ the eigenpairs associated to the corresponding covariance operator $Q$ as introduced in Section \ref{Subsec:GRF}. In particular, $(e_i,~i\in \mathbb{N})$ is an orthonormal basis of $L^2(\mathcal{D})$.
\begin{enumerate}
\item We assume that the eigenfunctions are continuously differentiable and there exist positive constants $\alpha, ~\beta, ~C_e, ~C_\lambda>0$ such that for any $i\in\mathbb{N}$ it holds
\begin{align*}
\|e_i\|_{L^\infty(\mathcal{D})}\leq C_e,\|\nabla e_i\|_{L^\infty(\mathcal{D})}\leq C_e i^\alpha,~ \sum_{i=1}^ \infty \lambda_ii^\beta\leq C_\lambda	< + \infty.
\end{align*}
\item $F:\mathbb{R}\rightarrow \mathbb{R}_+$ is Lipschitz continuous and globally bounded by $C_F>0$, i.e. $F(x)<C_F,~ x\in\mathbb{R}$.
\item $l$ is a Lévy process on $[0,C_F]$ with Lévy triplet $(\gamma,b,\nu)$ which is independent of $W$. Further, we assume there exists a constant $\eta>1$ and càdlàg approximations $l^\apprlevy$ of this process such that for every $s\in[1,\eta)$ it holds 
\begin{align}\label{EQ:AssApprLevy}
\mathbb{E}(|l(t)-l^\apprlevy(t)|^s)\leq C_l\varepsilon_l,~ t\in[0,C_F),
\end{align}
for $\varepsilon_l >0$ and 
\begin{align}\label{EQ:MomentsLevy}
\mathbb{E}(|l(t)|^s)\leq C_{l}t^\delta,~ t\in[0,C_F),
\end{align}
with $\delta\in(0,1]$ and a constant $C_l>0$ which may depend on $s$ but is independent of $t$ and $\varepsilon_l$.  
\end{enumerate}
\end{assumption}
We continue with a remark on Assumption \ref{ASS:CutProblemEigenvalues}.
\begin{rem}\label{rem:LPShortTimeBehaviour}
Assumtion \ref{ASS:CutProblemEigenvalues} \textit{i} is natural for GRFs and guarantees certain regularity properties for the paths of the GRF (see e.g. \cite{AStudyOfElliptic, SGRFPDE, StrongAndWeakErrorEstimatesForTheSolutionsOfEllipticPDEsWithRandomCoefficient}). Equation~\eqref{EQ:AssApprLevy} ensures that we can approximate the Lévy subordinators in an $L^s$-sense. This can be achieved under appropriate assumptions on the tails of the distribution of the Lévy processes, see \cite[Assumption 3.6, Assumption 3.7, Theorem 3.21]{ApproximationAndSimulation} and \cite[Section 7]{SGRFPDE}).
There are several results in the direction of condition \eqref{EQ:MomentsLevy}. For example, in \cite{MR2430710} and \cite{MR3373306}, the authors formulate general assumptions on the Lévy measure which guarantee Equation~\eqref{EQ:MomentsLevy} and similar properties. Further, in \cite{MR2430710} the authors explicitly derive the rate $\delta$ in Equation~\eqref{EQ:MomentsLevy} for several Lévy processes. In \cite[Proposition 2.3]{MR3004556}, an exact polynomial time-dependence of the absolute moments of a Lévy process under the assumption that the absolute moment of the Lévy process exists up to a certain order was proven.
% note: siehe auch Applebaum, Th. 2.5.2 für die existenz von Momenten von LPs
%evtl geht fuer ass 5.1 iii 2 auch unser eigener Beweis (siehe notes), hier ist aber unklar ob g complex diffbar ist, siehe aber auch klenke satz 15.31. Theoretisch auch einfacher möglich über $phi^a = \exp(-a\psi(x))$, jedoch dann unklar ob $\psi$ diffbar ist.
In order to illustrate \eqref{EQ:MomentsLevy}, we present a short numerical example: for three different Lévy processes, we estimate $\mathbb{E}(|l(t)|^s)$ for $t=2^i$ with $i\in\{1,0,-1,\dots,-16\}$ using $10^7$ samples of the process $l$ and different values for the exponent $s\geq 1$. The results are shown in Figure \ref{Fig:MomLevy}, where the estimated moments  $\mathbb{E}(|l(t)|^s)$ are plotted against the time parameter $t$. The results clearly indicate that $\mathbb{E}(|l(t)|^s) = \mathcal{O}(t)$, $t\rightarrow 0$ which implies \eqref{EQ:MomentsLevy} with $\delta=1$ in the considered examples.
\begin{figure}[ht]
	\centering
	\subfigure{\includegraphics[scale=0.3]{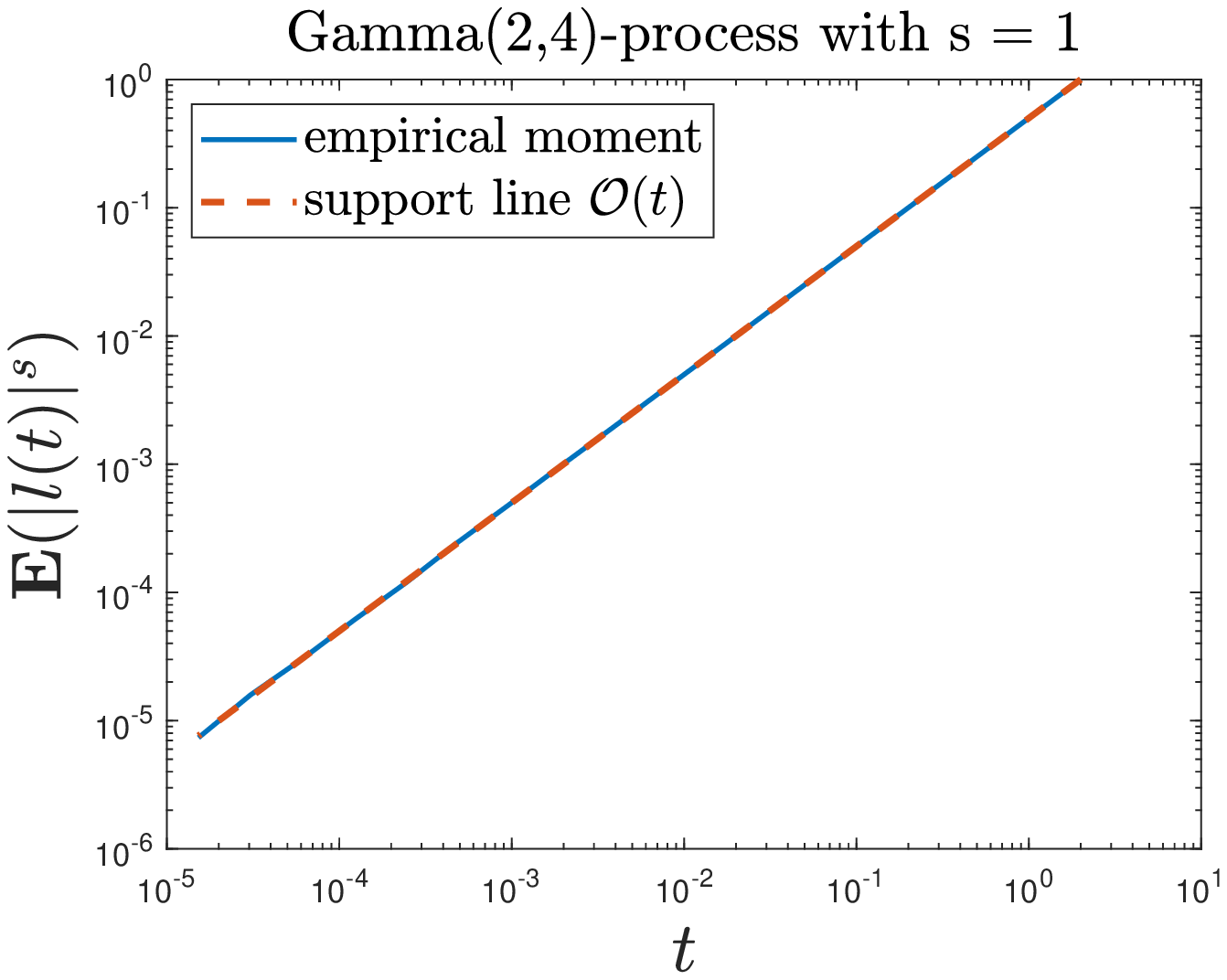}}
	\subfigure{\includegraphics[scale=0.3]{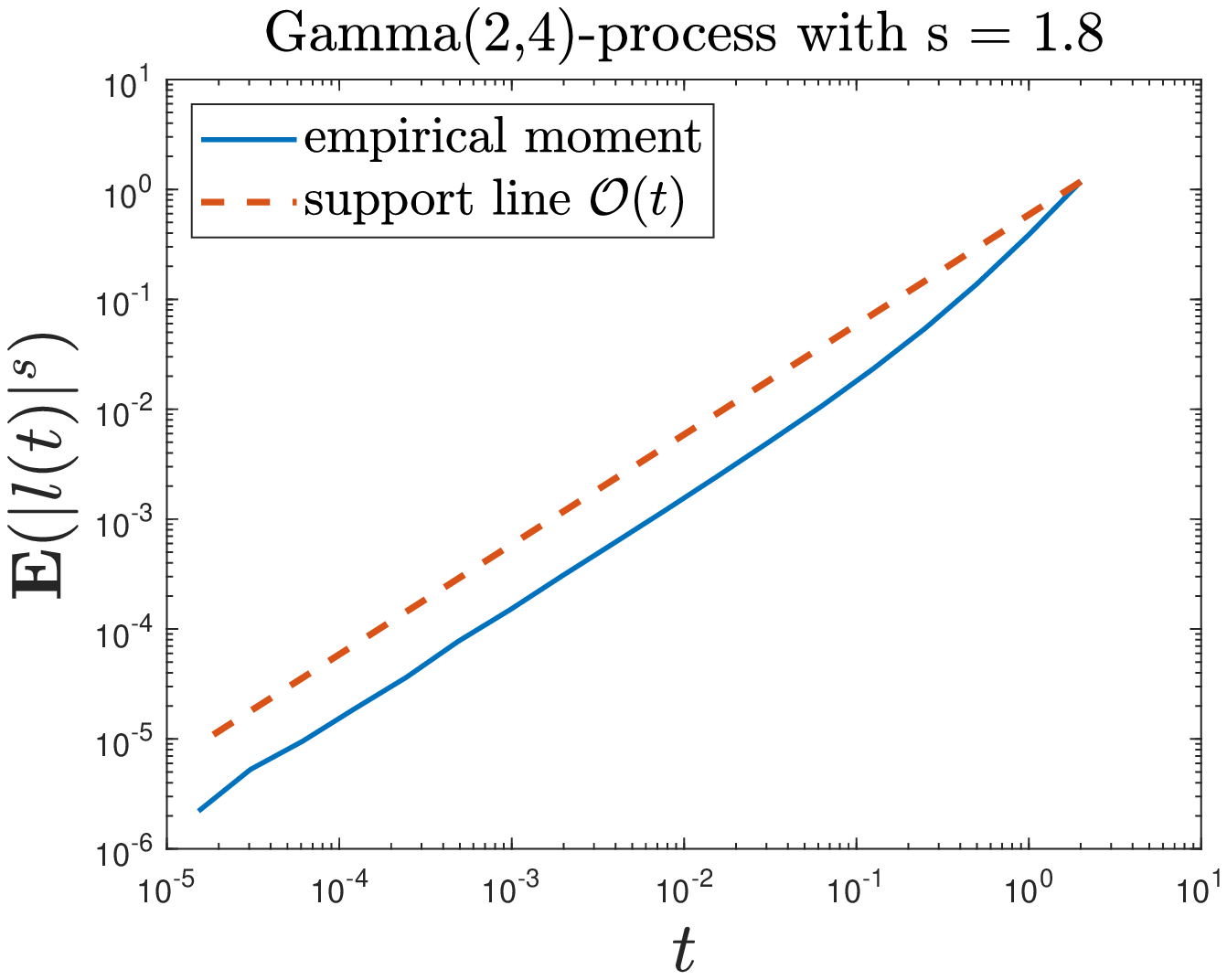}}
	\subfigure{\includegraphics[scale=0.3]{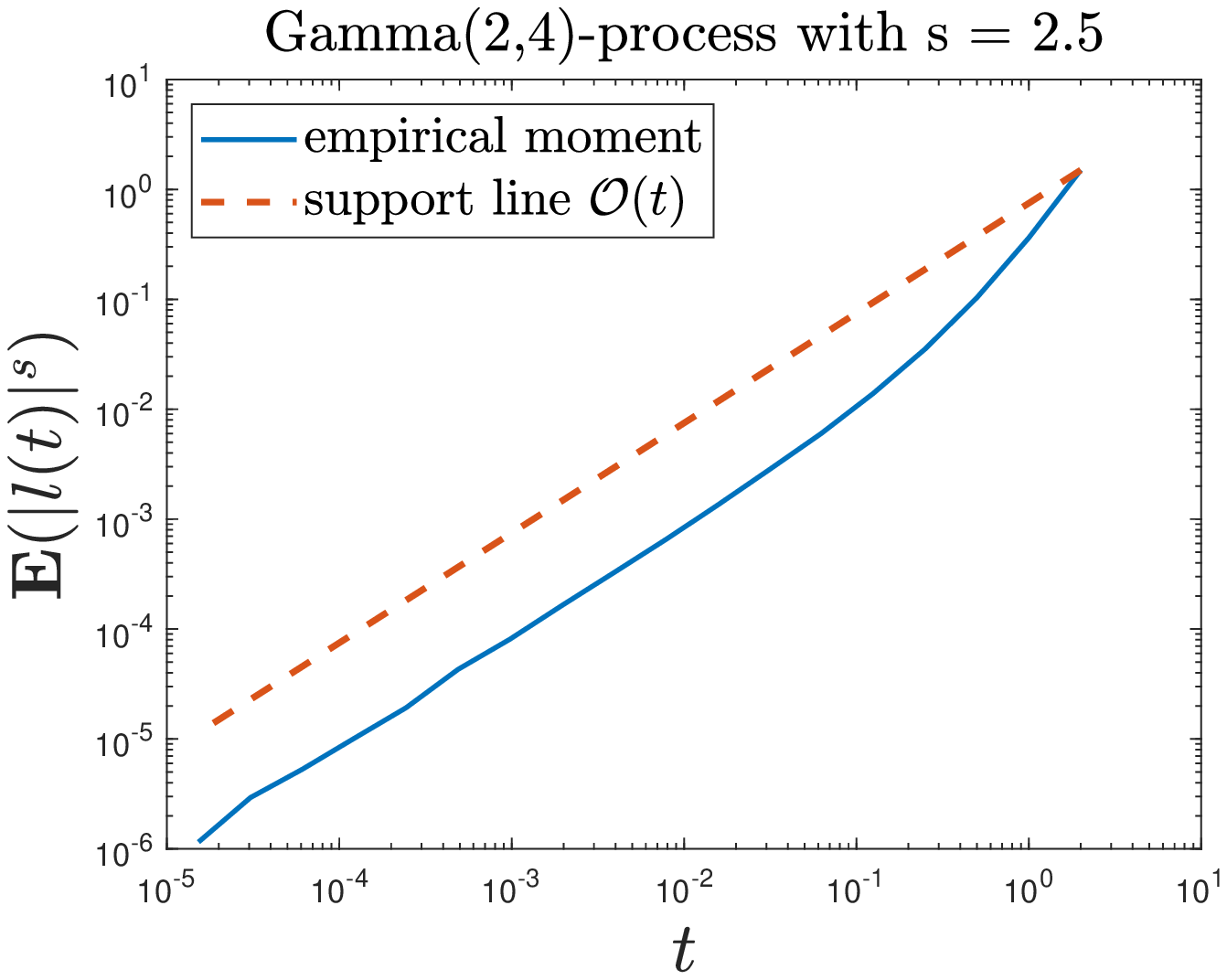}}\\
	\subfigure{\includegraphics[scale=0.3]{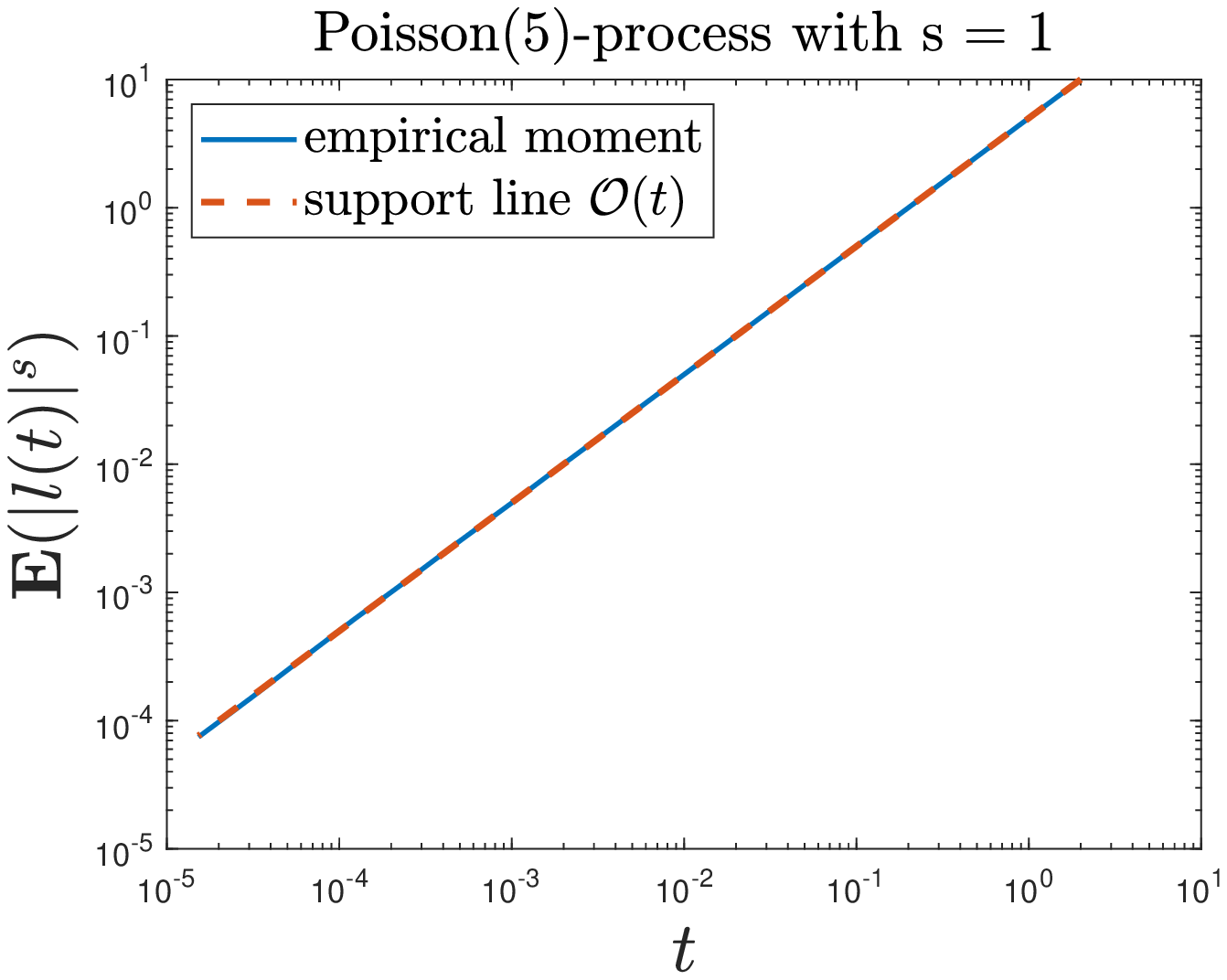}}
	\subfigure{\includegraphics[scale=0.3]{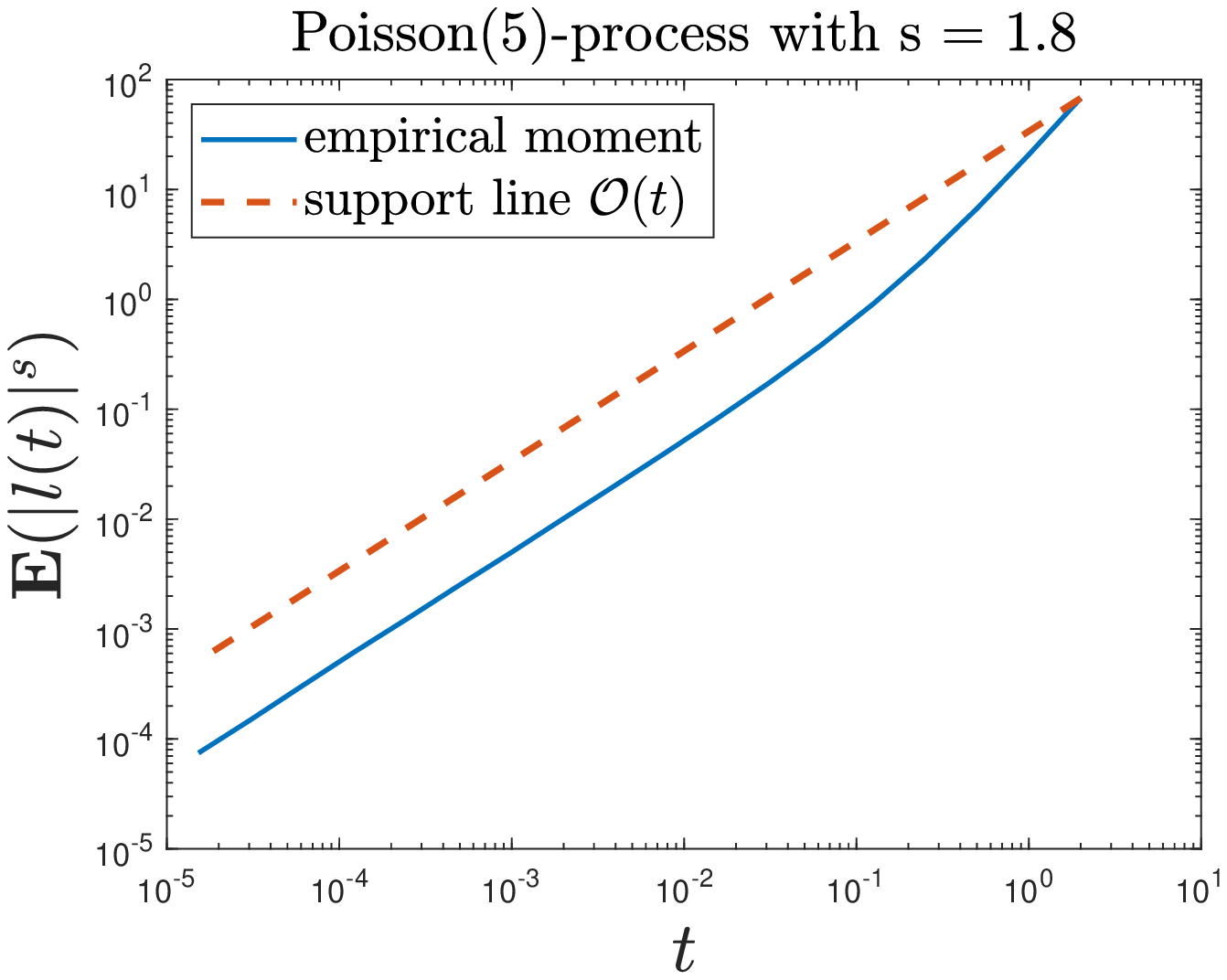}}
	\subfigure{\includegraphics[scale=0.3]{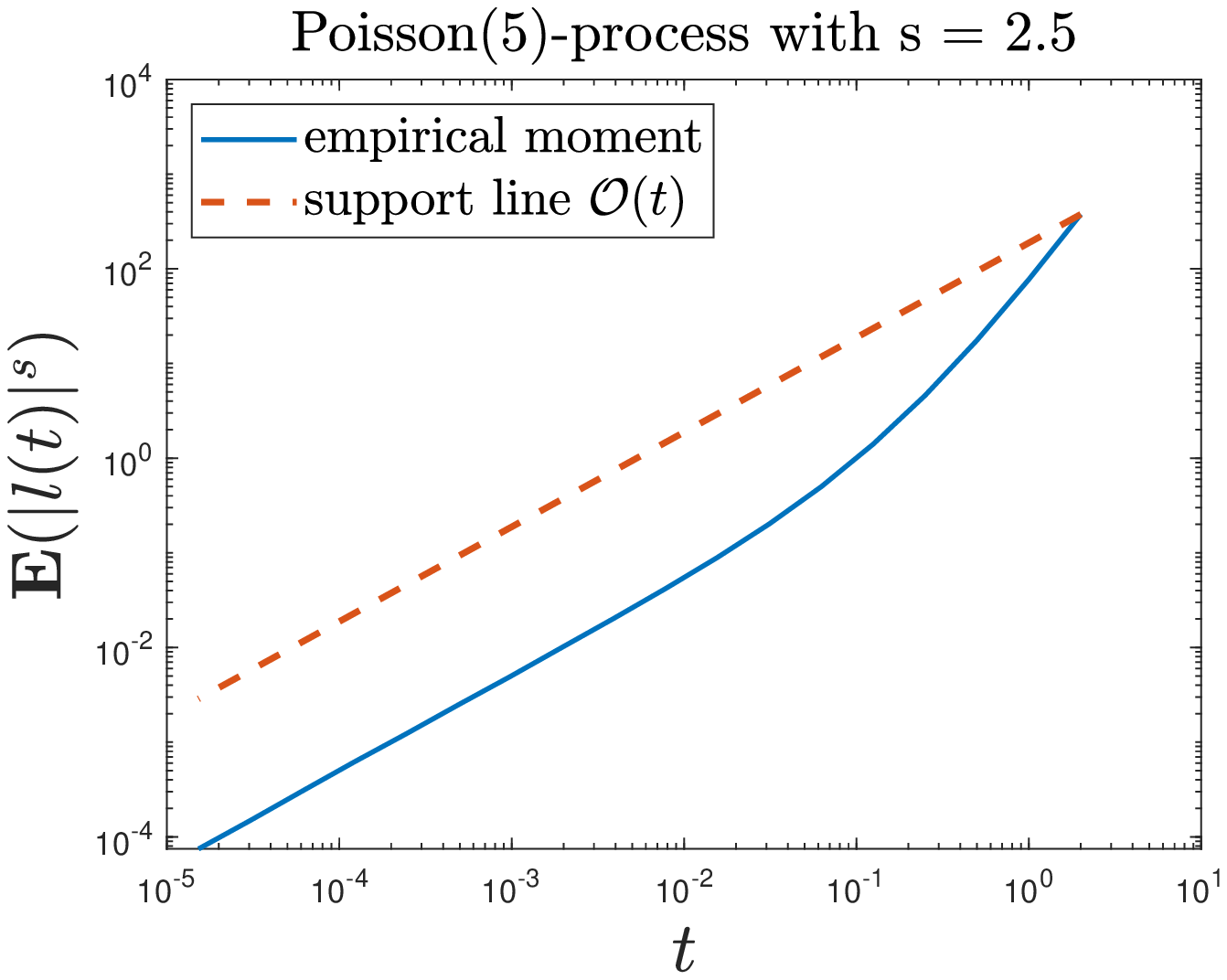}}\\  
	\subfigure{\includegraphics[scale=0.3]{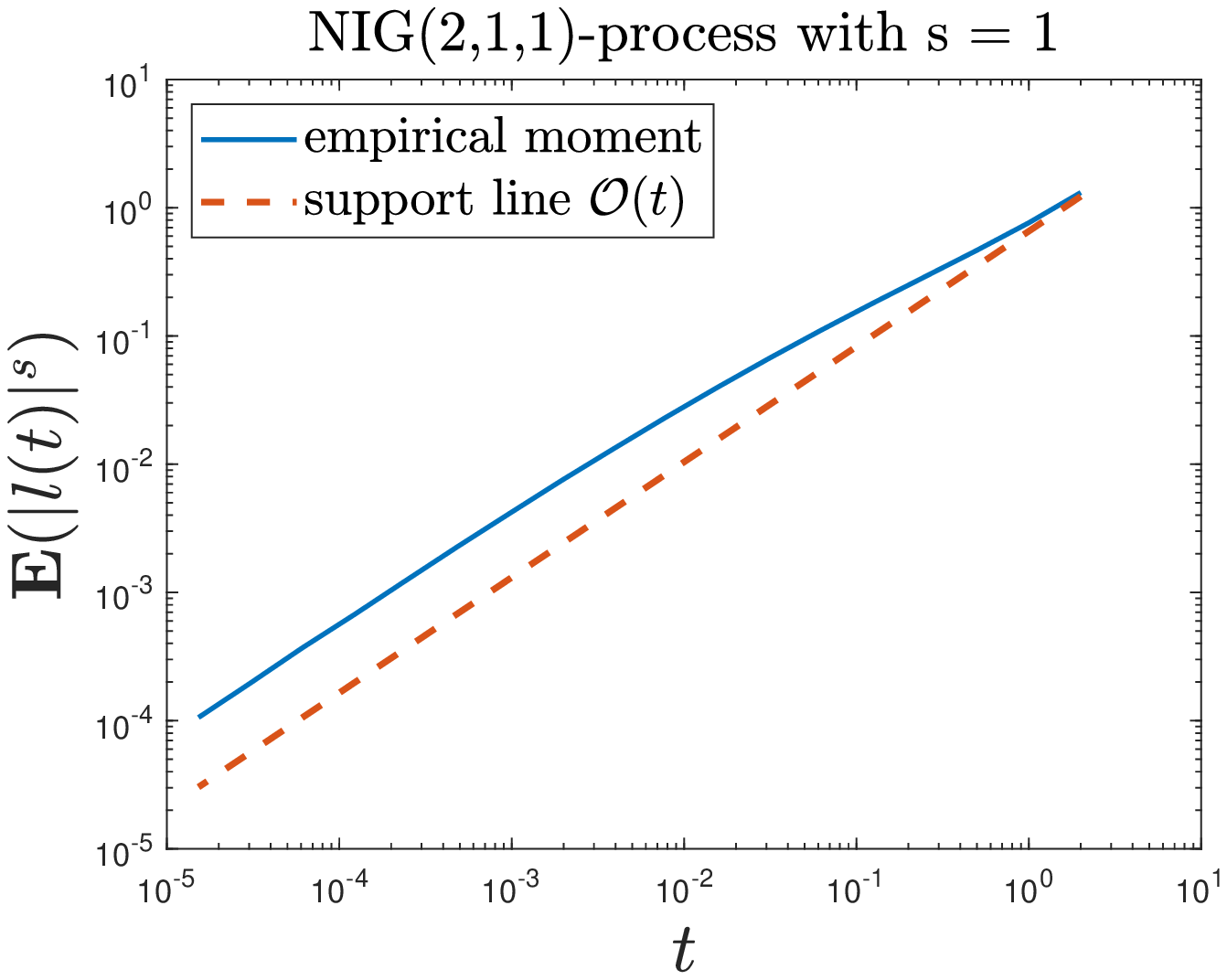}}
    \subfigure{\includegraphics[scale=0.3]{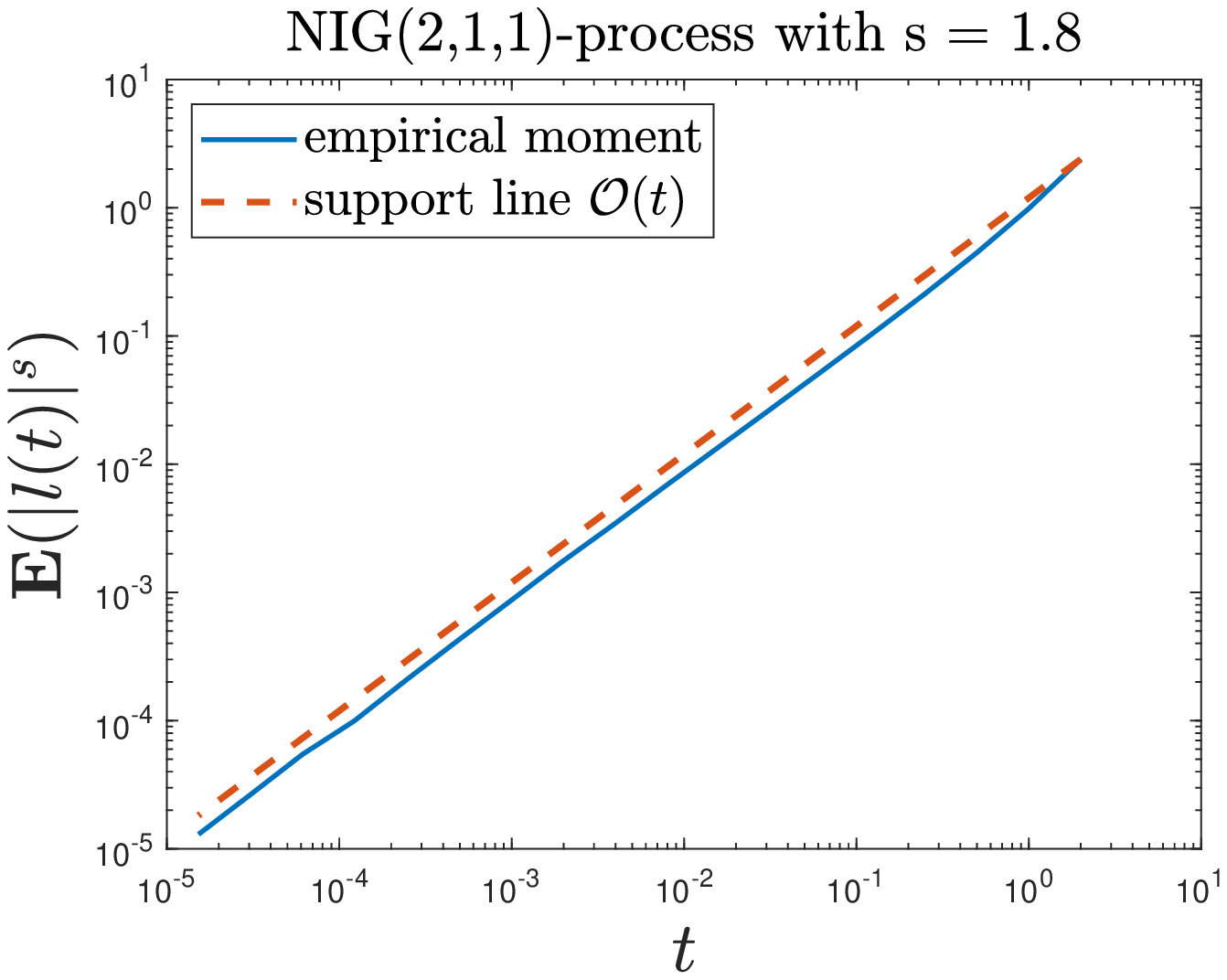}}
	\subfigure{\includegraphics[scale=0.3]{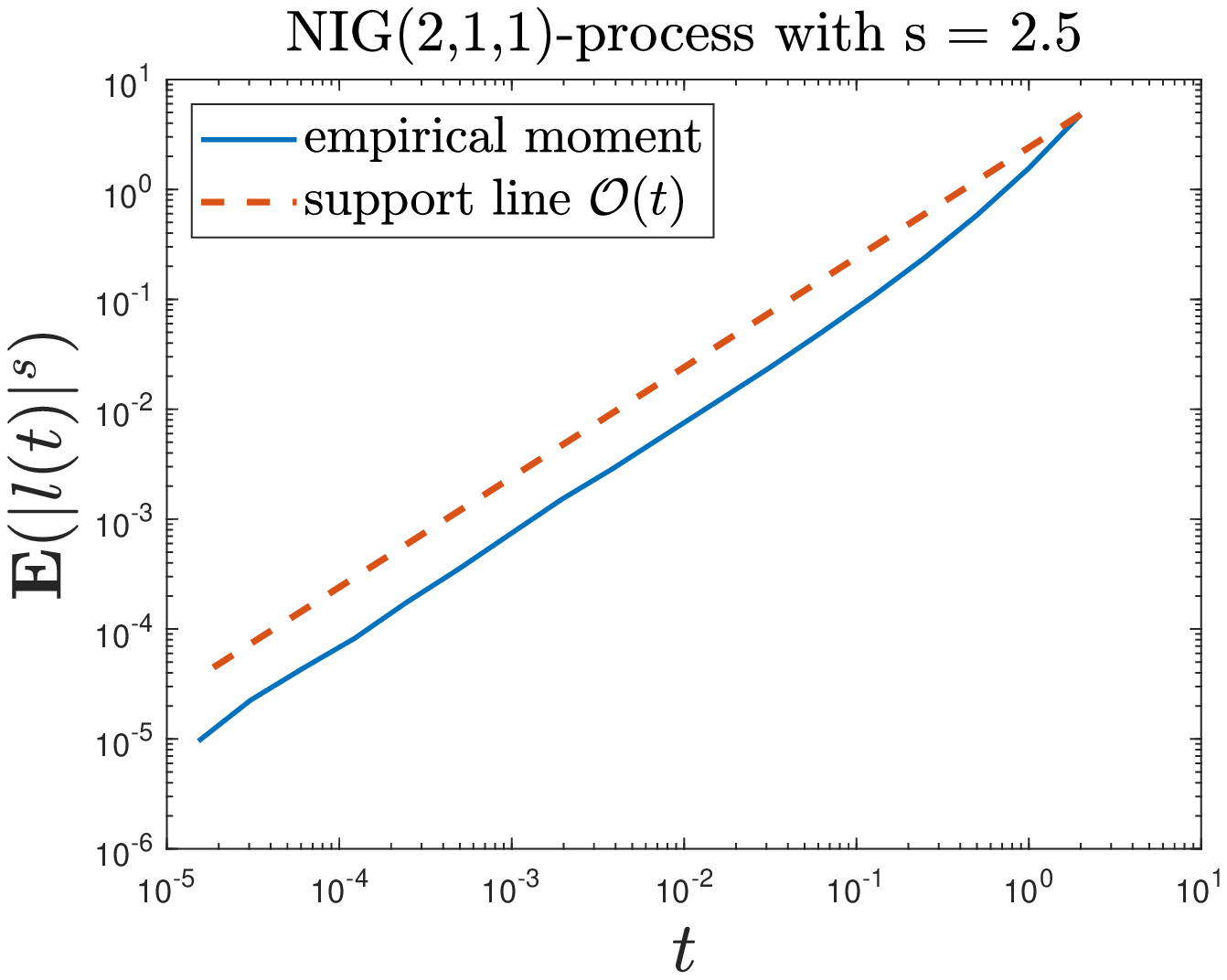}}
	\caption{Estimated $s$-moment for different Lévy processes and different values of $s$. Top: Gamma(2,4)-process, middle: Poisson(5)-process, bottom: NIG(2,1,1)-process.}\label{Fig:MomLevy}
\end{figure}
\end{rem}

We close this subsection with a remark on a possible way to construct approximations of Lévy processes.
\begin{rem}\label{rem:LPApprox}
One way to construct a càdlàg approximation $l^{(\varepsilon_l)}\approx l$ is a piecewise constant extension of the values of the process on a discrete grid: assume  $\{t_i,~i=0\dots,N_l\}$, with $t_0=0$ and $t_{N_l}=C_F$ is a grid on $[0,C_F]$ with $|t_{i+1}-t_i|=t_{i+1}-t_i=\varepsilon_l>0$ for $i=0,\dots,{N_l}-1$ and $\{l(t_i),~i=0,\dots,{N_l}\}$ are the values of the process $l$ on this grid. We define
\begin{align*}
l^{(\varepsilon_l)} (t) = \sum_{i=1}^{N_l} \mathds{1}_{[t_{i-1},t_i)} (t) \,l(t_{i-1}),~t\in[0,C_F).
\end{align*}
Such a construction yields an admissible approximation in the sense of Assumption \ref{ASS:CutProblemEigenvalues} \textit{iii} for many Lévy processes (see e.g. \cite[Section 7]{SGRFPDE} for the case of Poisson and Gamma processes). Note that the values of the process on the grid may be simulated easily for many Lévy processes using the independent and stationary increment property.
\end{rem}

\subsection{Approximation of the GRF}\label{subsec:ApprGRF}
In this subsection, we shortly introduce two possible approaches to approximate the GRF W. Both approaches are admissible for the construction of approximations of the GSLF in Subsection \ref{subsec:ApprGSLP}. In the following and for the rest of the section, we consider $\mathcal{D}=[0,D]^d$. Assumption \ref{ASS:CutProblemEigenvalues} allows conclusions to be drawn on the spatial regularity of the GRF $W$. The next lemma on the approximation of the GRF $W$ follows from the regularity properties of $W$ (see \cite[Lemma 4.4]{SGRFPDE} and \cite{StrongAndWeakErrorEstimatesForTheSolutionsOfEllipticPDEsWithRandomCoefficient}).
\begin{lemma}\label{LE:StrongErrorBoundGRFApprox}
Consider the discrete grid $G^{(\varepsilon_W)}=\{(x_{i^{(1)}},\dots,x_{i^{(d)}})|~{i^{(1)}},\dots, {i^{(d)}}=0,\dots,M_{\varepsilon_W}\}$ on $\mathcal{D}$ where $(x_i,~i=0,\dots,M_{\varepsilon_W})$ is an equidistant grid on $[0,D]$ with maximum step size $\varepsilon_W$. Further, let $W^\apprgrf$ be an approximation of the GRF $W$ on the discrete grids $G^\apprgrf$ which is constructed by point evaluation of the random field $W$ on the grid and multilinear interpolation between the grid points. Under Assumption \ref{ASS:CutProblemEigenvalues} \textit{i}, it holds for $n\in[1,+\infty)$:
\begin{align*}
\|W-W^\apprgrf\|_{L^n(\Omega;L^\infty(\mathcal{D}))}&\leq C(D,n,d)\varepsilon_W^\gamma
\end{align*}
for $\gamma< \min (1,\beta/(2\alpha))$, where $\beta$ and $\alpha$ are the parameters from Assumption \ref{ASS:CutProblemEigenvalues} \textit{i}.
\end{lemma}
It follows by the Karhunen-Loève-Theorem (see e.g. \cite[Section 3.2]{RandomFieldsAndGeometry}) that the GRF $W$ admits the representation
\begin{align*}
W(\underline{x}) = \sum_{i=1}^\infty \sqrt{\lambda_i} Z_i e_i(\underline{x}),~\underline{x}\in \mathcal{D},
\end{align*}
with i.i.d. $\mathcal{N}(0,1)$-distributed random variables $(Z_i,~i\in\mathbb{N})$ and the sum converges in the mean-square sense, uniformly in $\underline{x}\in\mathcal{D}$. The Karhunen-Loève expansion (KLE) motivates another approach to approximate the GRF $W$: For a fixed cut-off index $N\in\mathbb{N}$, we define the approximation $W^{N}$ by 
\begin{align}\label{EQ:GRFAPPRKL}
W^{N}(\underline{x}) = \sum_{i=1}^{N} \sqrt{\lambda_i} Z_i e_i(\underline{x}),~\underline{x}\in \mathcal{D}.
\end{align}
Under Assumption \ref{ASS:CutProblemEigenvalues} we can quantify the corresponding approximation error (cf. \cite[Theorem 3.8]{AStudyOfElliptic}).
\begin{lemma}\label{LE:StrongApprGRFKLE}
Let Assumption \ref{ASS:CutProblemEigenvalues} \textit{i} hold. Then, it holds for any $N\in\mathbb{N}$ and $n\in[1,+\infty)$:
\begin{align*}
\|W-W^N\|_{L^n(\Omega;L^\infty(\mathcal{D}))}&\leq C(D,n,d)N^{-\frac{\beta}{2}}.
\end{align*}
\end{lemma}
\begin{proof}
It follows by \cite[Theorem 3.8]{AStudyOfElliptic} that 
\begin{align*}
\|W-W^N\|_{L^n(\Omega;L^\infty(\mathcal{D}))}^2\leq C(D,n,d) \sum_{i=N+1}^\infty \lambda_i,
\end{align*}
for all $N\in\mathbb{N}$. We use Assumption \ref{ASS:CutProblemEigenvalues} \textit{i} to compute
\begin{align*}
\sum_{i=N+1}^\infty \lambda_i = \sum_{i=N+1}^\infty \lambda_i i^{\beta} i^{-\beta}\leq N^{-\beta}\sum_{i=N+1}^\infty \lambda_i i^{\beta}\leq C_\lambda N^{-\beta},
\end{align*}
which finishes the proof.
\end{proof}
In the following, we denote by $W^N$ an approximation of the GRF of $W$. The notation is clear if we approximate the GRF by the truncated KLE. If we approximate $W$ by sampling on a discrete grid (cf. Lemma \ref{LE:StrongErrorBoundGRFApprox}), then we use $\varepsilon_W = 1/N$ and abuse notation to write $W^N = W^{(1/N)}$. Regardless of the choice between these two approaches, we thus obtain an approximation $W^N\approx W$ with 
\begin{align}\label{EQ:ApproximationErrorGRFGeneral}
\|W-W^N\|_{L^n(\Omega;L^\infty(\mathcal{D}))}&\leq C(D,n,d)R(N),
\end{align}
for $N\in\mathbb{N}$ and $n\in[1,+\infty)$, where $R(N)=N^{-\gamma}$ if we approximate the GRF by discrete sampling and $R(N)=N^{-\frac{\beta}{2}} $ if we approximate it by the KLE approach.

\subsection{Approximation of the GSLF}\label{subsec:ApprGSLP}

We are now able to quantify the approximation error for the GSLF where both components, the Lévy process and the GRF, are approximated.

\begin{theorem}\label{TH:GSLPAPPR}
Let Assumption \ref{ASS:CutProblemEigenvalues} hold and assume an approximation $W^N\approx W$ of the GRF is given, as introduced in Subsection \ref{subsec:ApprGRF}. For a given real number $1\leq p<\eta$ and a globally Lipschitz continuous function $g:\mathbb{R}\rightarrow\mathbb{R}$, it holds
\begin{align*}
\|g(l^\apprlevy(F(W^N))) - g(l(F(W)))\|_{L^p(\Omega;L^p(\mathcal{D}))}\leq C( \varepsilon_l^{\frac{1}{p}} + R(N)^\frac{\delta}{p} ),
\end{align*}
for $N\in\mathbb{N}$ and $\varepsilon_l>0$, where $\delta$ is the positive constant from \eqref{EQ:MomentsLevy}.
\end{theorem}
\begin{proof}
We calculate using Fubini's theorem and the triangle inequality
\begin{align*}
&\|g(l^\apprlevy(F(W^N))) - g(l(F(W)))\|_{L^p(\Omega;L^p(\mathcal{D}))}\\
& \leq  \|g(l^\apprlevy(F(W^N))) - g(l(F(W^N)))\|_{L^p(\Omega\times \mathcal{D})} +  \|g(l(F(W^N))) - g(l(F(W)))\|_{L^p(\Omega\times \mathcal{D})}\\
&=:I_1+I_2
\end{align*}
 Further, we use the Lipschitz continuity of $g$ together with Lemma \ref{LE:IteratedExpectationGSLP} to obtain for any $\underline{x}\in\mathcal{D}$
\begin{align*}
\mathbb{E}(| g(l^\apprlevy(F(W^N(\underline{x})))) - g(l(F(W^N(\underline{x}))))|^p) &\leq C_g\mathbb{E}(| l^\apprlevy(F(W^N(\underline{x}))) - l(F(W^N(\underline{x})))|^p)\\
& = \mathbb{E}(m(F(W^N(\underline{x})))
\end{align*} 
 with 
 \begin{align*}
 m(z) = \mathbb{E}(|l(z)-l^\apprlevy(z)|^p)\leq C_l\varepsilon_l  
 \end{align*}
 for $z\in[0,C_F)$ by Assumption \ref{ASS:CutProblemEigenvalues} \textit{iii}. Therefore, we obtain
 \begin{align*}
\mathbb{E}(| g(l^\apprlevy(F(W^N(\underline{x})))) - g(l(F(W^N(\underline{x}))))|^p)\leq C_l\varepsilon_l,
\end{align*}
for all $\underline{x}\in\mathcal{D}$ and, hence,
 \begin{align*}
 I_1^p &= \int_\mathcal{D}\mathbb{E}(| g(l^\apprlevy(F(W^N(\underline{x})))) - g(l(F(W^N(\underline{x})))) |^p)d\underline{x}\\
 &\leq  D^d C_l\varepsilon_l.
 \end{align*}
 For the second summand we calculate using Lemma \ref{LE:IteratedExpectationGSLP} and Remark \ref{rem:ExtensionIteratedExpectation} with $\tilde{l}(t_1,t_2):=l(t_1)-l(t_2)$ and $\tilde{W}_+:=(F(W(\underline{x})),F(W^N(\underline{x})))$
\begin{align*}
\mathbb{E}(|l(F(W(\underline{x}))) - l(F(W^N(\underline{x})))|^p) = \mathbb{E}(\tilde m(F(W(\underline{x})),F(W^N(\underline{x})))),
\end{align*} 
with
\begin{align*}
\tilde m(t_1,t_2) = \mathbb{E}(|l(t_1)-l(t_2)|^p),~t_1,t_2\in [0,C_F).
\end{align*}
For $C_F> t_1\geq t_2\geq 0$, the stationarity of $l$ together with Assumption \ref{ASS:CutProblemEigenvalues} \textit{iii} yields
\begin{align*}
\tilde{m}(t_1,t_2)= \mathbb{E}(|l(t_1-t_2)|^p) \leq C_l |t_1-t_2|^\delta.
\end{align*}
Further, for $0\leq t_1\leq t_2< C_F$ we have 
\begin{align*}
\tilde{m}(t_1,t_2)= \mathbb{E}(|l(t_2)-l(t_1)|^p)=\mathbb{E}(|l(t_2-t_1)|^p) \leq C_l |t_2-t_1|^\delta.
\end{align*}
Overall, we obtain for any $\omega\in\Omega$ the pathwise estimate
\begin{align*}
\tilde{m}(F(W(\underline{x})),F(W^N(\underline{x})))&\leq C_l|F(W(\underline{x})) - F(W^N(\underline{x}))|^\delta,
\end{align*}
for any $\underline{x}\in\mathcal{D}$.
We apply Hölder's inequality and Equation \eqref{EQ:ApproximationErrorGRFGeneral} to obtain
\begin{align*}
\mathbb{E}(|l(F(W(\underline{x}))) - l(F(W^N(\underline{x})))|^p) &= \mathbb{E}(\tilde m(F(W(\underline{x})),F(W^N(\underline{x})))),\\ 
&\leq C_l\mathbb{E}(|F(W(\underline{x})) - F(W^N(\underline{x}))|^\delta)\\
&\leq  C_l\mathbb{E}(|F(W(\underline{x})) - F(W^N(\underline{x}))|)^\delta\\
&\leq  C_lC_F\mathbb{E}(|W(\underline{x}) - W^N(\underline{x})|)^\delta\\
&\leq C_lC_F C(D,d)^\delta R(N)^{\delta},
\end{align*}
for any $\underline{x}\in\mathcal{D}$. Finally, we use the Lipschitz contnuity of $g$ to obtain:
 \begin{align*}
 I_2^p &= \int_\mathcal{D}\mathbb{E}(| g(l(F(W^N(\underline{x})))) - g(l(F(W(\underline{x})))) |^p)d\underline{x}\\
 &\leq C(C_l,D,d,\delta,C_F,C_g)R(N)^{\delta}.
 \end{align*}
 Overall, we end up with
 \begin{align*}
 \|g(l^\apprlevy(F(W^N))) - g(l(F(W)))\|_{L^p(\Omega;L^p(\mathcal{D}))} \leq C(C_l,D,d,\delta,C_F,C_g,p)(\varepsilon_l^{\frac{1}{p}} + R(N)^{\frac{\delta}{p}}).
 \end{align*}
\end{proof}

\subsection{The pointwise distribution of the approximated GSLF}\label{subsec:ApprLevyPWDist}

In Section \ref{sec:Chracteristicfunction} we ivestigated the pointwise distribution of a GSLF and derived a formula for its pointwise characteristic function. In Section \ref{sec:approximation}, we demonstrated how approximations of the Lévy process $l$ and the underlying GRF $W$ may be used to approximate the GSLF and quantified the approximation error. This is of great importance especially in applications, since it is in general not possible to simulate the GRF or the Lévy process on their continuous parameter domains. The question arises how such an approximation affects the pointwise distribution of the field. For this purpose, we prove in the following a formula for the pointwise charateristic function of the approximated GSLF.

\begin{corollary}\label{COR:CharFctApprGSLP}
We consider a GSLF with Lévy process $l$ and an independent GRF $W$. Let $l^{(\varepsilon_l)}\approx l$ be a càdlàg approximation of the Lévy process and $W^N\approx W$ be an approximation of the GRF, where we assume that the mean function $\mu_W$ of the GRF $W$ is known and does not need to be approximated. For $\underline{x}\in\mathcal{D}$, the pointwise characteristic function of the approximated GSLF is given by
\begin{align*}
\phi_{l^{(\varepsilon_l)}(F(W^N(\underline{x})))}(\xi) =\mathbb{E}\Big(\exp(i\xi l^{(\varepsilon_l)}(F(W^N(\underline{x}))))\Big)=\mathbb{E}(\rho_{l^{(\varepsilon_l)}}(F(W^N(\underline{x})),\xi)),~\xi\in\mathbb{R},
\end{align*}
where
\begin{align*}
\rho_{l^{(\varepsilon_l)}}(t,\xi) = \mathbb{E}(\exp(i\xi l^{(\varepsilon_l)}(t))),
\end{align*}
denotes the pointwise characteristic function of $l^{(\varepsilon_l)}$.
Further, consider the discrete grid $\{t_i,~i=0,\dots,N_l\}$ with $t_0=0$, $t_{N_l}=C_F$ and $|t_{i+1}-t_i|=\varepsilon_l$, for $i=0,\dots,N_\ell-1$. 

If the Lévy process is approximated according to Remark \ref{rem:LPApprox} and the GRF is approximated with the truncated KLE, that is
\begin{align*}
W^N(\underline{x}) =\mu_W(\underline{x})+ \sum_{i=1}^N \sqrt{\lambda_i} Z_i e_i(\underline{x}) \approx \mu_W(\underline{x}) + \sum_{i=1}^\infty \sqrt{\lambda_i} Z_i e_i(\underline{x})=W(\underline{x}),~\underline{x}\in \mathcal{D},
\end{align*}
 with some $N\in\mathbb{N}$, i.i.d. standard normal random variables $(Z_i,~i\in\mathbb{N})$ and eigenbasis $((\lambda_i,e_i),~i\in\mathbb{N})$ corresponding to the covariance operator $Q$ of $W$ (cf. Section \ref{subsec:ApprGRF}) and $F(\cdot)<C_F$, we obtain
\begin{align*}
\phi_{l^{(\varepsilon_l)}(F(W^N(\underline{x})))}(\xi)&= \frac{1}{\sqrt{2\pi \sigma_{W,N}^2 (\underline{x})}}\int_\mathbb{R}exp\Big(t_{\lfloor F(y)/\varepsilon_l\rfloor} \psi(\xi) - \frac{(y - \mu_W(\underline{x}))^2}{2\sigma_{W,N}^2(\underline{x})}\Big)dy,~\xi \in\mathbb{R}.
\end{align*}
Here, $\sigma_{W,N}^2$ is the variance function of the approximation $W^N$, which is given by
\begin{align*}
\sigma_{W,N}^2(\underline{x}) := Var(W^N(\underline{x})) = \mathbb{E}\Big(\sum_{i=1}^N \sqrt{\lambda_i} Z_i e_i(\underline{x})\Big)^2 = \sum_{i=1}^N \lambda_i e_i^2(\underline{x}),~\underline{x}\in\mathcal{D}.
\end{align*}
The function $\psi$ denotes the characteristic exponent of $l$ defined by 
\begin{align*}
\psi(\xi) = i\gamma \xi - \frac{b}{2}\xi^2 + \int_{\mathbb{R}\setminus\{0\}}e^{i\xi y}-1-i\xi y\mathds{1}_{\{|y|\leq 1\}}(y)\nu(dy),
\end{align*}
and, for any positive real number $x$, we denote by $\lfloor x \rfloor$ the largest integer smaller or equal than $x$.
\end{corollary}
\begin{proof}
As in the proof of Corollary \ref{COR:CharFctGSLP}, we use Lemma \ref{LE:IteratedExpectationGSLP} to calculate
\begin{align*}
\phi_{l^{(\varepsilon_l)}(F(W^N(\underline{x})))}(\xi)&= \mathbb{E}\Big(\exp(i\xi l^{(\varepsilon_l)}(F(W^N(\underline{x}))))\Big) \\
&=\mathbb{E}(\rho_{l^{(\varepsilon_l)}}(F(W^N(\underline{x})),\xi)),~\xi\in\mathbb{R}.
\end{align*}
In the next step, we compute the charateristic function of the approximation $l^{(\varepsilon_l)}$ of the Lévy process constructed as described in Remark \ref{rem:LPApprox}. We use the independence and stationarity of the increments of the Lévy process to obtain
\begin{align*}
l^{(\varepsilon_l)} (t) &= \sum_{i=1}^{N_l} \mathds{1}_{[t_{i-1},t_i)} (t) l(t_{i-1})= \sum_{i=1}^{N_l} \mathds{1}_{[t_{i-1},t_i)} (t) \sum_{j=1}^{i-1}l(t_j)-l(t_{j-1})
\stackrel{\mathcal{D}}{=} \sum_{k=1}^{\lfloor t/\varepsilon_l \rfloor}l_k^{\varepsilon_l},
\end{align*}
where $(l_k^{\varepsilon_l},~k\in\mathbb{N})$ are i.i.d. random variables following the distribution of $l(\varepsilon_l)$ and we denote by $\stackrel{\mathcal{D}}{=}$ equivalence of the corresponding probability distributions. The convolution theorem (see e.g. \cite[Lemma 15.11]{WTheorie}) yields the representation
\begin{align*}
\phi_{l^{(\varepsilon_l)}(t)}(\xi) &= \prod_{k=1}^{\lfloor t/\varepsilon_l \rfloor} \mathbb{E}(\exp(i\xi l_k^{(\varepsilon_l)}))
= \prod_{k=1}^{\lfloor t/\varepsilon_l \rfloor} \exp(\varepsilon_l \psi(\xi))
 = \exp(\lfloor t/\varepsilon_l \rfloor \varepsilon_l \psi(\xi))
 = \exp(t_{\lfloor t/\varepsilon_l \rfloor} \psi(\xi)),
\end{align*}
for $t\in[0,C_F)$ and $\xi \in \mathbb{R}$. Therefore, we obtain 
 \begin{align*}
 \phi_{l^{(\varepsilon_l)}(F(W^N(\underline{x})))}(\xi) &= \mathbb{E}(\exp(t_{\lfloor F(W^N(\underline{x}))/\varepsilon_l \rfloor} \psi(\xi))),~\xi\in\mathbb{R}.
 \end{align*}
The second formula for the characteristic function of the approximated field now follows from the fact that $W^N(\underline{x})\sim \mathcal{N}(\mu_W(\underline{x}),\sigma_{W,N}^2(\underline{x}))$.
\end{proof}

    %%%%%%%%%%%%%%%%%%%%%%%%%%%%%%%%%%%%%%%%%%%%%%%%%%%%%%
    %%%%%%%%%%%%%%%%%%%%%%%%%%%
	%%%%%%%%%%%%%%%%%%%%%%%%%%%%%%%%%%%%%%%%%%%%%%%%%%%%%%%%%%%%%%%%%%%%%%%%%%%%%%%%%
	%% Covariance
	%%%%%%%%%%%%%%%%%%%%%%%%%%%%%%%%%%%%%%%%%%%%%%%%%%%%%%%%%%%%%%%%%%%%%%%%%%%%%%%%%
	%%%%%%%%%%%%%%%%%%%%%%%%%%%%%%%%%%%%%%%%%%%%%%%%%%%%%%%%%%%%%%%%%%%%%%%%%%%%%%%%%

\section{The covariance structure of GSLFs}\label{sec:Covariance}
In many modeling applications one aims to use a random field model that mimics a specific covariance structure which is, for example, obtained from empirical data. In such a situation it is useful to have access to the theoretical covariance function of the random fields used in the model. Therefore, we derive a formula for the covariance function of the GSLF in the following section.

\begin{lemma}\label{LE:CovGSLP}
Assume $W$ is a GRF and $l$ is an independent Lévy process with existing first and second moment. We deonte by $\mu_W$, $\sigma_W^2$ and $q_W$, the mean, variance and covariance function of $W$ and by $\mu_l(t)=\mathbb{E}(l(t))$, $\mu_l^{(2)}(t):=\mathbb{E}(l(t)^2)$ the functions for the first second moment of the Lévy process $l$. For $\underline{x}\neq \underline{x}'\in\mathcal{D}$, the covariance function of the GSLF $L$ is given by
\begin{align*}
q_L(\underline{x},\underline{x}') &=\int_{\mathbb{R}^2} c_l(F(u),F(v))f_{(W(\underline{x}),W(\underline{x}'))}(u,v)d(u,v) \\
&~-\int_\mathbb{R}\mu_l(F(u))f_{W(\underline{x})}(u)du\int_\mathbb{R}\mu_l(F(v))f_{W(\underline{x}')}(v)dv,
\end{align*}
where we define
\begin{align*}
f_{W(\underline{x})}(u)&:=\frac{exp\Big(-\frac{(u-\mu_W(\underline{x}))^2}{2\sigma_W^2(\underline{x})}\Big)}{\sqrt{2\pi\sigma_W^2(\underline{x})}},\\
\Sigma_W(\underline{x},\underline{x}')&:=\begin{bmatrix}
\sigma_W^2(\underline{x}) & q_W(\underline{x},\underline{x}') \\ 
q_W(\underline{x},\underline{x}')  & \sigma_W^2(\underline{x}')
\end{bmatrix} ,\\
f_{(W(\underline{x}),W(\underline{x}'))}(u,v) &:= \frac{exp\Big(-\frac{1}{2}\begin{pmatrix}u-\mu_W(\underline{x})\\v-\mu_W(\underline{x}')\end{pmatrix}^T\Sigma_W(\underline{x},\underline{x}')^{-1}\begin{pmatrix}u-\mu_W(\underline{x})\\v-\mu_W(\underline{x}')\end{pmatrix}\Big),}{\sqrt{(2\pi)^2\big(\sigma_W^2(\underline{x})\sigma_W^2(\underline{x}') - q_W^2(\underline{x},\underline{x}')\big)}}\\
c_l(u,v) &:= \mu_l(|u-v|)\mu_l(u\wedge v) + \mu_l^{(2)}(u\wedge v),
\end{align*}
for $u,v\in\mathbb{R}$ with $u\wedge v:=\min(u,v)$. For $\underline{x} = \underline{x}'\in\mathcal{D}$, the pointwise variance is given by 
\begin{align*}
\sigma_L^2(\underline{x}) = q_L(\underline{x},\underline{x}) = \int_\mathbb{R}\mu_\ell^{(2)}(F(u))f_{W(\underline{x})}(u)du - \Big(\int_\mathbb{R}\mu_\ell(F(u))f_{W(\underline{x})}(u)du \Big)^2.
\end{align*}
\end{lemma}
\begin{proof}
We compute using Lemma  \ref{LE:IteratedExpectationGSLP}, for $\underline{x}\in\mathcal{D}$
\begin{align*}
\mu_L(\underline{x}):=\mathbb{E}(L(\underline{x})) = \mathbb{E}\Big(l(F(W(\underline{x})))\Big)  = \mathbb{E}\Big(\mu_l(F(W(\underline{x}))))\Big).
\end{align*}
Further, we calculate for $\underline{x},\underline{x}'\in\mathcal{D}$
\begin{align*}
q_L(\underline{x},\underline{x}')&=\mathbb{E}\Big((L(\underline{x}) - \mu_L(\underline{x}))(L(\underline{x}') - \mu_L(\underline{x}'))\Big)\\
&= \mathbb{E}(L(\underline{x})L(\underline{x}')) - \mu_L(\underline{x})\mu_L(\underline{x}').
\end{align*}
Next, we consider $0\leq t_1\leq t_2$ and use the fact that $l$ has stationary and independent increments to compute
\begin{align*}
\mathbb{E}(l(t_1)l(t_2)) = \mathbb{E}\big((l(t_2)-l(t_1))l(t_1)\big) + \mathbb{E}(l(t_1)^2) = \mu_l(t_2-t_1)\mu_l(t_1) + \mu_l^{(2)}(t_1).
\end{align*}
Similarly, we obtain for $0\leq t_2\leq t_1$
\begin{align*}
\mathbb{E}(l(t_1)l(t_2)) = \mathbb{E}\big((l(t_1)-l(t_2))l(t_2)\big) + \mathbb{E}(l(t_2)^2) = \mu_l(t_1-t_2)\mu_l(t_2) + \mu_l^{(2)}(t_2),
\end{align*}
and, hence, it holds for general $t_1,t_2\geq 0$:
\begin{align*}
c_l(t_1,t_2):=\mathbb{E}(l(t_1)l(t_2)) = \mu_l(|t_1-t_2|)\mu_l(t_1 \wedge t_2) + \mu_l^{(2)}(t_1\wedge t_2).
\end{align*}
Another application of Lemma \ref{LE:IteratedExpectationGSLP} and Remark \ref{rem:ExtensionIteratedExpectation} yields
\begin{align*}
\mathbb{E}(L(\underline{x})L(\underline{x}')) = \mathbb{E}(l(F(W(\underline{x})))l(F(W(\underline{x}')))) = \mathbb{E}(c_l(F(W(\underline{x})),F(W(\underline{x}')))).
\end{align*}
Putting these results together, we end up with
\begin{align*}
q_L(\underline{x},\underline{x}') &=  \mathbb{E}(L(\underline{x})L(\underline{x}')) - \mu_L(\underline{x})\mu_L(\underline{x}')\\
& =  \mathbb{E}(c_l(F(W(\underline{x})),F(W(\underline{x}')))) - \mathbb{E}(\mu_l(F(W(\underline{x}))))\mathbb{E}(\mu_l(F(W(\underline{x'})))),
\end{align*}
for $\underline{x},~\underline{x}'\in\mathcal{D}$. The assertion now follows from the fact that $W(\underline{x})\sim \mathcal{N}(\mu_W(\underline{x}),\sigma_W^2(\underline{x}))$ and $(W(\underline{x}),W(\underline{x}'))^T\sim \mathcal{N}_2((\mu_W(\underline{x}),\mu_W(\underline{x}'))^T, \Sigma_W(\underline{x},\underline{x}'))$, for $\underline{x}\neq\underline{x}'\in \mathcal{D}$, together with $c_\ell(t,t) = \mu_\ell^{(2)}(t)$ for $t\geq 0$.
\end{proof}

    %%%%%%%%%%%%%%%%%%%%%%%%%%%%%%%%%%%%%%%%%%%%%%%%%%%%%%%%%%%%%%%%%%%%%%%%%%%%%%%%%
	%%%%%%%%%%%%%%%%%%%%%%%%%%%%%%%%%%%%%%%%%%%%%%%%%%%%%%%%%%%%%%%%%%%%%%%%%%%%%%%%%
	%% Numerical examples
	%%%%%%%%%%%%%%%%%%%%%%%%%%%%%%%%%%%%%%%%%%%%%%%%%%%%%%%%%%%%%%%%%%%%%%%%%%%%%%%%%
	%%%%%%%%%%%%%%%%%%%%%%%%%%%%%%%%%%%%%%%%%%%%%%%%%%%%%%%%%%%%%%%%%%%%%%%%%%%%%%%%%

\section{Numerical examples}\label{sec:numerics}
%~\mero{TODO:\begin{itemize}
%\item richtige Struktur in den Matlab folder bringen! (done for ellPDE und MomentsLP and Pointwise distribution and Sampling and approximation and Covariance)
%\item codes druchgehen und aufräumen (done for ellPDE und MomentsLP and Pointwise distribution and Sampling and Approximation and Covariance)
%\end{itemize} }
In this section, we present numerical experiments on the theoretical results presented in the previous sections. The aim is to investigate the results of existing numerical methods and to illustrate the theoretical properties of the GSLF which have been proven in the previous sections, e.g. the pointwise distribution of the approximated fields or the quality of this approximation (see Theorem \ref{TH:GSLPAPPR}). Further, the presented numerical experiments and methods may also be useful for fitting the GSLF to existing data in various applications.

\subsection{Pointwise characteristic function}\label{subsec:PWCharFctvalid}
Corollary \ref{COR:CharFctGSLP} gives access to the pointwise characteristic function of the GSLF $L(\underline{x})=l(F(W(\underline{x}))),~\underline{x}\in\mathcal{D}$, with a Lévy process $l$ and a GRF $W$, which is independent of $l$. Using the characteristic function and the Fourier inversion (FI) method (see \cite{NoteOnTheInversionTheorem}) we may compute the pointwise density function of the GSLF. Note that in both of these steps, the application of Corollary \ref{COR:CharFctGSLP} and of the Fourier inversion theorem, numerical integration is necessary which may be inaccurate or computationally expensive. In this subsection, we choose specific Lévy processes together with a GRF and compare the computed density function of the corresponding GSLF with the histogram of samples of the simulated field. To be precise, we choose a Matérn-1.5-GRF, a Gamma process, set $F=|\cdot|+1$ and consider the GSLF $L(\underline{x})=l(F(W(\underline{x}))$ on $\mathcal{D}=[0,1]^2$. The distributions and the corresponding density functions are presented in Figure \ref{FIG:NumExPWDist}. In line with our expectations, the FI approach perfectly matches the true distribution of the field at $(1,1)$.
\begin{figure}[ht]
\begin{center}
\subfigure{\includegraphics[scale=0.5]{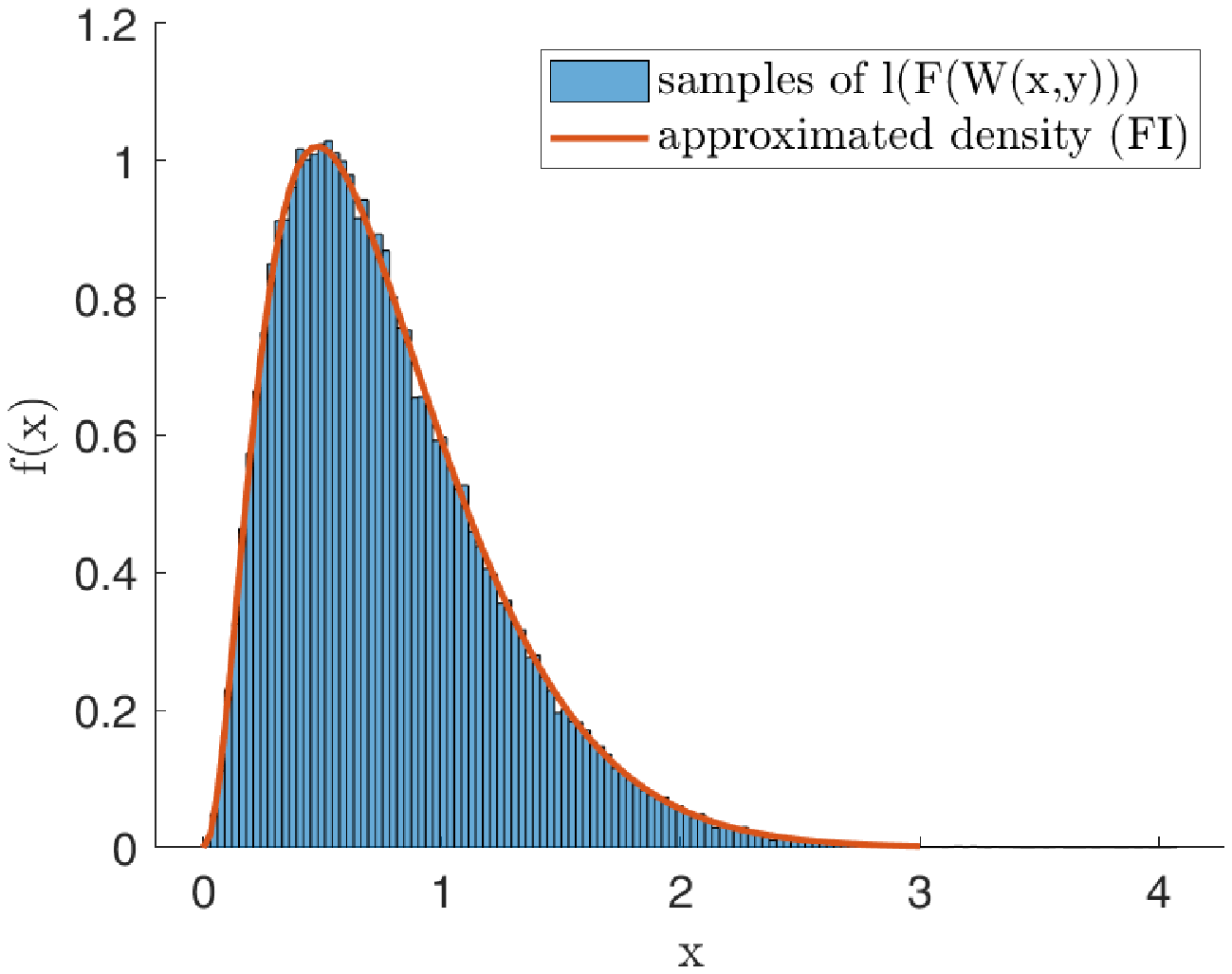}}
\subfigure{\includegraphics[scale=0.5]{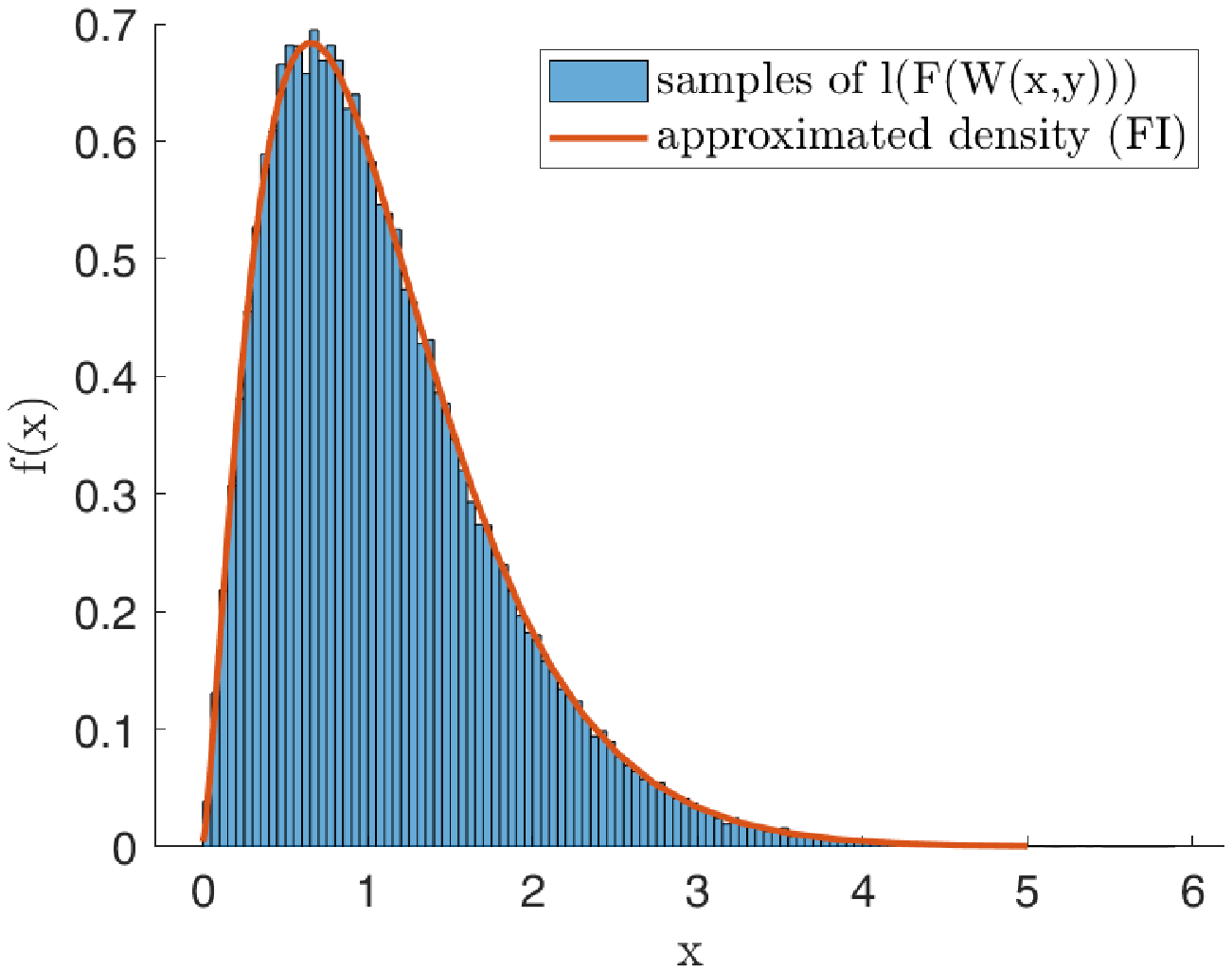}}
\caption{$10^5$ samples of the GSLF evaluated at $(1,1)$ and the corresponding density functions approximated via FI. Left: Gamma(3,10)-process and a Matérn GRF with $\sigma=2$. Right: Gamma(2,4)-process with Matérn GRF with $\sigma=1.5$.}
\label{FIG:NumExPWDist}
\end{center}
\end{figure}

\subsection{Pointwise distribution of approximated fields}\label{subsec:PWCharFctvalidApprFields}

In Subsection \ref{subsec:PWCharFctvalid} we presented a numerical example regarding the pointwise distribution of the GSLF. In applications, it is not possible to draw samples from the exact GSLF and, hence, one has to use  approximations instead. The effect of such an approximation on the pointwise distribution of the random field has been investigated theoretically in Subsection \ref{subsec:ApprLevyPWDist}. In this section, we aim to provide a numerical example in order to visualize the distributional effect of sampling from an approximation of the GSLF. For a specific choice of the Lévy process $l$, the transformation function $F$ and the GRF $W$, we use Corollary \ref{COR:CharFctApprGSLP} and the FI method to compute the pointwise distribution of the approximated field $l^{(\varepsilon_l)}(F(W^N(\underline{x}))) \approx l(F(W(\underline{x})))$ for different approximation parameters $\varepsilon_l$ (resp. $N$) of the Lévy process (resp. the GRF). The computed densities are then compared with samples of the approximated field. We set $\mathcal{D}=[0,1]^2$ and consider the evaluation point $\underline{x}=(0.4,0.6)$. The GRF $W$ is given by the KLE
\begin{align*}
W=\sum_{k=1}^\infty \sqrt{\lambda_k} e_k(x,y)Z_k,
\end{align*}
where the eigenbasis is defined by
\begin{align*}
e_k(x,y)=2\, sin\left(\pi kx\right)sin\left(\pi ky\right),~\lambda_k=10(k^2\pi^2+0.1^2)^{-\nu},
\end{align*}
with $\nu=0.6$ (see \cite[Appendix A]{fasshauer2015kernel}). Further, we set $F(x)=1+\min(|x|,30)$ and choose $l$ to be a Gamma(3,10)-process. The threshold $30$ in the definition of $F$ is large enough to have no effect in our numerical example since the absolute value of $W$ does not exceed 30 in all considered samples. We use piecewise constant approximations $l^{(\varepsilon_l)}$ of the Lévy process $l$ on an equidistant grid with stepsize $\varepsilon_l>0$ (see Remark \ref{rem:LPApprox}) and approximate the GRF $W$ by the truncated KLE with varying truncation indices $N$. To be more precise, we choose 6 different approximation levels for these two approximation parameters, as described in Table \ref{TAB:PWDistApprFieldApprParams}.
\begin{table}
\begin{center}
\begin{tabular}{|c|c|c|c|c|c|c|}
 \hline 
 level & 1 & 2 & 3 & 4 & 5 & 6 \\ 
 \hline 
 $\varepsilon_l$ & 1 & 0.5 & 0.01 & 0.001 & 0.0001 & 1e-7 \\ 
 \hline 
 $N$ & 3 & 5 & 250 & 2500 & 25000 & 1e7 \\ 
 \hline  
 \end{tabular} 
 \caption{Discretization parameters for the approximation of the GSLF.}
\label{TAB:PWDistApprFieldApprParams}
\end{center} 
 \end{table}
For each discretization level, we compute the characteristic function using Corollary \ref{COR:CharFctApprGSLP}, approximate the density function via FI and compare it with $10^5$ samples of the approximated field evaluated at the point $(0.4,0.6)$. The results are given in Figure \ref{FIG:NumExPWDistApprLevels}.
\begin{figure}[ht]
\begin{center}
\subfigure{\includegraphics[scale=0.33]{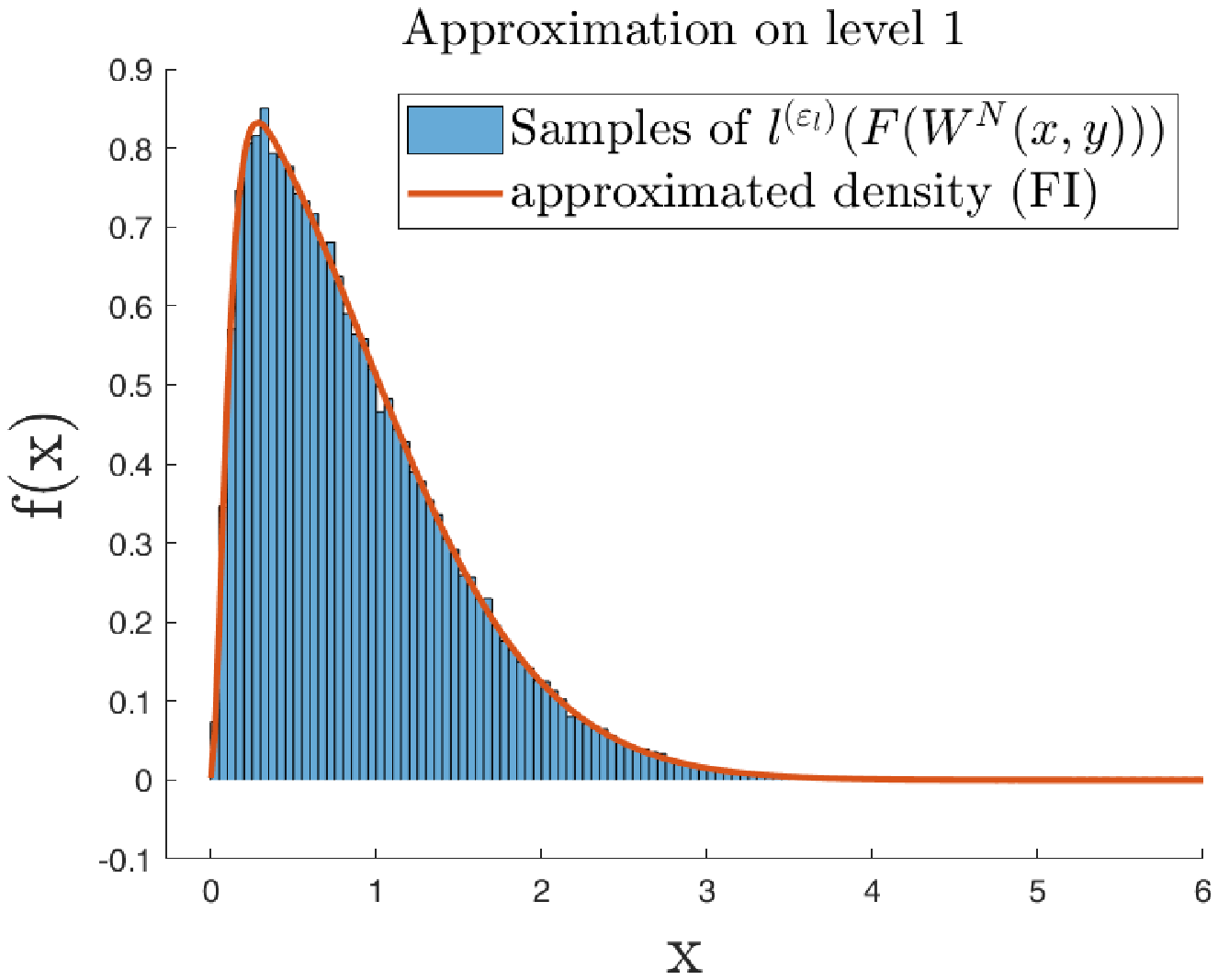}}
\subfigure{\includegraphics[scale=0.33]{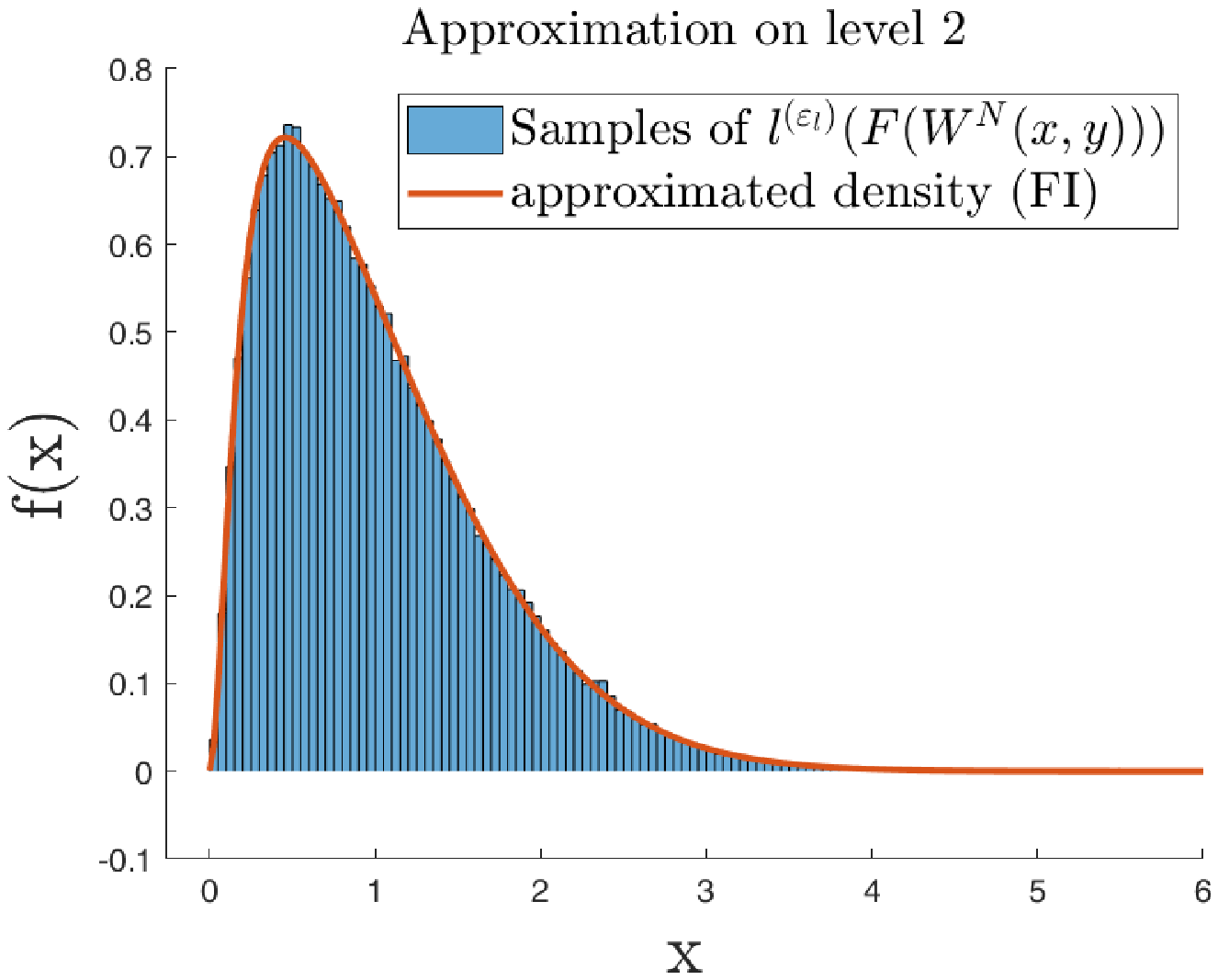}}
\subfigure{\includegraphics[scale=0.33]{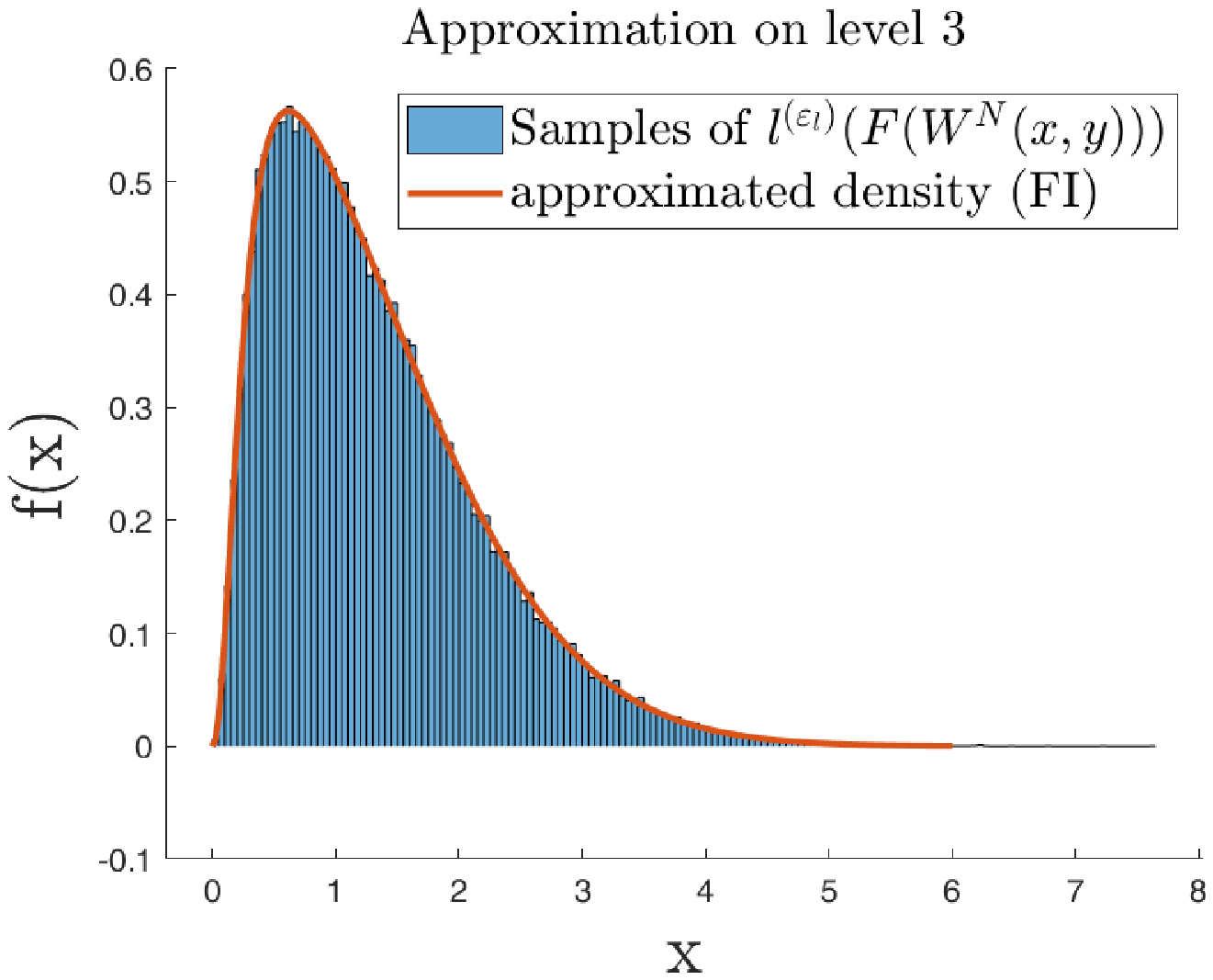}}\\
\subfigure{\includegraphics[scale=0.33]{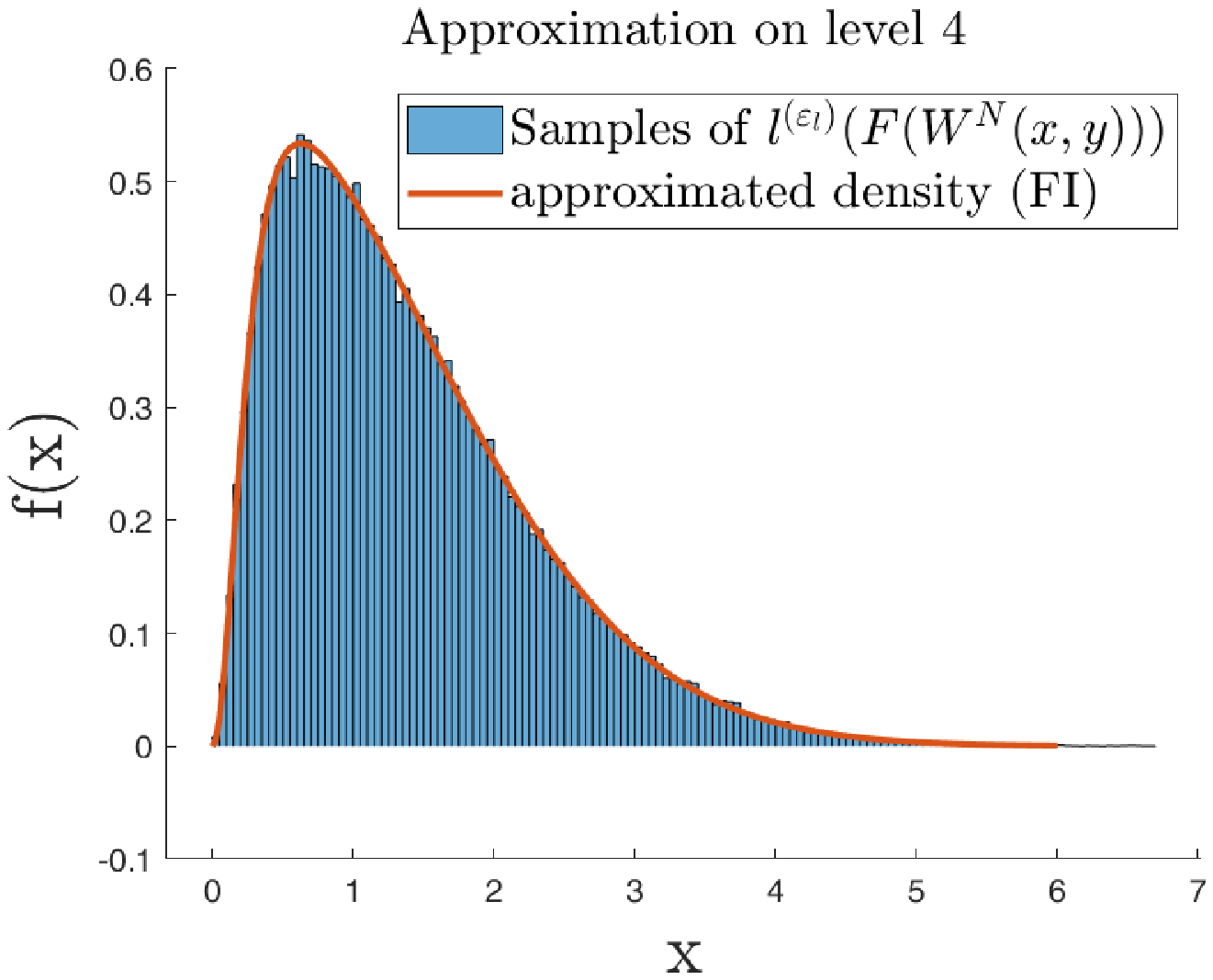}}
\subfigure{\includegraphics[scale=0.33]{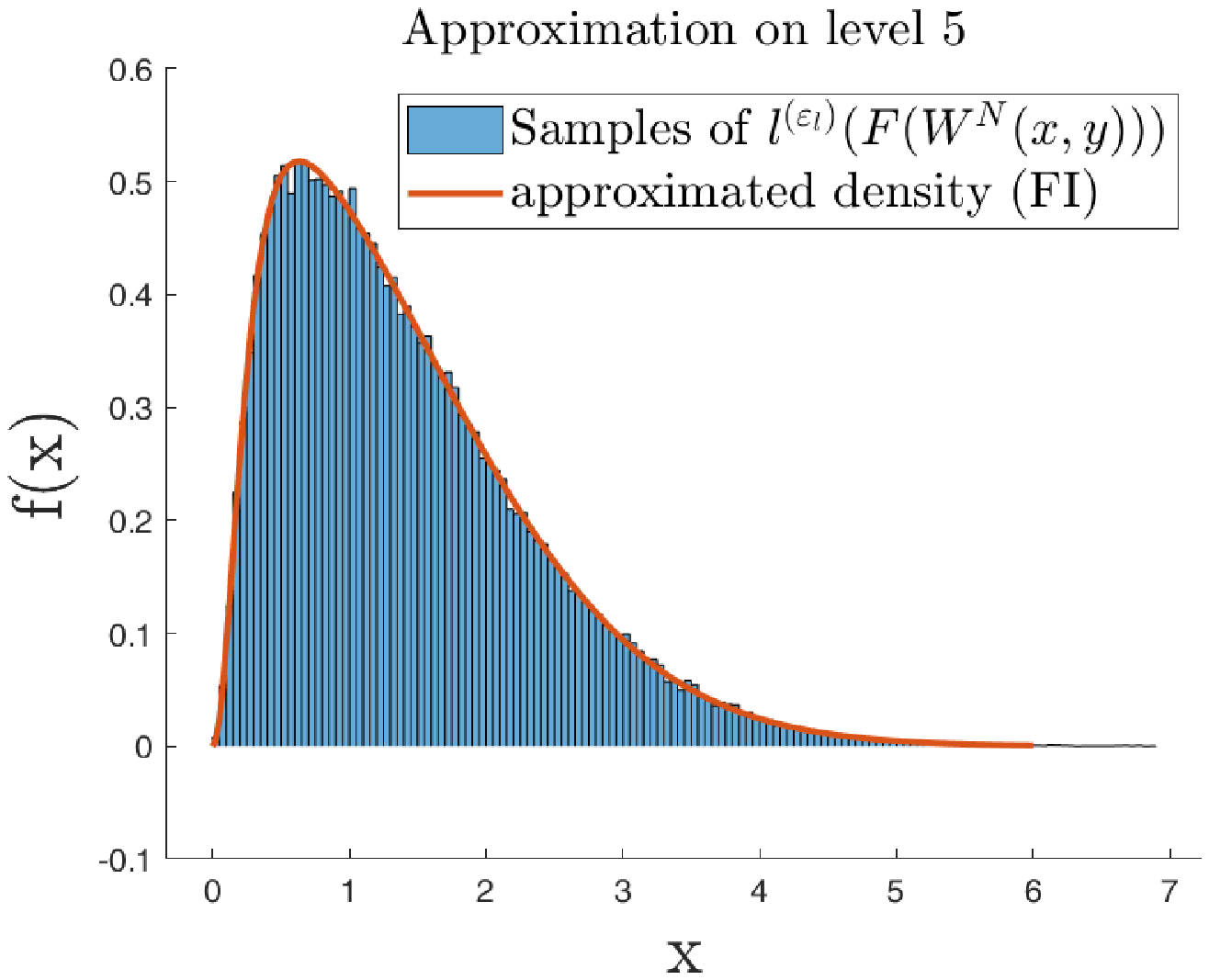}}
\subfigure{\includegraphics[scale=0.33]{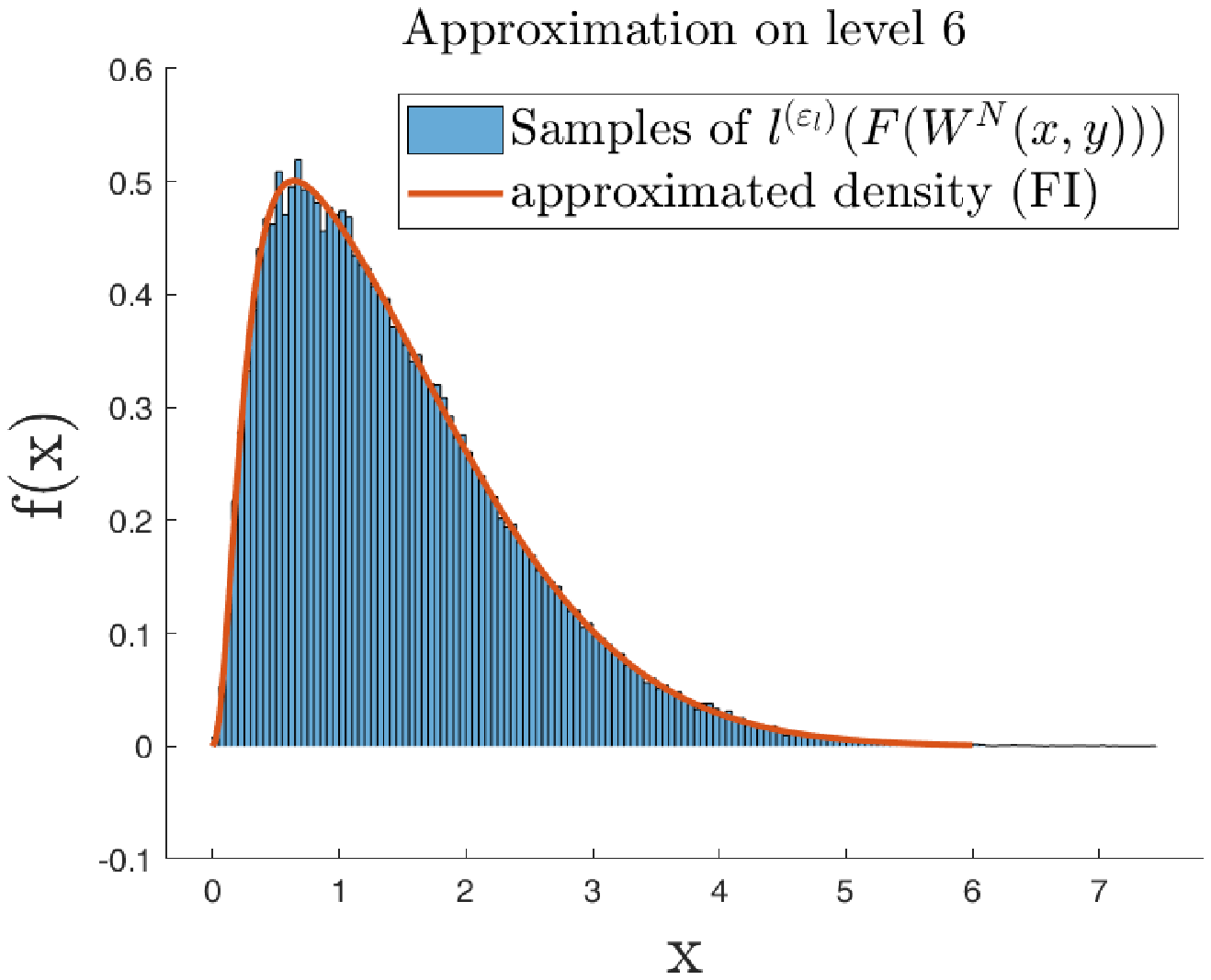}}
\caption{Approximated pointwise densities (FI) of the GSLF on the different discretization levels together with samples of the corresponding approximated GSLF.}
\label{FIG:NumExPWDistApprLevels}
\end{center}
\end{figure}
As expected from Subsection \ref{subsec:PWCharFctvalid}, we see that the densities of the pointwise distribution of the approximated GSLF, which are approximated by FI, fit the actual samples of the approximated field accurately. In order to get a better impression of the influence of the approximation on the different levels, Figure \ref{FIG:NumExPWDistApprDensitiesOverview} shows the densities of the evaluated approximated GSLF on the different levels, together with the density of the exact field. We see how the densities of the approximted GSLF converge to the density of the exact field for $\varepsilon_l\rightarrow 0$ and $N \rightarrow \infty$. The results also show that the effect of the approximation of the GSLF should not be unterestimated: on the lower levels, we obtain comparatively large deviations of the pointwise densities from the density of the exact field, which should be taken into account in applications. Obviously, the effect of the approximation depends heavily on the specific choice of the Lévy process and the underlying GRF.
\begin{figure}[ht]
\begin{center}
\subfigure{\includegraphics[scale=0.7]{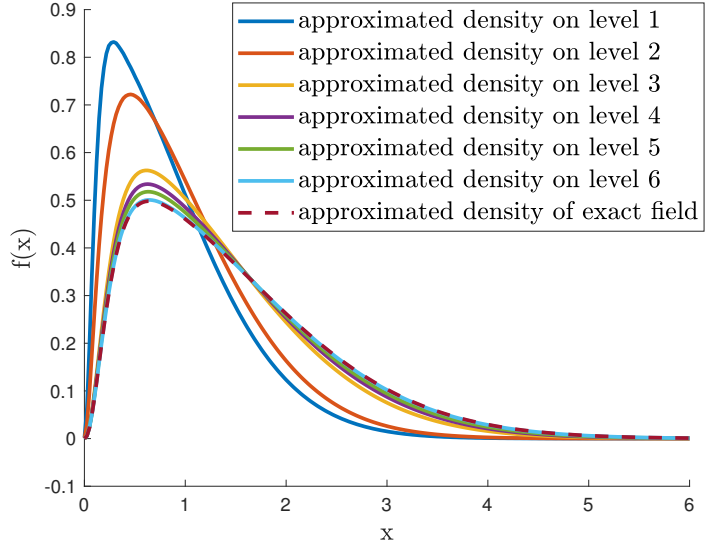}}
\caption{Approximated pointwise densities (FI) of the GSLF on the different discretization levels together with the pointwise density of the actual GSLF.}
\label{FIG:NumExPWDistApprDensitiesOverview}
\end{center}
\end{figure}

\subsection{Numerical approximation of the GSLF}
    
In Section \ref{sec:approximation}, we considered approximations $l^{(\varepsilon_l)}\approx l$ of the Lévy process and $W^N\approx W$ of the GRF and derived an error bound on the corresponding approximation $l^{(\varepsilon_l)}(F(W^N(\underline{x})))\approx l(F(W(\underline{x})))$ (see Theorem \ref{TH:GSLPAPPR}). In fact, under Assumption \ref{ASS:CutProblemEigenvalues}, if we choose the approximation parameter $N$ of the GRF $W$ such that $R(N) \sim \varepsilon_l^{1/\delta}$ with $R(N)$ from Equation  \eqref{EQ:ApproximationErrorGRFGeneral}, we obtain an overall approximation error which is dominated by $\varepsilon_l^{1/p}$:
\begin{align*}
\|l^\apprlevy(F(W^N)) - l(F(W))\|_{L^p(\Omega;L^p(\mathcal{D}))} = \mathcal{O}(\varepsilon_l^{1/p}),~\varepsilon_l\rightarrow 0.
\end{align*}
In this section, we present numerical experiments to showcase this approximation result.

We set $F(x)=\min(|x|,30)$ and consider the domain $\mathcal{D}=[0,1]^2$. Let
\begin{align*}
W(x,y)=\sum_{k=1}^\infty \sqrt{\lambda_k} e_k(x,y)Z_k,
\end{align*}
be a GRF with corresponding eigenbasis
\begin{align}\label{EQ:NumExGSLPApprEigenbasis}
e_k(x,y)=2 sin\left(\pi kx\right)sin\left(\pi ky\right),~\lambda_k=100\,(k^2\pi^2+0.2^2)^{-\nu}, ~\nu>0.5.
\end{align}
With this choice, Assumption \ref{ASS:CutProblemEigenvalues} \textit{i} is satisfied with $\alpha = 1$ and $0<\beta < 2\nu-1$ , since
\begin{align*}
\sum_{k=1}^\infty \lambda_k k^\beta\leq C\sum_{k=1}^\infty k^{\beta-2\nu}<\infty,
\end{align*}
where we used that $\lambda_k\leq C k^{-2\nu}$ for $k\in\mathbb{N}$. Hence, we obtain by Lemma \ref{LE:StrongApprGRFKLE}
\begin{align*}
\|W-W^N\|_{L^n(\Omega;L^\infty(\mathcal{D}))}\leq C(D,n)N^{-\frac{\beta}{2}},
\end{align*}
for $0<\beta<2\nu-1$, i.e. $R(N)= N^{-\frac{\beta}{2}}$ in Equation \eqref{EQ:ApproximationErrorGRFGeneral} and Theorem \ref{TH:GSLPAPPR}. In our experiments, we choose the Lévy subordinator $l$ to be a Poisson or Gamma process. Approximations $l^{(\varepsilon_l)}\approx l$ of these processes satisfying of Assumption \ref{ASS:CutProblemEigenvalues} \textit{iii} may be obtained by piecewise constant extensions of values of the process on a grid with stepsize $\varepsilon_l$ (see Remark \ref{rem:LPApprox}). For these processes, we obtain $\delta=1$ in \eqref{EQ:MomentsLevy} (see Remark \ref{rem:LPShortTimeBehaviour} and \cite[Section 7]{SGRFPDE}). Overall, Theorem \ref{TH:GSLPAPPR} yields the error bound
\begin{align*}
\|l^\apprlevy(F(W^N)) - l(F(W))\|_{L^p(\Omega;L^p(\mathcal{D}))} = C(\varepsilon_l^\frac{1}{p} + N^{-\frac{\beta}{2p}}),
\end{align*}
for $0<\beta<2\nu-1$. Hence, in order to equilibrate the error contributions from the GRF approximation and the approximation of the Lévy process, one should choose the cut-off index $N$ of the KLE approximation of the GRF $W$ according to 
\begin{align}\label{EQ:NumExGSLPApprChoiceN}
N \simeq \varepsilon_l^{-\frac{2}{2\nu-1}},
\end{align}
which then implies
\begin{align}\label{EQ:NumExGSLPApprErrorBound}
\|l^\apprlevy(F(W^N)) - l(F(W))\|_{L^p(\Omega;L^p(\mathcal{D}))} \leq  C\varepsilon_l^\frac{1}{p}.
\end{align}
In our experiments, we set the approximation parameters of the Lévy process to be $\varepsilon_l = 2^{-\ell}$, for $\ell=2,\dots,10$ and the cut-off indices of the GRF are choosen according to \eqref{EQ:NumExGSLPApprChoiceN}. In order to verify \eqref{EQ:NumExGSLPApprErrorBound}, we draw 150 samples of the random variable 
\begin{align*}
l^\apprlevy(F(W^N)) - l(F(W)),
\end{align*}
for the described approximation parameters to estimate the corresponding $L^p$ norm of the approximation error for $p \in\{1,2,2.5,3,3.3\}$ which is expected to behave as $\mathcal{O}(\varepsilon_l^{1/p})$. 

In our first example, we set $l$ to be a Poisson(8)-process and $\nu=2.5$ in Equation~\eqref{EQ:NumExGSLPApprEigenbasis}. For each sample, we compute a reference field by taking $150$ summands in the KLE of the GRF $W$ and use an exact sample of a Poisson process computed by the Uniform Method (see \cite[Section 8.1.2]{LevyProcessesInFinance}). The resulting estimates for the $L^p$ approximation error are plotted against $\varepsilon_l$ for $p\in\{1,2,2.5,3.3\}$, which is illustrated in the Figure \ref{FIG:NumExApprGSLPPoiss}.
\begin{figure}[ht]
\begin{center}
\subfigure{\includegraphics[scale=0.24]{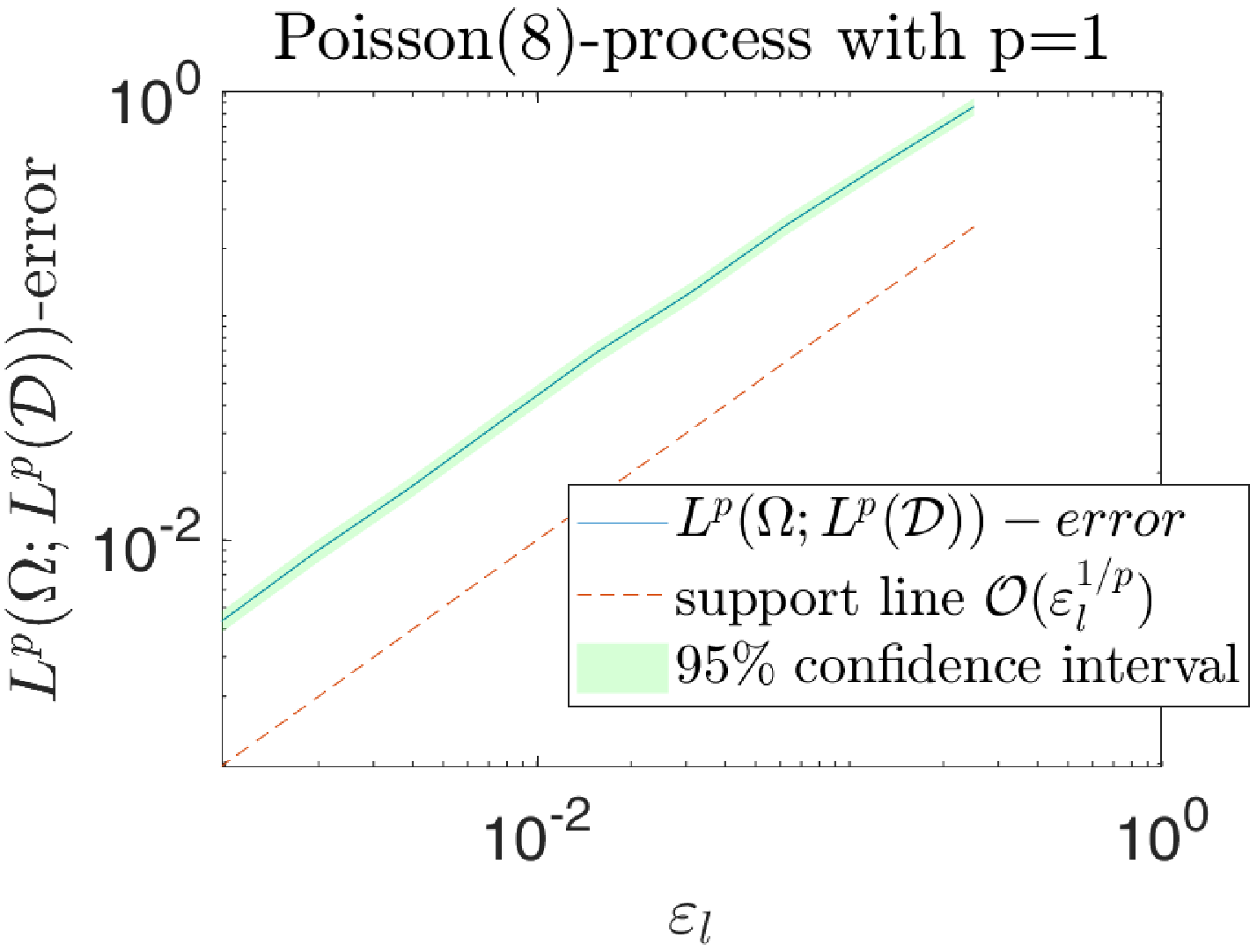}}
\subfigure{\includegraphics[scale=0.24]{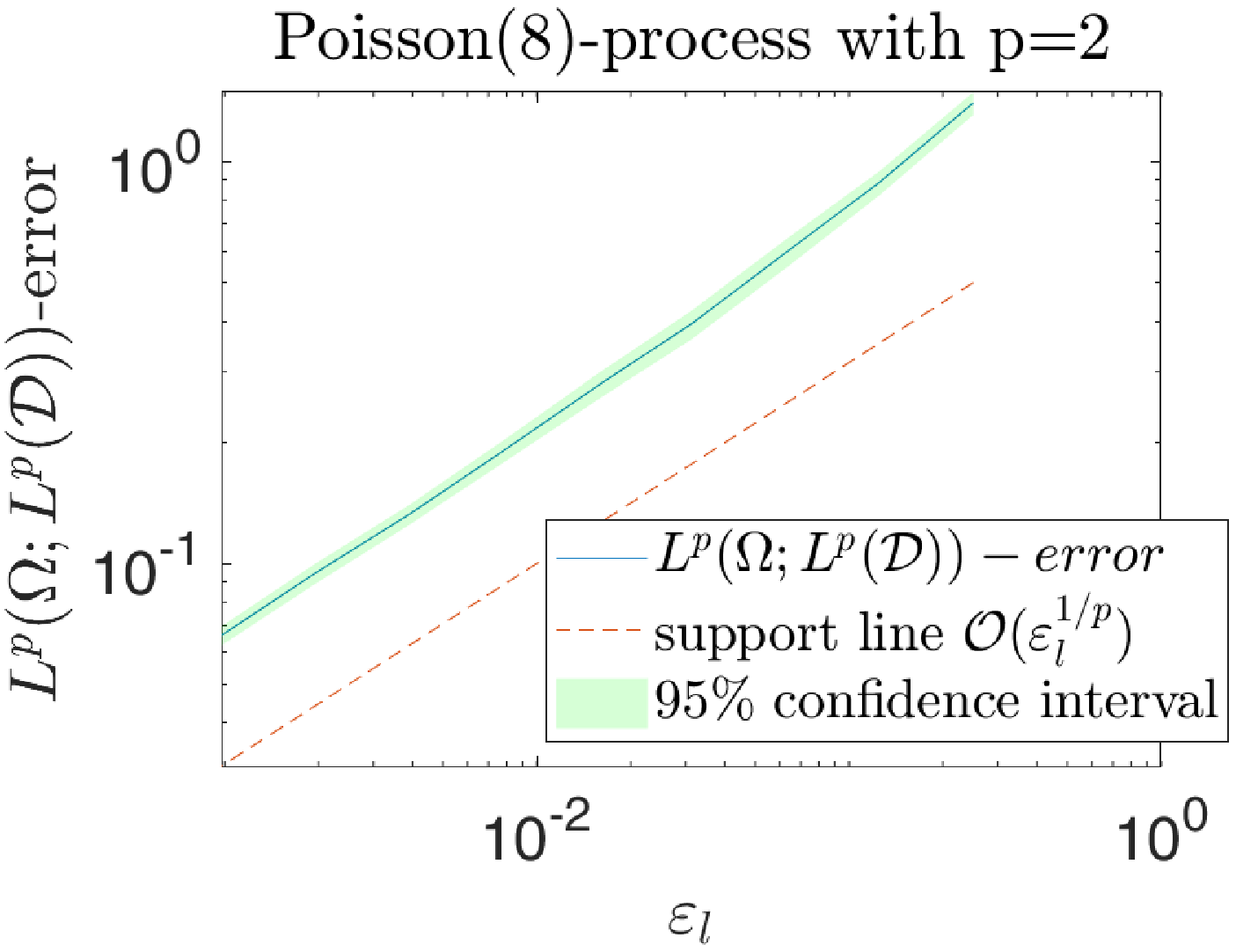}}
\subfigure{\includegraphics[scale=0.24]{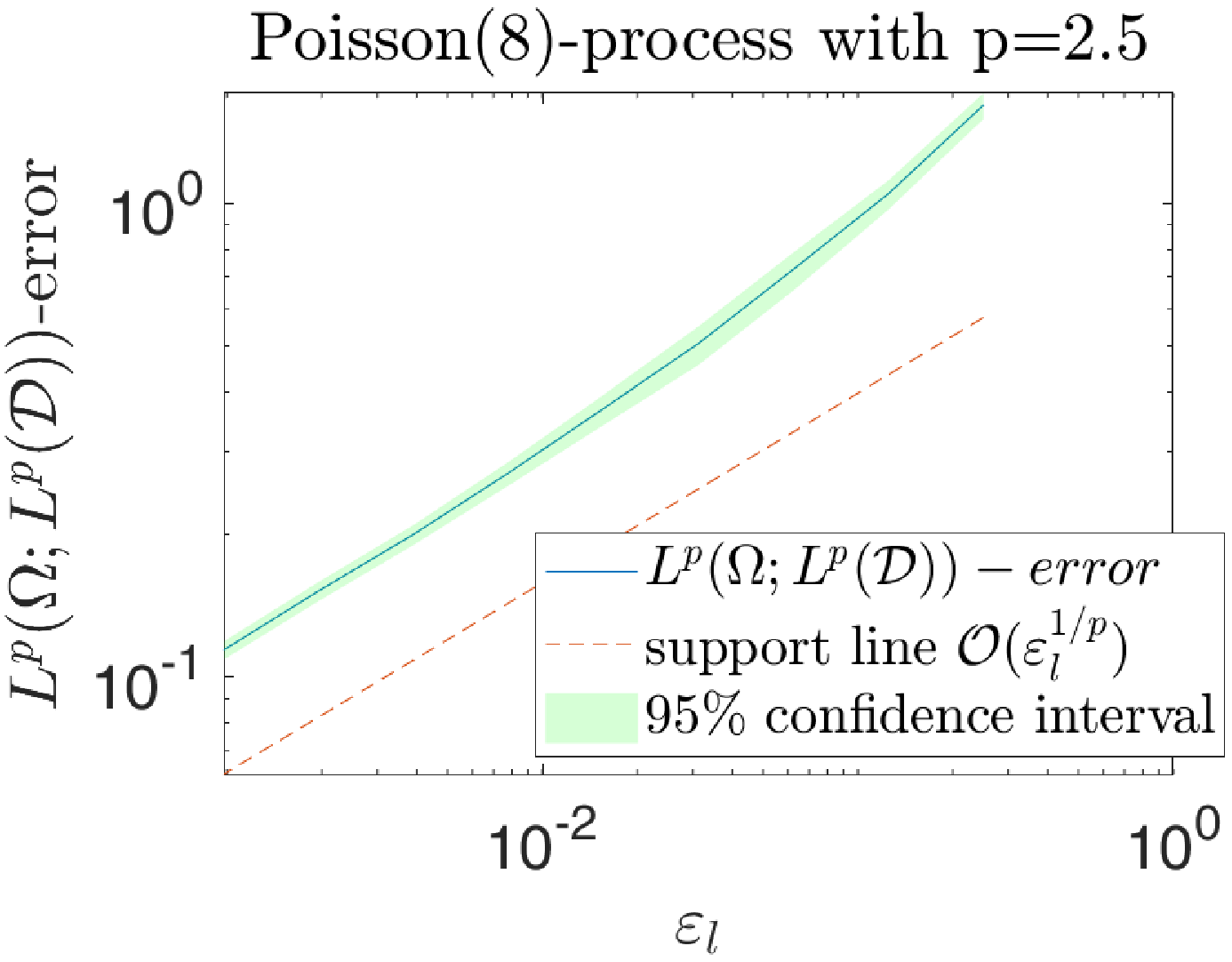}}
\subfigure{\includegraphics[scale=0.24]{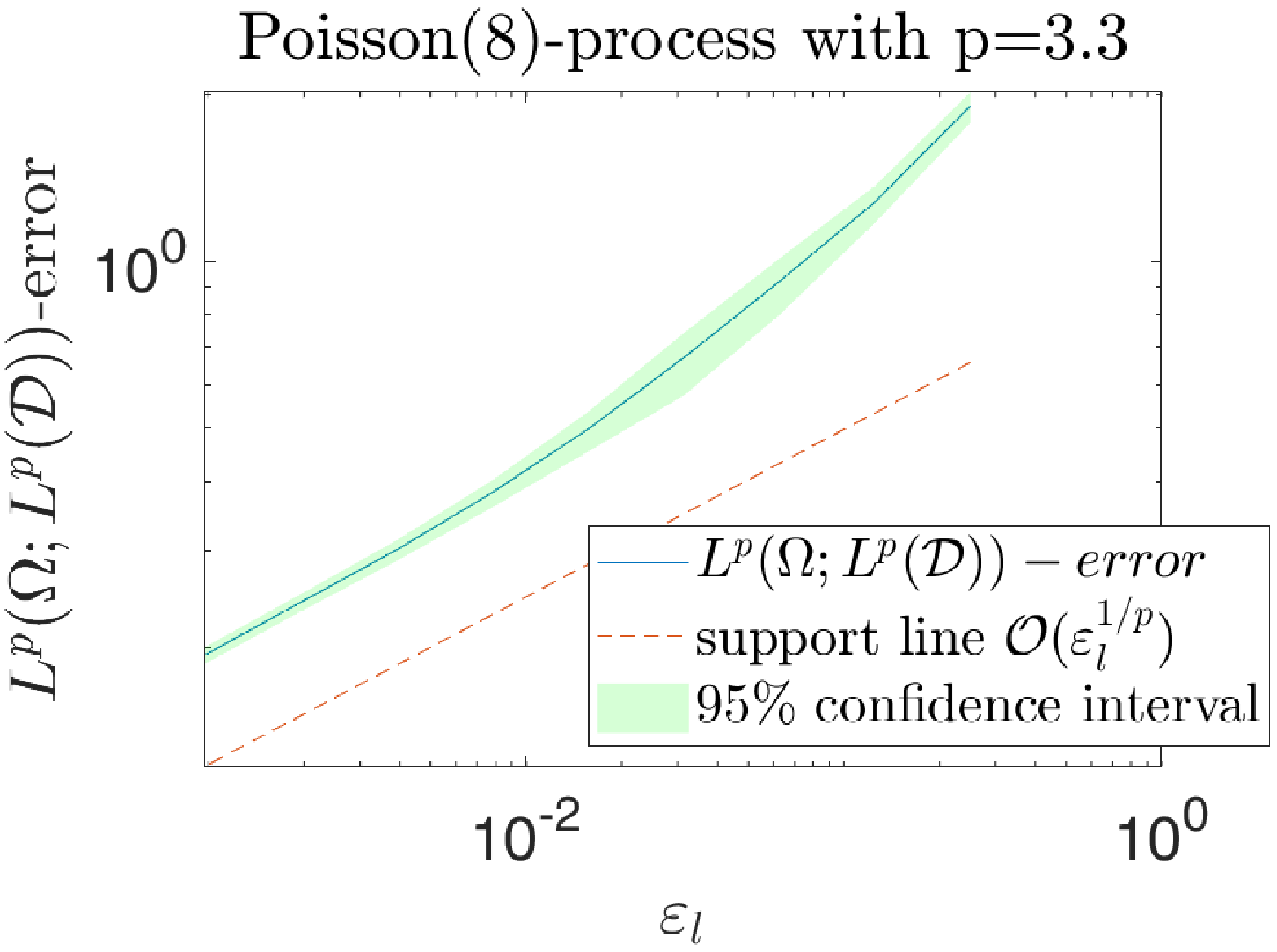}}\\
\subfigure{\includegraphics[scale=0.7]{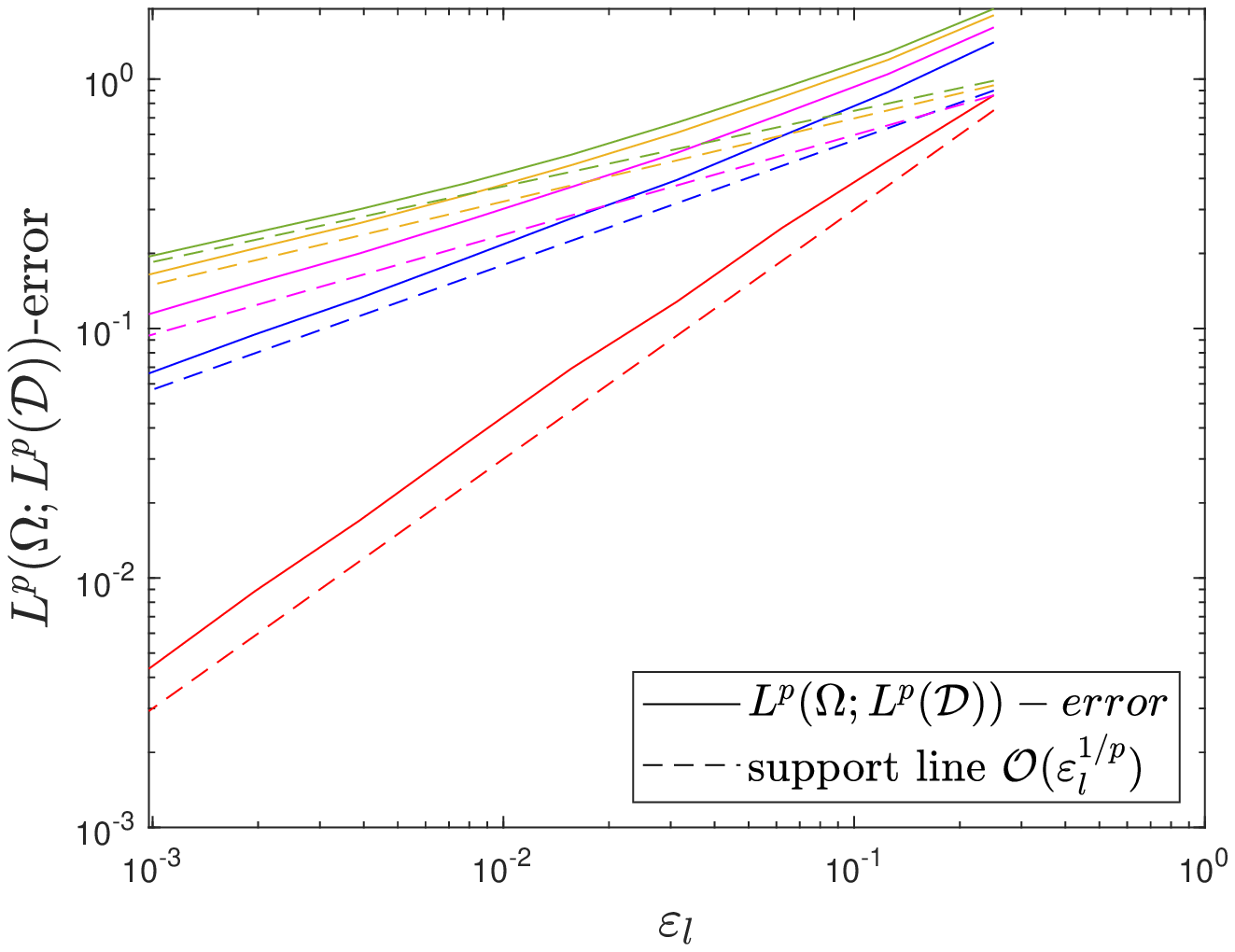}}
\caption{Estimated $L^p$ approximation errors against the discretization parameter $\varepsilon_l$ for different values of $p$ using a Poisson(8)-process (top) and overview for $p = \textcolor{r}{1},~\textcolor{b}{2},~\textcolor{m}{2.5},~\textcolor{dy}{3},~\textcolor{dg}{3.3}$ (bottom).}
\label{FIG:NumExApprGSLPPoiss}
\end{center}
\end{figure}
As expected, the approximated $L^p$ errors converge asymptotically with rate $\varepsilon_l^{1/p}$ for the considered moments $p$. The bottom plot of Figure~\ref{FIG:NumExApprGSLPPoiss} shows an overview of the $L^p$ approximation errors for different values of $p$ and  illustrates that the $L^p$-errors indeed converge with different rates.

In our second numerical example, we set $l$ to be a Gamma(2,4)-process. Further, we set $\nu=2$ in Equation~\eqref{EQ:NumExGSLPApprEigenbasis}. The approximations $l^{(\varepsilon_l)}\approx l$ of the Lévy process on the different levels are again computed by piecewise constant extensions of values of the process on an equidistant grid with stepsize $\varepsilon_l$. Aiming to verify Equation~\eqref{EQ:NumExGSLPApprErrorBound}, we use $150$ samples to estimate the $L^p$ approximation error, for $p \in\{1,2,2.5,3,3.3\}$. In order to compute a sufficiently accurate reference field in each sample, we use $500$ summands in the KLE approximation of the GRF $W$ and the Gamma process is computed on a reference grid with stepsize $\varepsilon_l=2^{-13}$. The convergence of the estimated $L^p$ error plotted against the discretization parameter $\varepsilon_l$ is visualized in Figure~\ref{FIG:NumExApprGSLPGam}, which shows the expected behaviour of the estimated $L^p$ error for the considered moments $p$. As in the previous experiment, we provide a plot with all estimated $L^p$ errors in one figure (see bottom plot in Figure \ref{FIG:NumExApprGSLPGam}), which again confirms that the $L^p$-error of the approximation converge in $\varepsilon_l$ with rate $1/p$ for the considered values of $p$.
\begin{figure}[ht]
\begin{center}
\subfigure{\includegraphics[scale=0.24]{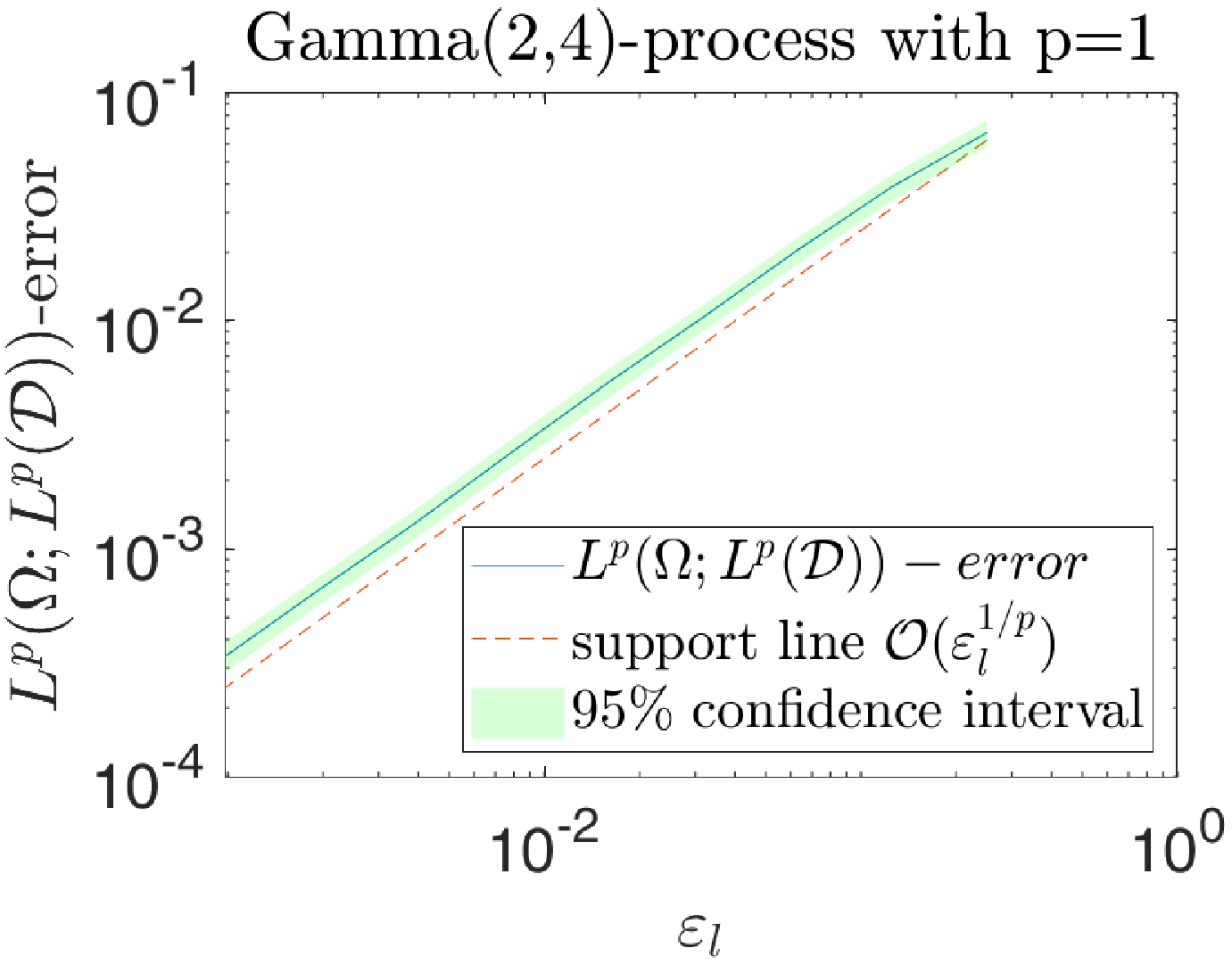}}
\subfigure{\includegraphics[scale=0.24]{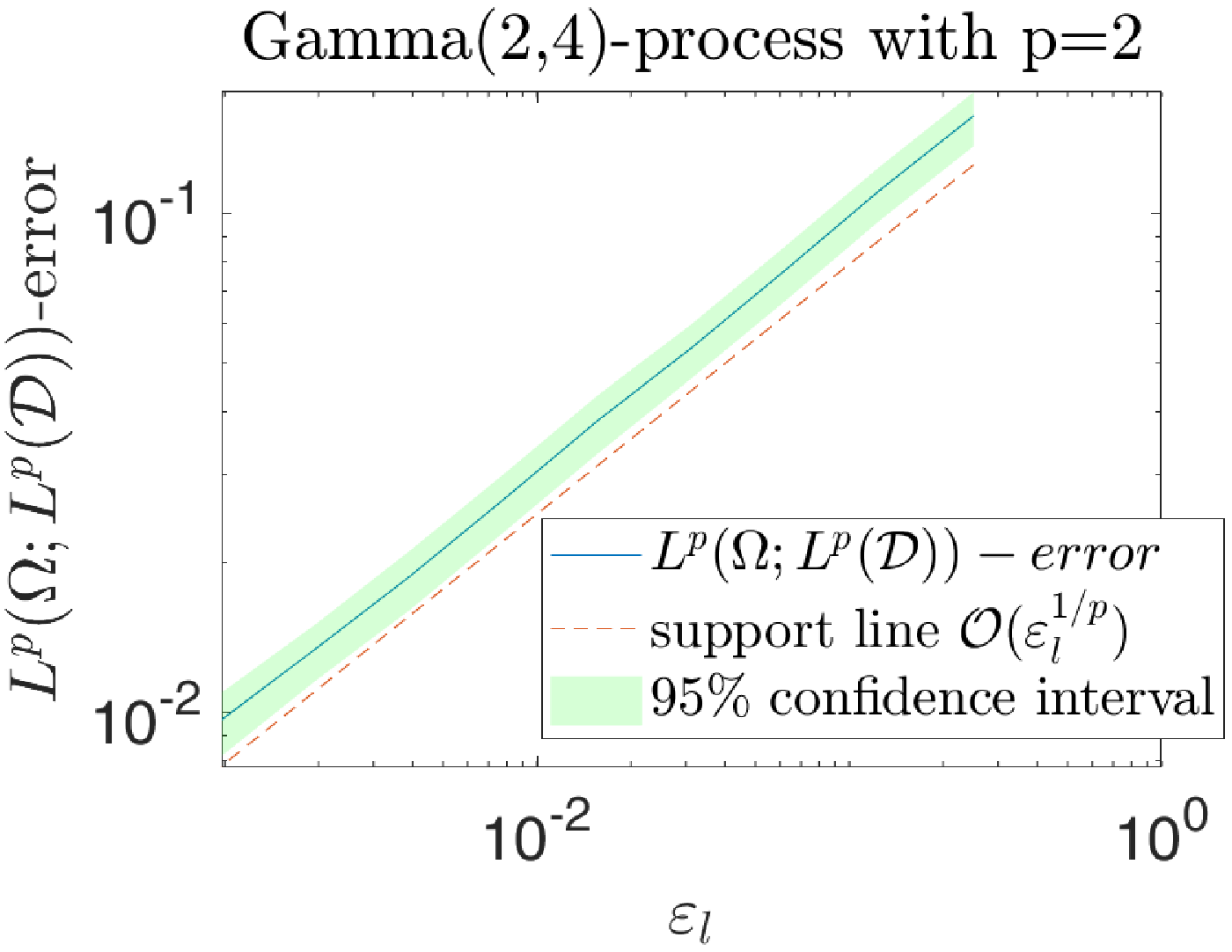}}
\subfigure{\includegraphics[scale=0.24]{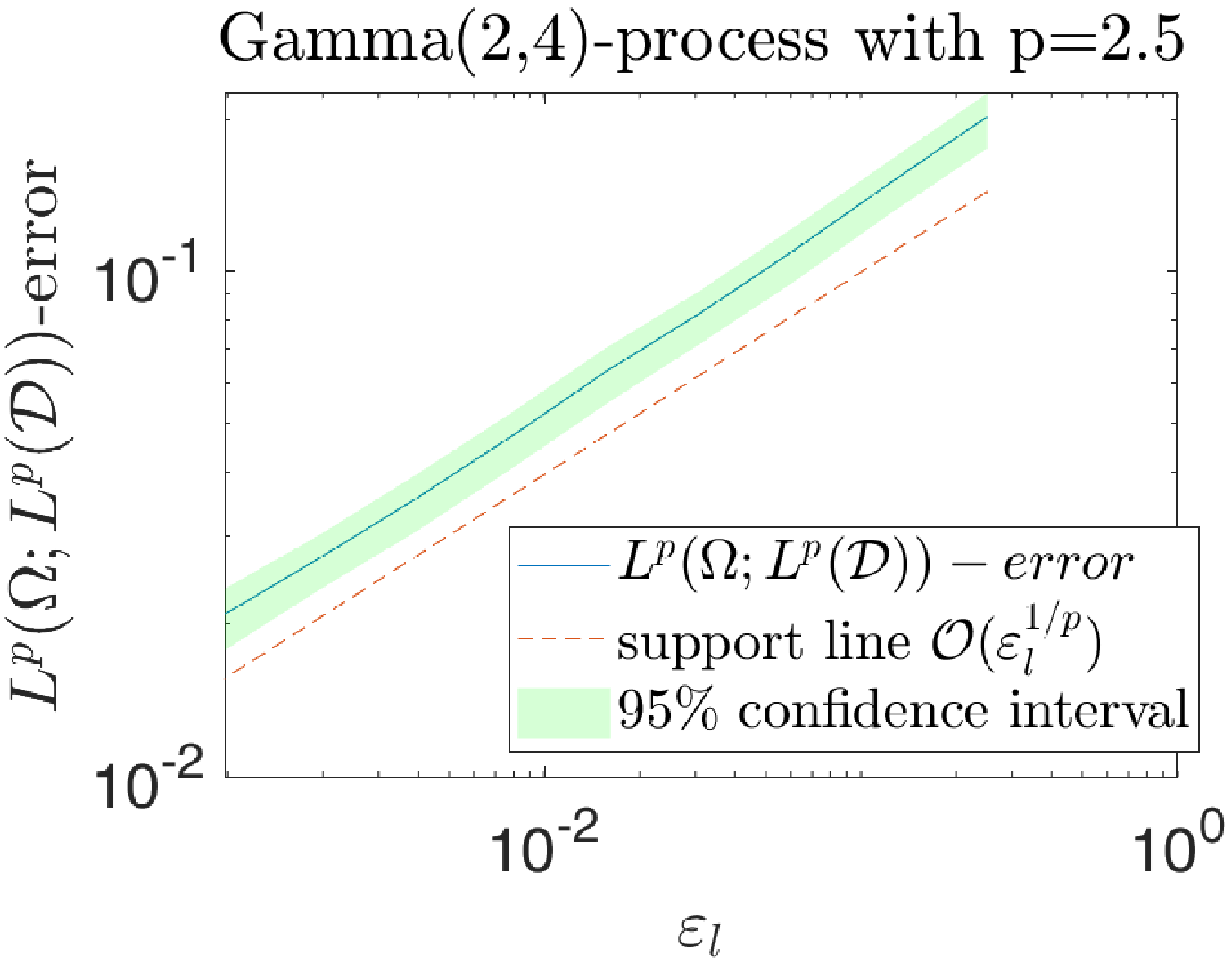}}
\subfigure{\includegraphics[scale=0.24]{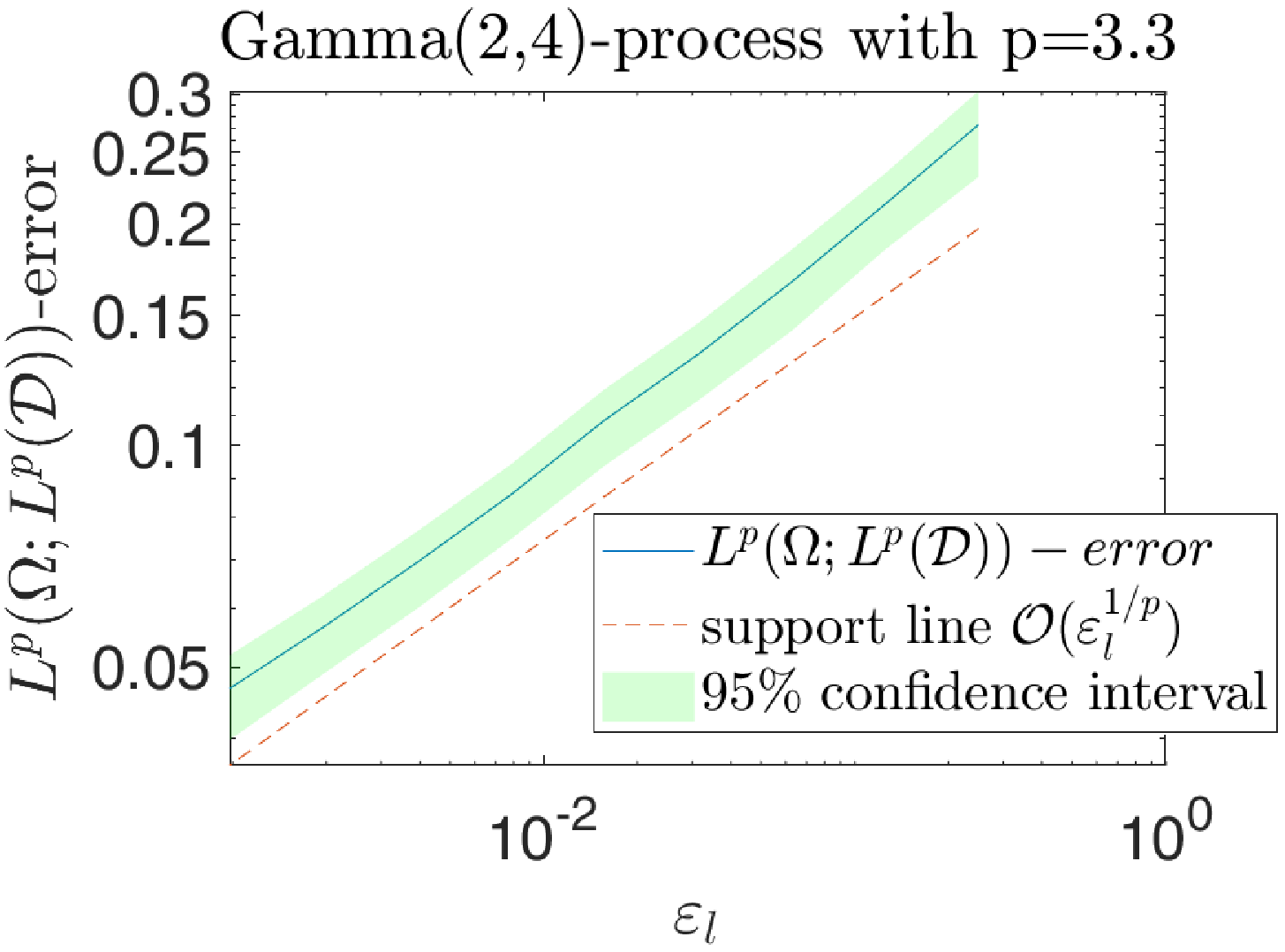}}\\
\subfigure{\includegraphics[scale=0.7]{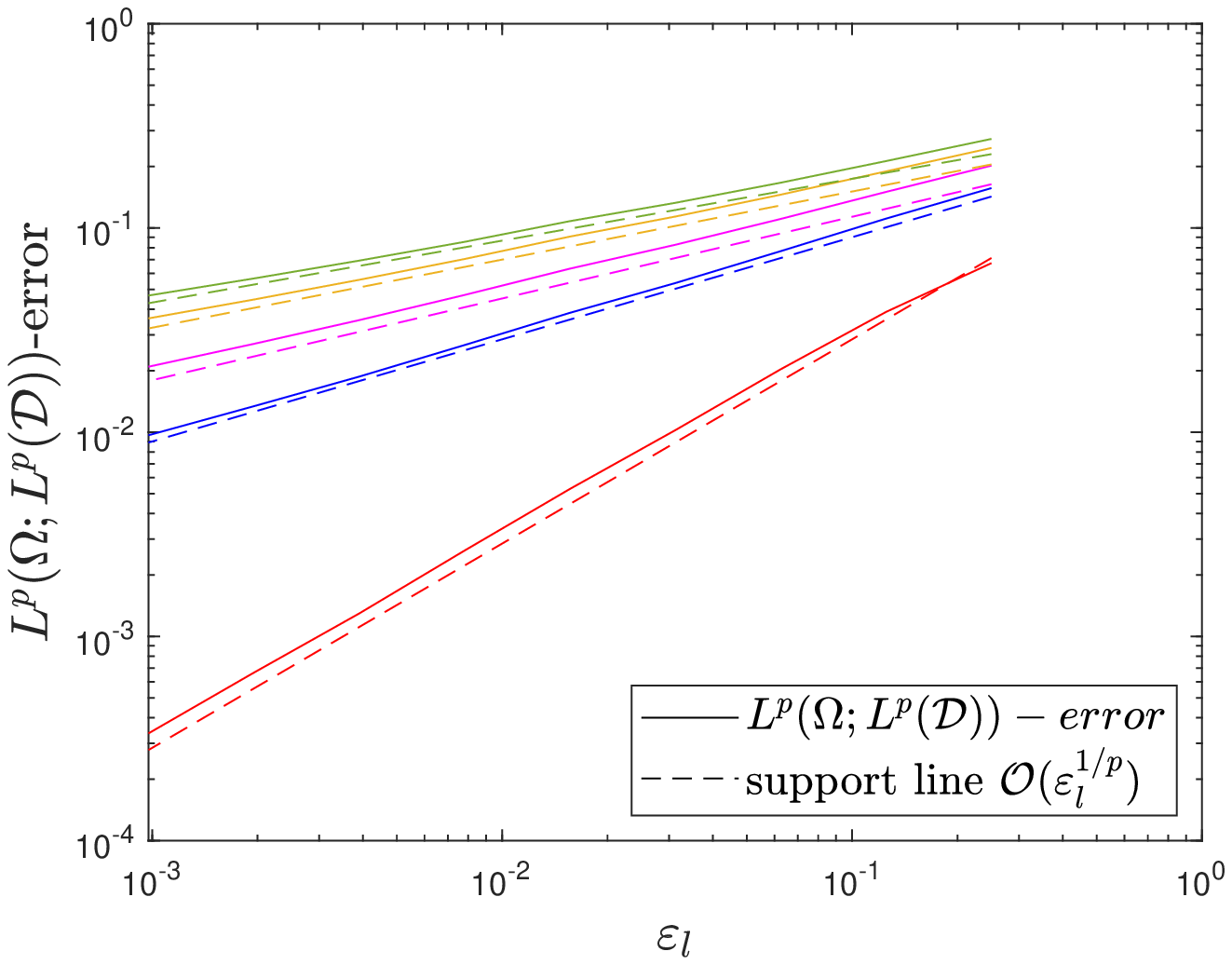}}
\caption{Estimated $L^p$ approximation errors against the discretization parameter $\varepsilon_l$  for different values of $p$ using a Gamma(2,4)-process (top) and overview for $p = \textcolor{r}{1},~\textcolor{b}{2},~\textcolor{m}{2.5},~\textcolor{dy}{3},~\textcolor{dg}{3.3}$ (bottom).}
\label{FIG:NumExApprGSLPGam}
\end{center}
\end{figure}
    
\subsection{Empirical estimation of the covariance of the GSLF}\label{subsec:NumExCov}
Lemma \ref{LE:CovGSLP} gives access to the covariance function of the GSLF. This is of interest in applications, since it is often required that a random field mimics a specific covariance structure which is determined by (real world) data. In this example, we choose specific GSLFs and spatial points to compute the corresponding covariance using Lemma \ref{LE:CovGSLP} and compare it with empirically estimated covariances using samples of the GSLF. We choose $\mathcal{D}=[0,2]^2$ and estimate the covariance with a single level Monte Carlo estimator using a growing number of samples of the field $L$ evaluated at the two considered points $\underline{x},~\underline{x}'\in\mathcal{D}$. If $M$ denotes the number of samples used in this estimation and $E_M(\underline{x},\underline{x}')$ denotes the Monte Carlo estimation of the covariance, the convergence rate of the corresponding RMSE for a growing number of samples is $-1/2$, i.e.
\begin{align*}
RMSE := \|q_L(\underline{x},\underline{x}')-E_M(\underline{x},\underline{x}')\|_{L^2(\Omega)} = \mathcal{O}(1/\sqrt{M}),~M\rightarrow \infty.
\end{align*} 
In our experiments, we use Poisson and Gamma processes and $W$ is set to be a centered GRF with squared exponential covariance function, i.e.
\begin{align*}
q_W(\underline{x},\underline{x}') = \sigma^2 \exp(-\|\underline{x}-\underline{x}'\|_2^2/r^2),~\underline{x},~\underline{x}'\in\mathcal{D},
\end{align*}
with pointwise variance $\sigma^2>0$, correlation length $r>0$ and $F=|\cdot|$. In our first experiment we choose $l$ to be a Gamma process with varying parameters. The RMSE is estimated for a growing number of samples $M$ and 100 independent MC runs. The results are shown in Figure \ref{FIG:CovEstGam}. As expected, the approximated RMSE convergences with order $\mathcal{O}(1/\sqrt{M})$ for both experiments.
\begin{figure}[ht]
\begin{center}
\subfigure{\includegraphics[scale=0.45]{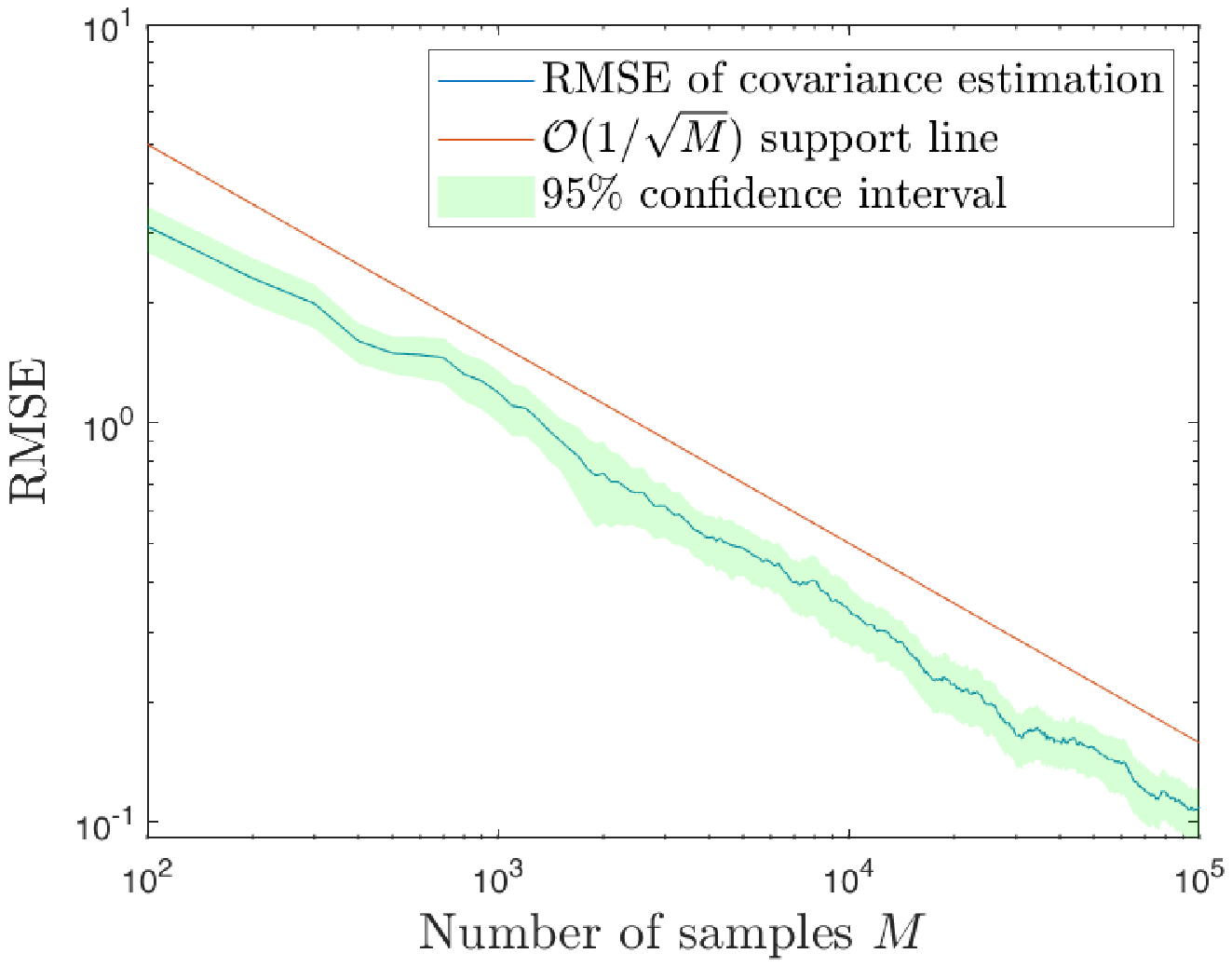}}
\subfigure{\includegraphics[scale=0.45]{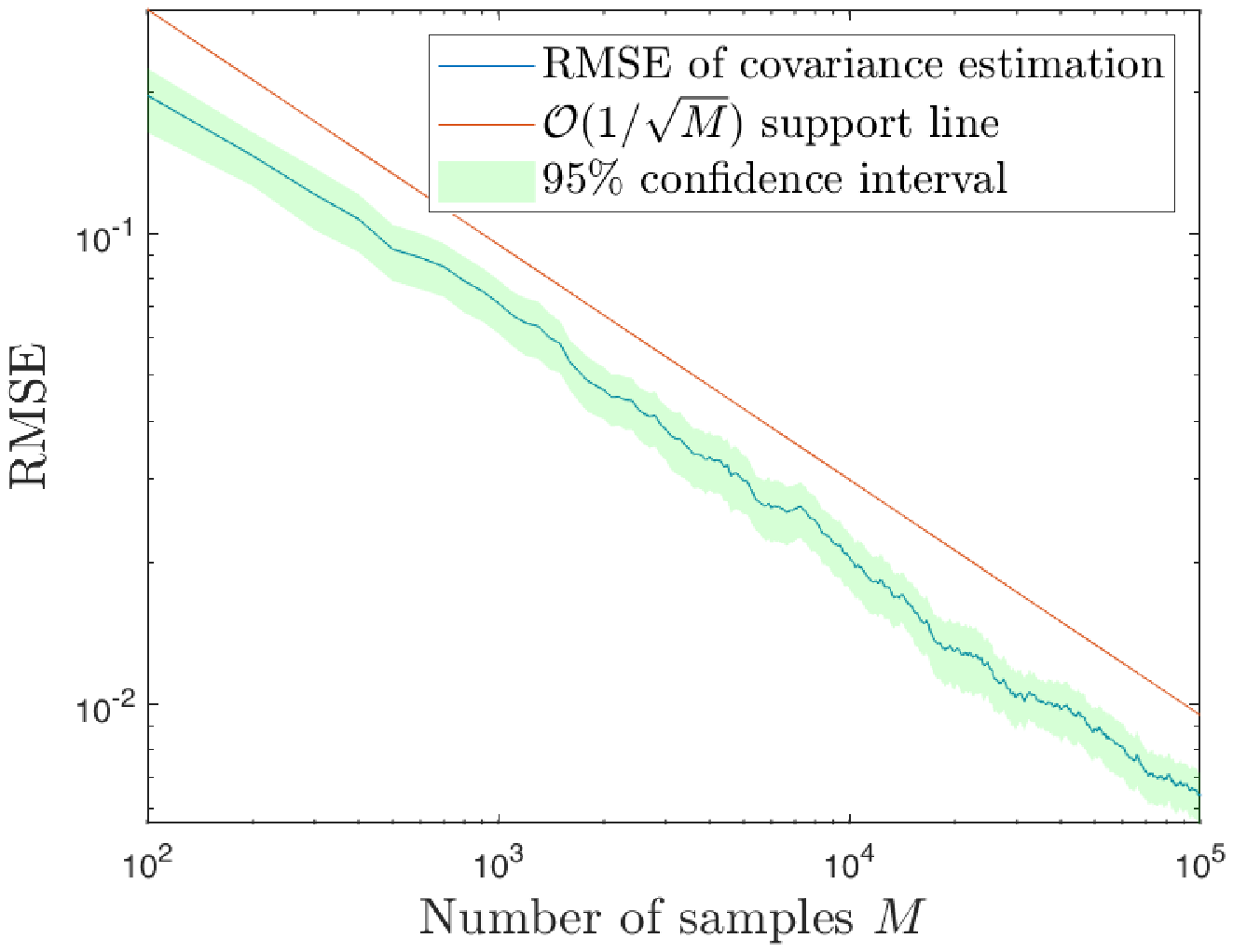}}
\caption{Convergence of RMSE of empirical covariance $q_L(\underline{x},\underline{x}')$; left: $l$ is a Gamma(4,1.5)-process, $r=1$, $\sigma^2 = 4$, $\underline{x} = (0.2,1.5)$, $\underline{x}' = (0.9,0.8)$, $q_L(\underline{x},\underline{x}') \approx 3.0265$; right: $l$ is a Gamma(5,6)-process, $r=1.2$, $\sigma^2 = 1.5^2$, $\underline{x} = (0.9,1.2)$, $\underline{x}' = (1.6,0.5)$, $q_L(\underline{x},\underline{x}') \approx 0.236$.}
\label{FIG:CovEstGam}
\end{center}
\end{figure}
In our second example we set $l$ to be a Poisson process. The results are shown in Figure \ref{FIG:CovEstPoiss}.
\begin{figure}[ht]
\begin{center}
\subfigure{\includegraphics[scale=0.45]{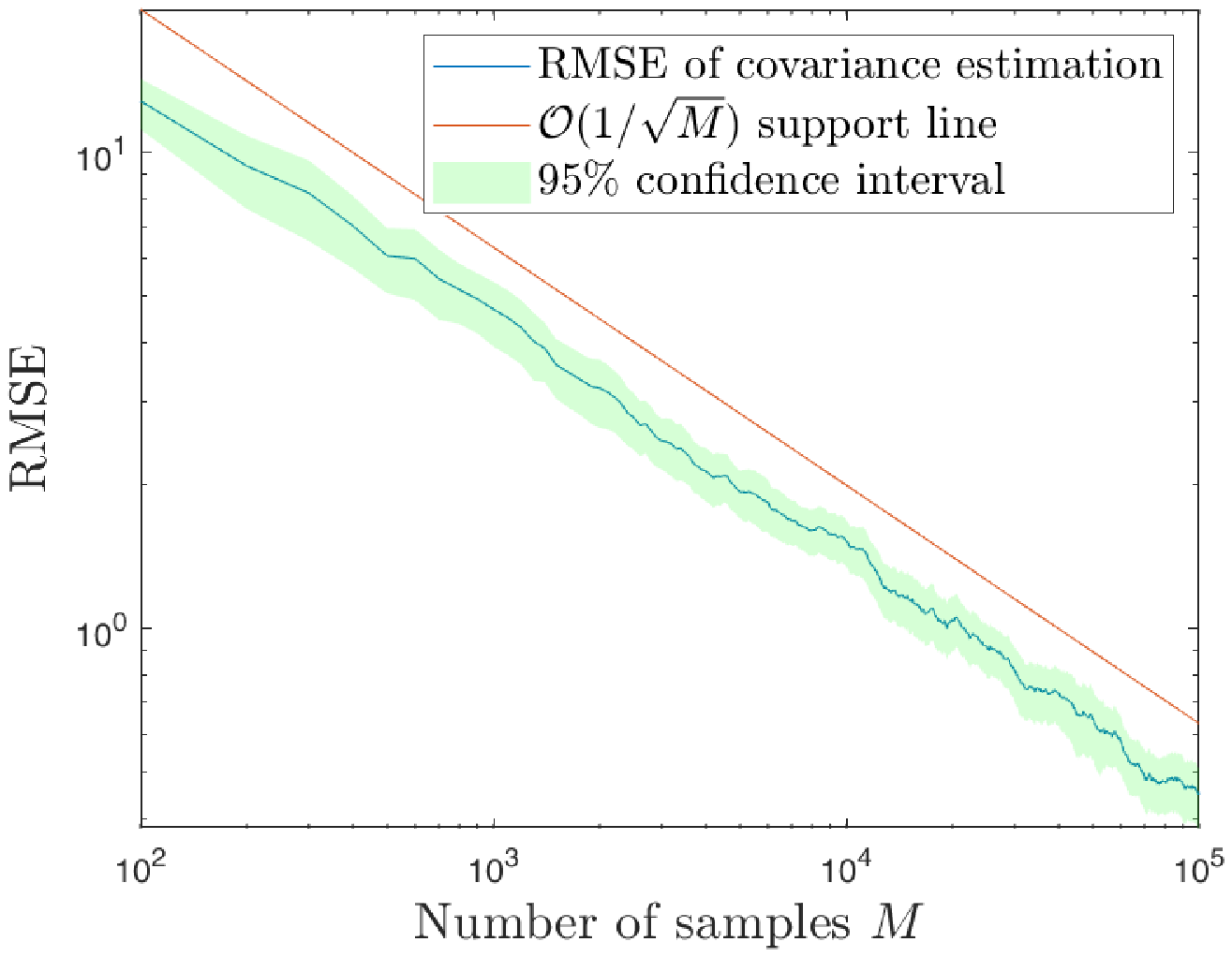}}
\subfigure{\includegraphics[scale=0.45]{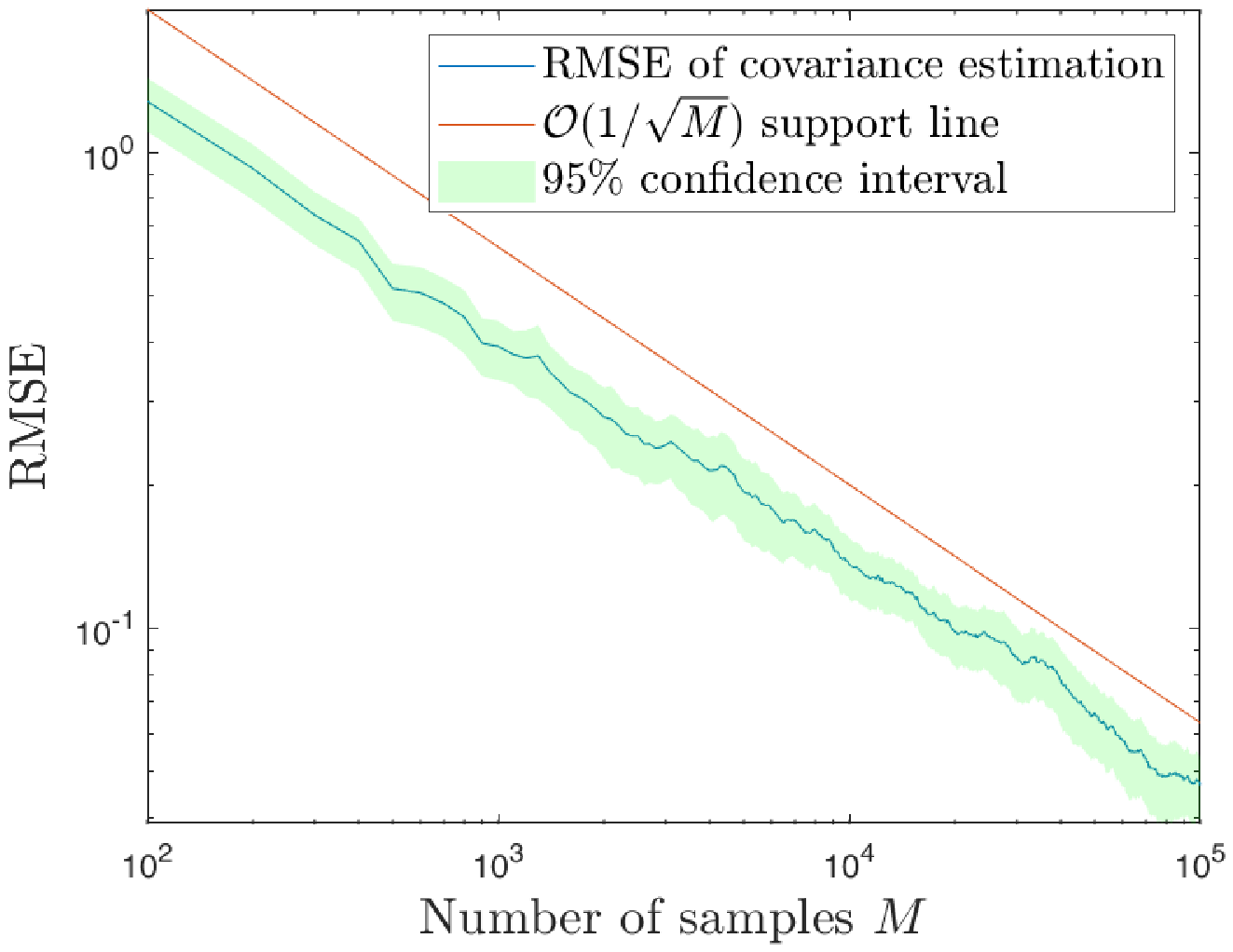}}
\caption{Convergence of RMSE of empirical covariance $q_L(\underline{x},\underline{x}')$; left: $l$ is a Poisson(8)-process, $r=0.5$, $\sigma^2 = 1.5^2$, $\underline{x}=(0.5,0.8)$, $\underline{x}' = (0.6,1.5)$, $q_L(\underline{x},\underline{x}') \approx 6.4807$; right: $l$ is a Poisson(3)-process, $r=1$, $\sigma^2 = 1$, $\underline{x}=(1.2,0.6)$, $\underline{x}' = (0.9,1.7)$, $q_L(\underline{x},\underline{x}') \approx 1.6485$.}
\label{FIG:CovEstPoiss}
\end{center}
\end{figure}
As in the previous example, we see a convergence rate of order $\mathcal{O}(1/\sqrt{M})$ of the MC estimator for the covariance in each experiment. For small sample numbers, the error values for the estimation of the covariances might seem quite high: for example, in the left plot of Figure \ref{FIG:CovEstPoiss}, using $M=100$ samples, we obtain an approximated RMSE which is approximately two times the exact value of the covariance. This seems to large at first glance. However, one has to keep in mind that the RMSE is bounded by $std(L(\underline{x})\cdot L(\underline{x}'))/\sqrt{M}$ and the standard deviation of the product of the evaluated field may become quite large. In fact, in the left hand side of Figure \ref{FIG:CovEstPoiss} we obtain $std(L(\underline{x})\cdot L(\underline{x}'))\approx 138,46$ which fits well to the observed results, being approximately $10 = \sqrt{100}$ times the approximated RMSE for $M=100$ samples. In practice, it is therefore important to keep in mind that the estimation of the covariance based on existing data might require a large number of observations.
    
\section{GSLFs in elliptic PDEs}\label{sec:GSLPellPDE}
In the previous sections, we considered theoretical properties of the GSLF and presented numerical examples to validate and visualize these results. In the last section of this paper we present an application of the GSLF in the context of PDEs. This might be interesting, for example, to model subsurface flow in uncertain heterogeneous or fractured media. In such a modeling situation the media is often modeled by a random field, which should therefore be distributionally flexible and allow for spatial discontinuities, both of which are fulfilled in the case of the GSLF. This motivates the investigation of a general elliptic PDE where the GSLF occurs in the diffusion coefficient. We introduce the considered elliptic PDE in Subsection \ref{subsec:ProblemFormEllPDE} following \cite{SGRFPDEMLMC} and \cite{SGRFPDE}. Spatial discretization methods are discussed in Subsection \ref{subsec:FEMethod} and we conclude this section with numerical experiments in Subsection \ref{subsec:NumExEllPDE}.

\subsection{Problem formulation and existence of solutions}\label{subsec:ProblemFormEllPDE}
Let $\mathcal{D}\subset \mathbb{R}^2$ be a bounded, connected Lipschitz domain\footnote{Note that the extension to dimensions $d>2$ is straightforward.}. Define $H:=L^2(\mathcal{D})$ and consider the elliptic PDE
	\begin{align}\label{EQ:EllProblem}
	-\nabla\cdot(a(\omega,\underline{x})\nabla u(\omega,\underline{x}))=f(\omega,\underline{x}) \text{ in }\Omega\times\mathcal{D},
	\end{align}
	where we impose the boundary conditions
	\begin{align}
	u(\omega,\underline{x})&=0 \text{ on } \Omega\times \Gamma_1,\label{EQ:EllProblemBCD}\\
	a(\omega,\underline{x}) \overrightarrow{n}\cdot\nabla u(\omega,\underline{x})&=g(\omega,\underline{x}) \text{ on } \Omega\times \Gamma_2.\label{EQ:EllProblemBCN}
	\end{align}
	Here, we split the boundary $\partial\mathcal{D}=\Gamma_1\overset{.}{\cup}\Gamma_2$ in two one-dimensional manifolds $\Gamma_1,~\Gamma_2$ and assume that $\Gamma_1 $ is of positive measure and that the exterior normal derivative $\overrightarrow{n}\cdot\nabla v$ on $\Gamma_2$ is well-defined for every $v\in C^1(\overline{\mathcal{D}})$, where $\overrightarrow{n}$ is the outward unit normal vector to $\Gamma_2$. The mappings $f:\Omega\times\mathcal{D}\rightarrow\mathbb{R}$ and $g:\Omega\times\Gamma_2\rightarrow\mathbb{R}$ a measurable functions and $a:\Omega\times\mathcal{D}\rightarrow\mathbb{R}$ is defined by
\begin{align}\label{EQ:DiffCoeffDefi}
a:\Omega\times\mathcal{D}\rightarrow (0,+\infty),~(\omega,x,y)\mapsto \overline{a}(x,y) + \Phi_1(W_1(x,y)) + \Phi_2(l(F(W_2(x,y)))),
\end{align}
where 
\begin{itemize}
\item $\overline{a}:\mathcal{D}\rightarrow (0,+\infty)$ is deterministic, continuous and there exist finite constants $\overline{a}_+,\overline{a}_->0$ with $\overline{a}_-\leq \overline{a}(x,y)\leq \overline{a}_+$ for $(x,y)\in\mathcal{D}$.
\item $\Phi_1,~\Phi_2:\mathbb{R}\rightarrow [0,+\infty)$ are continuous.
\item $F:\mathbb{R}\rightarrow \mathbb{R}_+$ is Lipschitz continuous and globally bounded by $C_F>0$, i.e. $F(x)<C_F, ~x\in\mathbb{R}$.
\item $W_1$ and $W_2$ are zero-mean GRFs on $\overline{\mathcal{D}}$ with $\mathbb{P}-a.s.$ continuous paths.
\item $l$ is a L\'evy process on $[0,C_F]$.
\end{itemize}
Note that we consider the case of a homogeneous Dirichlet boundary condition on $\Gamma_1$ in the theoretical analysis to simplify notation. Non-homogeneous Dirichlet boundary conditions could also be considered, since such a problem can always be formulated as a version of \eqref{EQ:EllProblem} - \eqref{EQ:EllProblemBCN} with modified source term and Neumann data (see also \cite[Remark 2.1]{AStudyOfElliptic}).
    
Next, we shortly introduce the pathwise weak solution of problem \eqref{EQ:EllProblem} - \eqref{EQ:EllProblemBCN} following \cite{SGRFPDE} and \cite{SGRFPDEMLMC}.
We denote by $H^1(\mathcal{D})$ the Sobolev space and by $T$ the trace operator  $T:H^1(\mathcal{D})\rightarrow H^{\frac{1}{2}}(\partial \mathcal{D})$ 
where $Tv=v|_{\partial \mathcal{D}}$ for $v\in C^\infty (\overline{\mathcal{D}})$ (see \cite{TraceTheoremLipDomain}). Further, we introduce the solution space $V\subset H^1(\mathcal{D})$ by
\begin{align*}
V:=\{v\in H^1(\mathcal{D})~|~Tv|_{\Gamma_1}=0\},
\end{align*}
where we take over the standard Sobolev norm, i.e. $\|\cdot\|_V:=\|\cdot\|_{H^1(\mathcal{D})}$. 	We identify $H$ with its dual space $H'$ and work on the Gelfand triplet $V\subset H\simeq H'\subset V'$. Multiplying the left hand side of Equation \eqref{EQ:EllProblem} by a test function $v\in V$ and integrating by parts (see e.g. \cite[Section 6.3]{ValliACompactCourseOnLinPDEs}) we obtain the following pathwise weak formulation of the problem: For fixed $\omega\in\Omega$ and given mappings $f(\omega,\cdot)\in V'$ and $g(\omega,\cdot)\in H^{-\frac{1}{2}}(\Gamma_2)$, find $u(\omega,\cdot)\in V$ such that 
\begin{align}\label{EQ:WeakFormProblem}
B_{a(\omega)}(u(\omega,\cdot),v) = F_\omega(v)
\end{align}
for all $v\in V$. The function $u(\omega,\cdot)$ is then called pathwise weak solution to problem \eqref{EQ:EllProblem} - \eqref{EQ:EllProblemBCN} and the bilinear form $B_{a(\omega)}$ and the linear operator $F_\omega$ are defined by
\begin{align*}
B_{a(\omega)}:V\times V\rightarrow \mathbb{R}, ~(u,v)\mapsto \int_{\mathcal{D}}a(\omega,\underline{x})\nabla u(\underline{x})\cdot \nabla v(\underline{x})d\underline{x},
\end{align*}
and
\begin{align*}
F_\omega:V\rightarrow\mathbb{R}, ~v\mapsto \int_\mathcal{D}f(\omega,\underline{x})v(\underline{x})d\underline{x} + \int_{\Gamma_2}g(\omega,\underline{x})[Tv](\underline{x})d\underline{x},
\end{align*}	
where the integrals in $F_\omega$ are understood as the duality pairings:
\begin{align*}
\int_\mathcal{D}f(\omega,\underline{x})v(\underline{x})d\underline{x} = \prescript{}{V'}{\langle}f(\omega,\cdot),v\rangle_V \text{ and }
\int_{\Gamma_2}g(\omega,\underline{x})[Tv](\underline{x})d\underline{x} = \prescript{}{H^{-\frac{1}{2}}(\Gamma_2)}{\langle	}g(\omega,\cdot),Tv\rangle_{H^\frac{1}{2}(\Gamma_2)},
\end{align*}
for $v\in V$.

\begin{rem}\label{rem:CadlagFunctionsBounded}
Note that any real-valued, càdlàg function $f:[a,b]\rightarrow \mathbb{R}$ on a compact interval $[a,b]$, with $a<b$, is bounded. Otherwise, one could find a sequence $(x_n,~n\in\mathbb{N})\subset[a,b]$ such that $|f(x_n)|>n$ for all $n\in \mathbb{N}$. In this case, since $[a,b]$ is compact, there exists a subsequence $(x_{n_k},~k\in\mathbb{N})\subset[a,b]$ with a limit in $[a,b]$, i.e. $x_{n_k}\rightarrow x^*\in[a,b]$ for $k\rightarrow \infty$. Since $f$ is right-continuous in $x^*$ and $lim_{x\nearrow x^*}f(x)$ exists, $f$ is bounded in a neighborhood of $x^*$, i.e. there exists $\delta,C>0$ such that $|f(x)|\leq C$ for $x\in[a,b]$ with $|x-x^*|<\delta$, which contradicts the fact that $|f(x_{n_k})|>n_k$ since $|x_{n_k}-x^*|<\delta$ and $n_k>C$ both hold for $k$ large enough.
\end{rem}

The diffusion coefficient $a$ is jointly measurable by construction (see Remark \ref{rem:GSLPmeasurable}) and, for any fixed $\omega\in\Omega$, it holds $a(\omega,x,y)\geq \overline{a}_->0$ for all $(x,y)\in\mathcal{D}$ and 
\begin{align*}
%a_+:=\underset{(x,y)\in\overline{\mathcal{D}}}{ess\,sup}\,
a(\omega,x,y) \leq \overline{a}_+ + \underset{(x,y)\in\overline{\mathcal{D}}}{sup}\Phi_1(W_1(\omega,x,y)) + \underset{z\in[0,C_F]}{sup}\Phi_2(l(\omega,z)) < +\infty,
\end{align*}
for all $(x,y)\in\mathcal{D}$, where the finiteness follows by the continuity of $\Phi_1,~\Phi_2,~W_1$ and Remark \ref{rem:CadlagFunctionsBounded}.

It follows by a pathwise application of the Lax-Milgram lemma that the elliptic model problem ~\eqref{EQ:EllProblem} - \eqref{EQ:EllProblemBCN} with the diffusion coefficient $a$ has a unique, measurable pathwise weak solution. 
\begin{theorem}(see \cite[Theorem 3.7, Remark 2.4]{SGRFPDE} and \cite[Theorem 2.5]{AStudyOfElliptic})\label{TH:ExistenceOfSolutionSubord}
Let $f\in L^q(\Omega;H),$ $g\in L^q(\Omega;L^2(\Gamma_2))$ for some $q\in [1,+\infty)$. There exists a unique pathwise weak solution $u(\omega,\cdot)\in V$ to problem~\eqref{EQ:EllProblem} - \eqref{EQ:EllProblemBCN} with diffusion coefficient \eqref{EQ:DiffCoeffDefi} for $\mathbb{P}$-almost every $\omega\in\Omega$. Furthermore, $u\in L^r(\Omega;V)$ for all $r\in[1,q)$ and 
\begin{align*}
\|u\|_{L^r(\Omega;V)}\leq C(\overline{a}_-,\mathcal{D})(\|f\|_{L^q(\Omega;H)} + \|g\|_{L^q(\Omega;L^2(\Gamma_2))}),
\end{align*}
where $C(\overline{a}_-,\mathcal{D})>0$ is a constant depending only on the indicated parameters. 
\end{theorem}	
 
\subsection{Spatial discretization of the elliptic PDE}\label{subsec:FEMethod} 
In the following, we briefly describe numerical methods to approximate the  pathwise solution to the random elliptic PDE, partially following \cite[Section 6]{SGRFPDE}. Our goal is to approximate the weak solution $u$ to problem~\eqref{EQ:EllProblem} - \eqref{EQ:EllProblemBCN} with diffusion coefficient $a$ given by Equation~\eqref{EQ:DiffCoeffDefi}. Therefore, for almost all $\omega\in\Omega$, we aim to approximate the function $u(\omega,\cdot)\in V$ such that
	\begin{align}\label{EQ:VariationalProblemApprSol}
	B_{a(\omega)} (u(\omega,\cdot), v) &:=\int_{\mathcal{D}}a(\omega,\underline{x})\nabla	u(\omega,\underline{x})\cdot\nabla v(\underline{x})d\underline{x} \notag \\ 
	&=\int_\mathcal{D}f(\omega,\underline{x})v(\underline{x})d\underline{x} + \int_{\Gamma_2} g(\omega,\underline{x})[Tv](\underline{x})d\underline{x}=:F_\omega(v),
	\end{align}
	for every $v\in V$. We compute an approximation of the solution using a standard Galerkin approach with linear basis functions. Therefore, assume a sequence of finite-element subspaces is given, which is denoted by $\mathcal{V}=(V_\ell,~\ell\in\mathbb{N}_0)$, where $V_\ell\subset	V$ are subspaces with growing dimension $\dim(V_\ell)=d_\ell$. We denote by $(h_\ell,~\ell\in\mathbb{N}_0)$ the corresponding refinement sizes with $h_\ell \rightarrow 0 $, for $\ell\rightarrow \infty$. Let $\ell\in\mathbb{N}_0$ be fixed and denote by $\{v_1^{(\ell)},\dots,v_{d_\ell}^{(\ell)}\}$ a basis of $V_\ell$. The (pathwise) discrete version of problem \eqref{EQ:VariationalProblemApprSol} reads: Find $u_\ell(\omega,\cdot)\in V_\ell$ such that \begin{align*}
	B_{a(\omega)}(u_\ell(\omega,\cdot),v_i^{(\ell)})=F_\omega(v_i^{(\ell)}) \text{ for all } i=1,\dots,d_\ell.
\end{align*}		
If we expand the solution $u_\ell(\omega,\cdot)$ with respect to the basis $\{v_1^{(\ell)},\dots,v_{d_\ell}^{(\ell)}\}$, we end up with the representation
\begin{align*}
u_\ell(\omega,\cdot)=\sum_{i=1}^{d_\ell} c_iv_i^{(\ell)},
\end{align*}
where the coefficient vector $\textbf{c}=(c_1,\dots,c_{d_\ell})^T\in\mathbb{R}^{d_\ell}$ is determined by the linear system of equations 
\begin{align*}
\textbf{B}(\omega)\textbf{c}=\textbf{F}(\omega),
\end{align*}
with stochastic stiffness matrix $\textbf{B}(\omega)_{i,j}=B_{a(\omega)}(v_i^{(\ell)},v_j^{(\ell)})$ and load vector $\mathbf{F}(\omega)_i=F_\omega (v_i^{(\ell)})$ for $i,j=1,\dots,d_\ell$.
\subsubsection{Standard linear finite elements}
	Let $(\mathcal{K}_\ell,~\ell\in \mathbb{N}_0)$ be a sequence of admissible triangulations of the domain $\mathcal{D}$ (cf. \cite[Definition 8.36]{EllipticDifferentialEquations}) and denote by $\theta_\ell>0$ the minimum interior angle of all triangles in $\mathcal{K}_\ell$, where we assume $\theta_\ell\geq \theta>0$ for a positive constant $\theta$. For $\ell \in \mathbb{N}_0$, we denote by $h_\ell:=\underset{K\in\mathcal{K}_\ell}{\max}\, diam(K)$ the maximum diameter of the triangulation $\mathcal{K}_\ell$ and define the finite dimensional subspaces by $V_\ell:=\{v\in V~|~v|_K\in\mathcal{P}_1,K\in \mathcal{K}_\ell\}$,	where $\mathcal{P}_1$ denotes the space of all polynomials up to degree one. If we assume that there exists a positive regularity parameter $\kappa_a>0$ such that for $\mathbb{P}-$almost all $\omega\in\Omega$ it holds $u(\omega,\cdot)\in H^{1+\kappa_a}(\mathcal{D})$ and $\|u\|_{L^2(\Omega;H^{1+\kappa_a}(\mathcal{D}))}\leq C_u$ for some constant $C_u>\infty$, we immediately obtain the following bound by C\'ea's lemma (see \cite[Section 4]{AStudyOfElliptic}, \cite[Section 6]{SGRFPDE}, \cite[Chapter 8]{EllipticDifferentialEquations})
	\begin{align}\label{EQ:ConvFEM}
	\|u-u_\ell\|_{L^2(\Omega;V)}\leq C\, h_\ell^{\min(\kappa_a,1)},
	\end{align}	
	for some constant $C$ which does not depend on $h_\ell$.
For general deterministic elliptic interface problems, one obtains a discretization error of order $\kappa_a < 1$ and, hence, one cannot expect the full order of convergence $\kappa_a=1$ in general for standard triangulations of the domain (see \cite{TheFiniteElementMethodForELlipticEquationsWithDiscontinuousCoefficients} and~\cite{AStudyOfElliptic}). Therefore, we present one possible approach to enhance the performance of the FE method for the considered elliptic PDE in Subsection \ref{subsubsec:AFEM}. We point out that it is not possible to derive optimal rates $\kappa_a>0$ such that \eqref{EQ:ConvFEM} holds for our general random diffusion coefficient (see also \cite{SGRFPDE}, \cite{AStudyOfElliptic}). However, we investigate the existence of such a constant numerically in Section \ref{subsec:NumExEllPDE}. We close this subsection with a remark on the practical simulation of the GRFs $W_1,~W_2$ and the Lévy process $l$ in the diffusion coefficient \eqref{EQ:DiffCoeffDefi}.

\begin{rem}
It is in general not possible to draw exact samples of the Lévy process and the GRFs in the diffusion coefficient \eqref{EQ:DiffCoeffDefi}. In practice, one has to use appropriate approximations $l^{(\varepsilon_l)}\approx l$ of the Lévy process and $W_1^{N}\approx W_1,~W_2^{N}\approx W_2$ of the GRFs with approximation parameters $\varepsilon_l >0$ and $N\in\mathbb{N}$, as introduced in Section \ref{sec:approximation}. In the context of FE approximations of the PDE \eqref{EQ:EllProblem} - \eqref{EQ:EllProblemBCN} with diffusion coefficient \eqref{EQ:DiffCoeffDefi}, the question arises, how to choose these approximation parameters in practice, given a specific choice of the FE approximation parameter $h_\ell$ in \eqref{EQ:ConvFEM}. Obviously, the choice of $\varepsilon_l$ and $N$ should depend on the FE parameter $h_\ell$, since a higher resolution of the FE approximation will be worthless if the approximation of the diffusion coefficient is poor. 

In \cite{SGRFPDE}, the authors considered the PDE \eqref{EQ:EllProblem} - \eqref{EQ:EllProblemBCN} with a different diffusion coefficient $a$. The coefficient considered in the mentioned paper also incorporates GRFs and Lévy processes, which in turn have to be approximated in practice, leading to an approximation $\tilde{a}\approx a$. The authors derived a rule on the choice of the approximation parameters, such that the error contributions from the approximation of the diffusion coefficient and the FE discretization are equilibrated (see \cite[Section 7.1]{SGRFPDE}). We want to emphasize that this result is essentially based on the quantification of the approximation error $\tilde{a}-a$ of the corresponding diffusion coefficient in an $L^p(\Omega,L^p(\mathcal{D}))$-norm, for some $p\geq 1$ (see \cite[Theorem 4.8 and Theorem 5.11]{SGRFPDE}). 

In Theorem \ref{TH:GSLPAPPR}, we derived an error bound on the approximation error of the GSLF in the $L^p(\Omega,L^p(\mathcal{D}))$-norm under Assumption \ref{ASS:CutProblemEigenvalues}. A corresponding error bound on the approximation of the diffusion coefficient defined in \eqref{EQ:DiffCoeffDefi} immediately follows under mild assumptions on $\Phi_1$ and $\Phi_2$. Therefore, following exactly the same arguments as in \cite{SGRFPDE} together with Theorem \ref{TH:GSLPAPPR}, we obtain the following rule for the practical choice of the approximation parameters $\varepsilon_l$ and $N$ such that the overall error contributions from the approximation of the GRFs, the Lévy process and the FE discretization are equilibrated and the error is dominated by the FE refinement parameter $h_\ell$: 
For $\ell\in\mathbb{N}$, choose $\varepsilon_l$ and $N$ such that
\begin{align*}
\varepsilon_l \simeq h_\ell^{2\kappa_a}\text{ and } R(N) \simeq h_\ell^{\frac{2\kappa_a}{\delta}}.
\end{align*}
For example, if we approximate the GRF by the KLE approach (see Subsection \ref{subsec:ApprGRF}), one should choose the cut-off index such that $N\simeq h_\ell^{-\frac{4\kappa_a}{\beta\delta}}$ with $\beta$ from Assumption \ref{ASS:CutProblemEigenvalues} \textit{i}.
\end{rem}

\subsubsection{Adaptive finite elements}\label{subsubsec:AFEM}
As we pointed out in the last section, one cannot expect full order convergence ($\kappa_a=1$) for the FE method with linear elements in the considered elliptic problem due to the discontinuities in the diffusion coefficient. One common way to improve the FE method is to use triangulations which are adapted to the jump discontinuities in the sense that the spatial jumps lie on edges of elements of the triangulation. This approach leads to sample-dependent triangulations which has been proven to enhance the performance of the FE method significantly compared to the use of standard triangulations, which are not adjusted to the jumps (see for example \cite{AStudyOfElliptic} and \cite{SGRFPDE} and the references therein). In the cited papers, the jump locations are known and the jump geometries allow for an (almost) exact alignment of the triangulation to the interfaces due to their specific geometrical properties. This is, however, not the case for the diffusion coefficient considered in the current paper, where the spatial jump positions are not known explicitly, nor is it possible to align the triangulation exactly in the described sense due to the complex jump geometries. 

Another possible approach to improve the FE method are adaptive Finite Elements (see e.g. \cite{MR1442375} and \cite{MR2115986} and the references therein). The idea is to identify triangles with a high contribution to the overall error of the FE approximation, which are then refined. For a given FE approximation $u_\ell \in V_\ell \subset V$ of the solution $u\in V$, the discretization error $u-u_\ell$ is estimated in terms of the approximated solution $u_\ell$ and no information about the true solution $u$ is needed. This may be achieved by the use of the Galerkin orthogonality and partial integration. For a given triangulation $\mathcal{T} = \{K,~K\subset\mathcal{D} \}$ with corresponding FE approximation $u_\ell$, the approximation error of the FE approximation to the considered elliptic problem \eqref{EQ:EllProblem} - \eqref{EQ:EllProblemBCN} may be estimated by
\begin{align*}
err_\mathcal{T}(u_\ell, a, f,g) = \left(\sum_{K\in\mathcal{T}} err_\mathcal{T}(u_\ell,a,f,g,K)\right)^{\frac{1}{2}}.
\end{align*}
Here, $err_\mathcal{T}(u_\ell,a,f,g,K)$ corresponds the error contribution of the approximation on the element $K\in\mathcal{T}$, which may be approximated by
\begin{align}\label{EQ:ErrContributionTriangle}
err_{\mathcal{T}}(u_\ell, a, f,g, K) &= h_K^2\|f\|_{L^2(K)}^2 +  \sum_{e\in \partial K \cap \Gamma_2} h_e \|g-\overrightarrow{n}_K\cdot a\,\nabla u_\ell{\big|_K}\|_{L^2(e)}^2 \notag\\
&+\sum_{e\in\partial K\setminus(\Gamma_2\cup \Gamma_1)} h_e\|\overrightarrow{n}_K\cdot a\,\nabla u_\ell{\big|_K}+ \overrightarrow{n}_{J(K,e)}\cdot a\,\nabla u_\ell{\big|_{J(K,e)}}\|_{L^2(e)}^2  ,
\end{align}
where we denote by $e\in\partial K$ an edge of the triangle $K\in\mathcal{T}$ and ${J(K,e)}\in\mathcal{T}$ is the unique triangle which is the direct neighbor of the triangle $K$ along the edge $e$. Further, $h_K$ is the diameter of the triangle $K$ and $h_e$ is the length of the edge $e$ (see \cite{MR2115986} and \cite{MR1771781}). The triangle-based estimated error obtained by Equation~\eqref{EQ:ErrContributionTriangle} allows for an identification of the triangles with the largest contribution to the overall error. This may then be used to perform a local mesh refinement, which usually consists in the refinement of triangles with high error contribution. One common strategy is to start with an initial triangulation, compute the error contribution of each triangle according to Equation~\eqref{EQ:ErrContributionTriangle} and refine all triangles which have an approximated error contribution which is at least $50\%$ of the error contribution with the largest approximated error (see, e.g.,  \cite[Section 5]{MR1284252}).
\begin{example}
Consider the PDE problem \eqref{EQ:EllProblem} - \eqref{EQ:EllProblemBCN} with homogeneous Dirichlet boundary conditions on $\mathcal{D}=(0,1)^2$ and $f\equiv 10$. The diffusion coefficient $a$ as defined in Equation~\eqref{EQ:DiffCoeffDefi} is given by 
\begin{align*}
a(x,y) = 0.1 + l(min\{|W_1(x,y)|,20\}),
\end{align*}
where $l$ is a Poisson(2)-process and $W_1$ is a centered GRF with squared exponential covariance function. In order to illustrate the adaptive mesh generation decribed above, we consider 3 samples of the diffusion coefficient and compute the adaptive meshes, where we use an initial mesh with FE refinement parameter $h=0.025$ and refine all triangles which have an estimated error contribution exceeding $50\%$ of the largest estimated error. 
\begin{figure}[ht]
	\centering
	\subfigure{\includegraphics[scale=0.45]{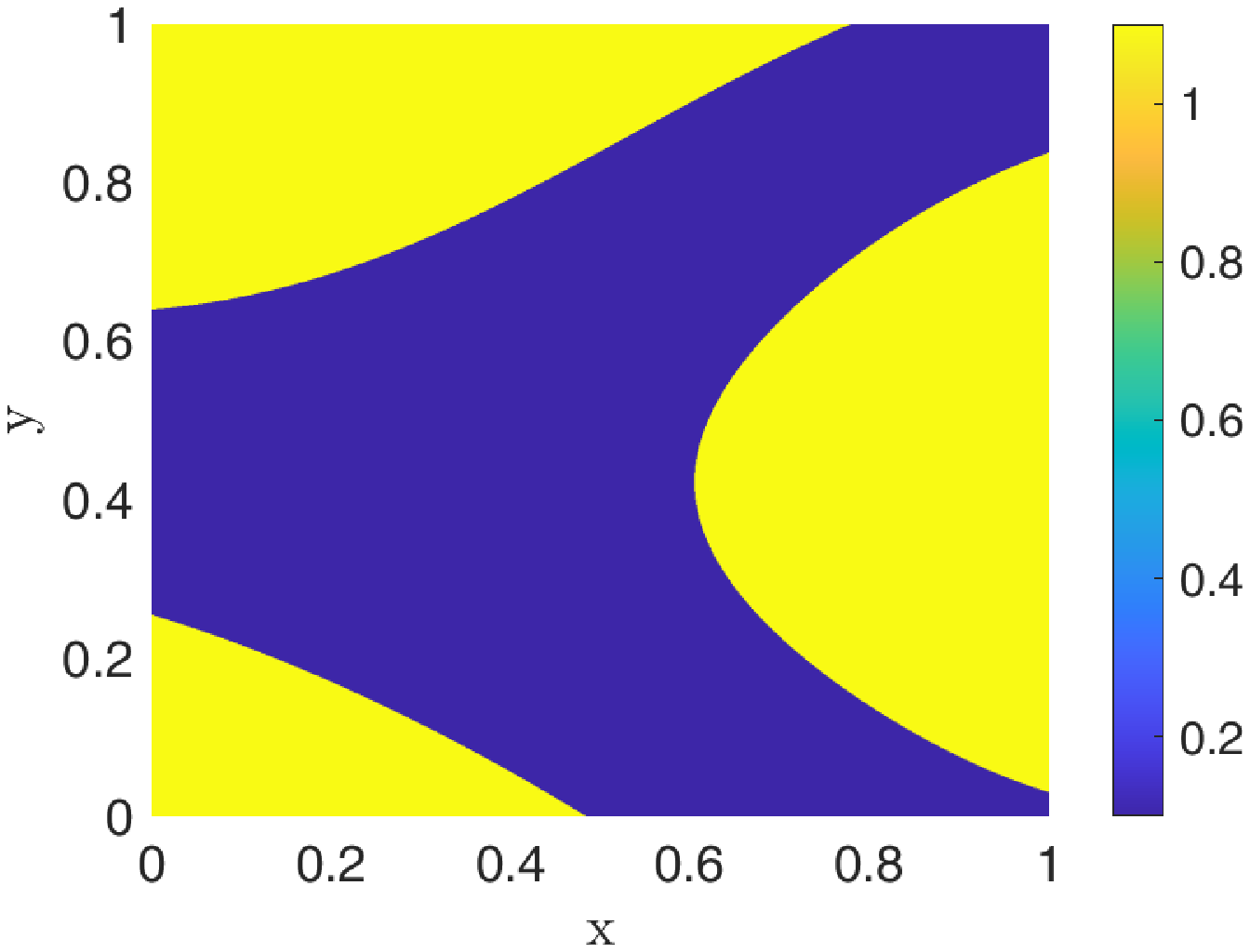}}
	\subfigure{\includegraphics[scale=0.45]{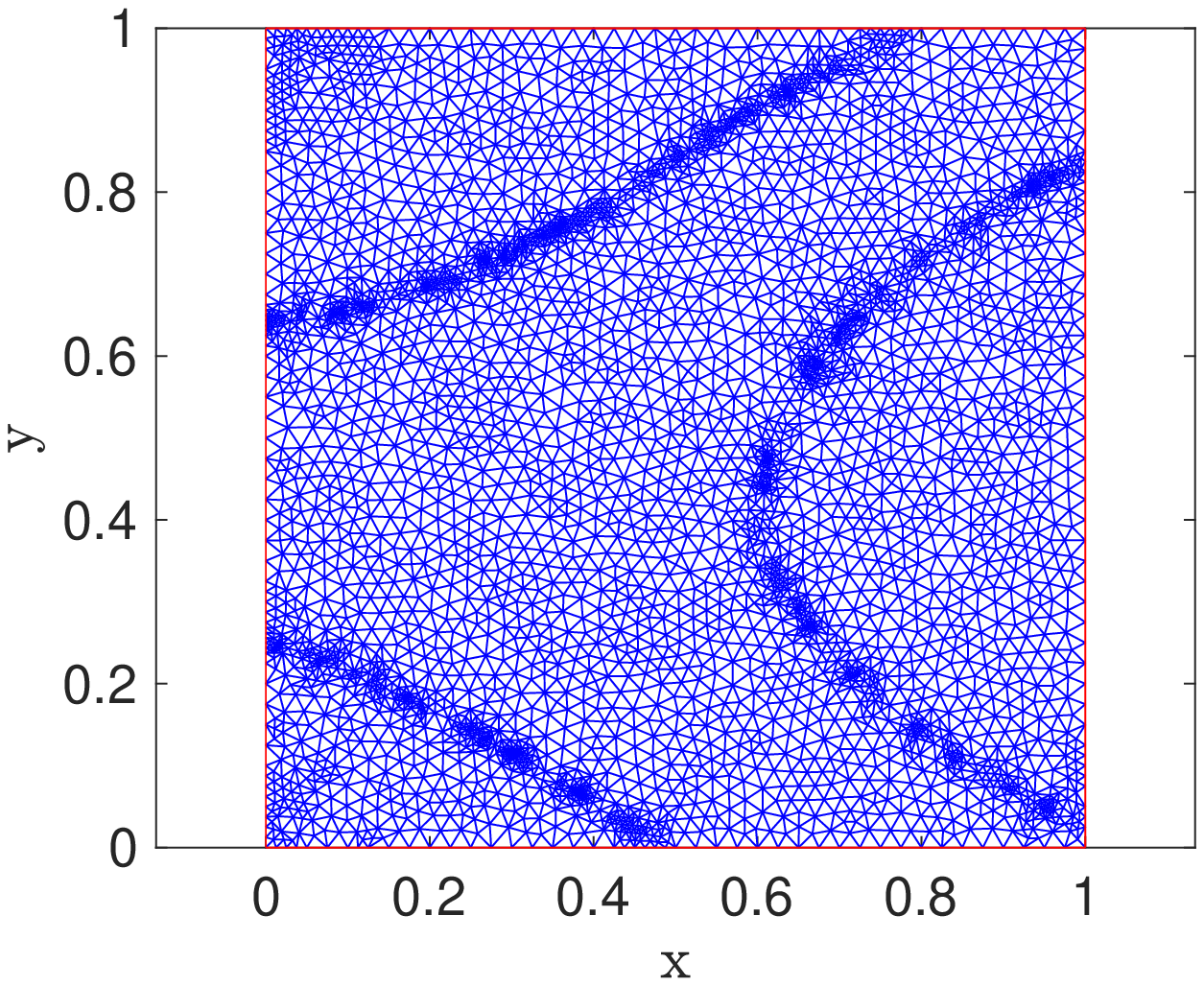}}\\
	\subfigure{\includegraphics[scale=0.45]{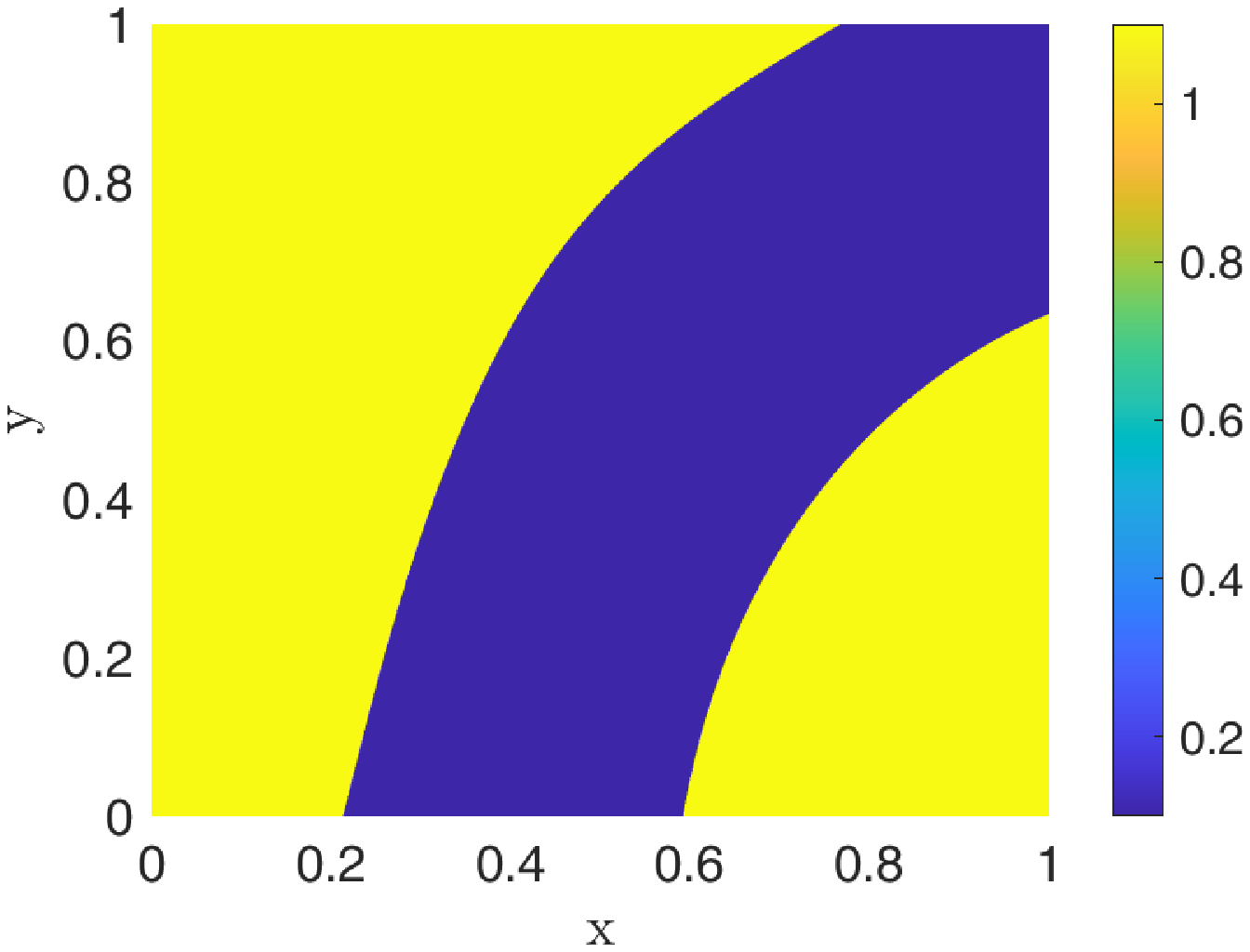}}
	\subfigure{\includegraphics[scale=0.45]{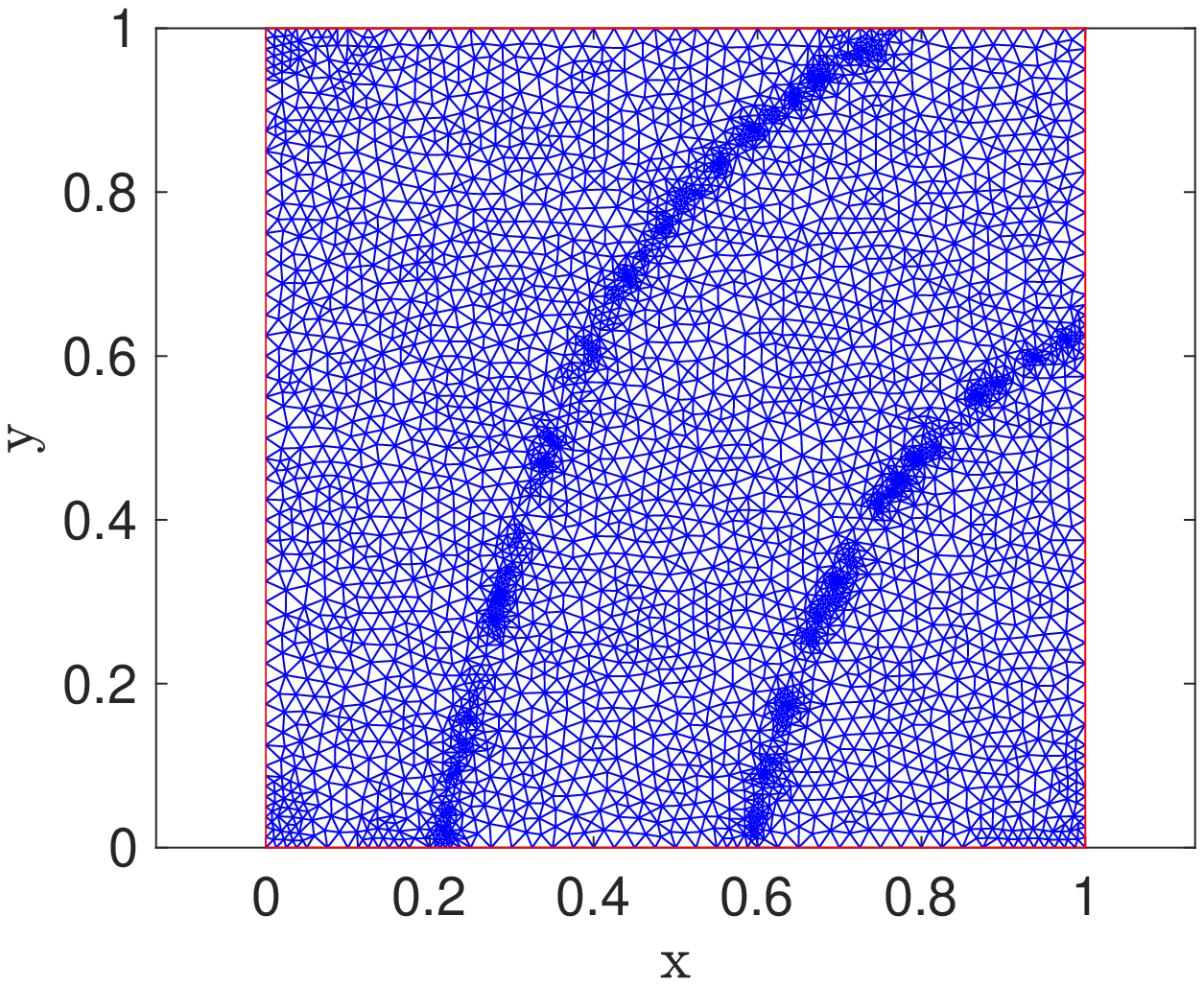}}\\
	\subfigure{\includegraphics[scale=0.45]{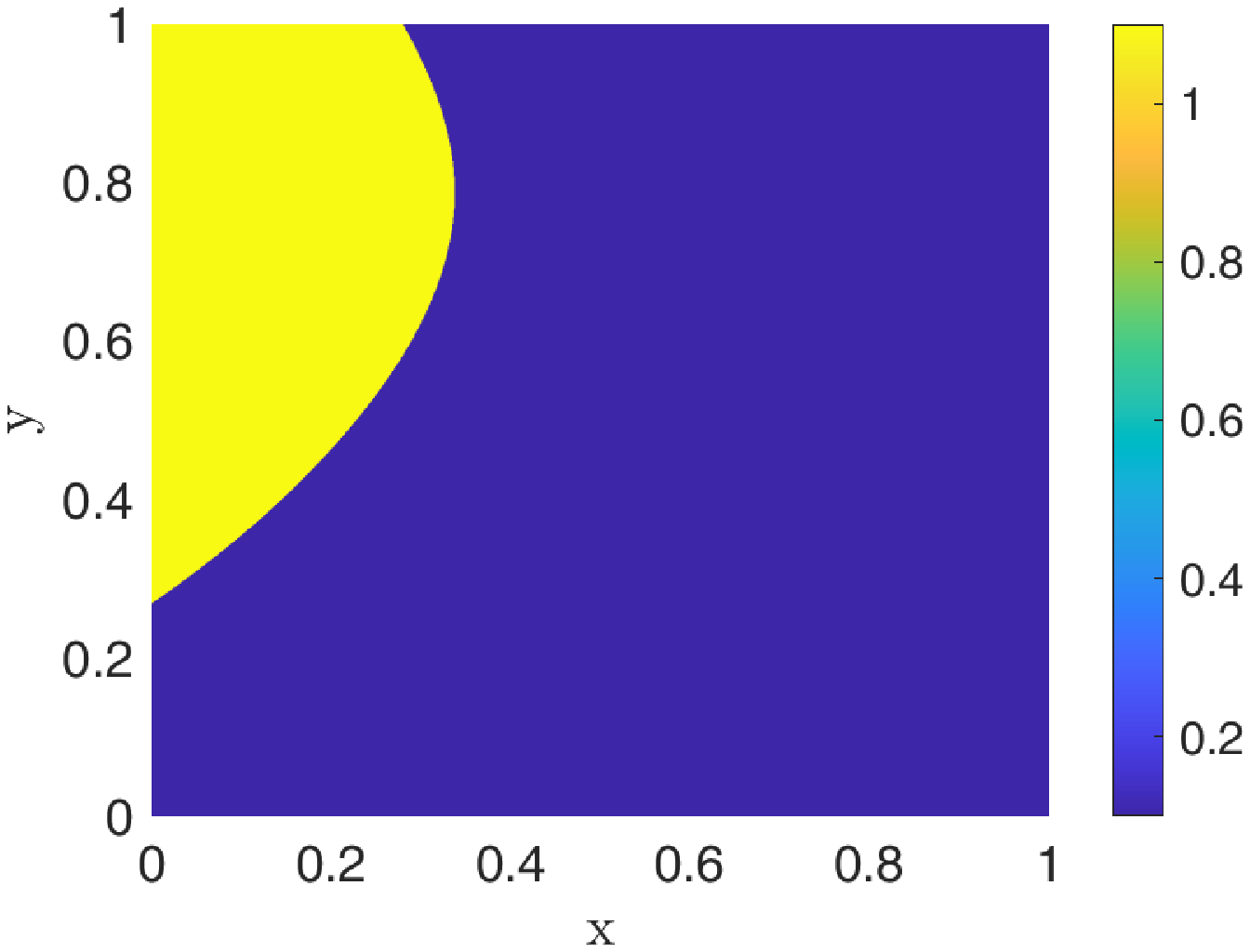}}
	\subfigure{\includegraphics[scale=0.45]{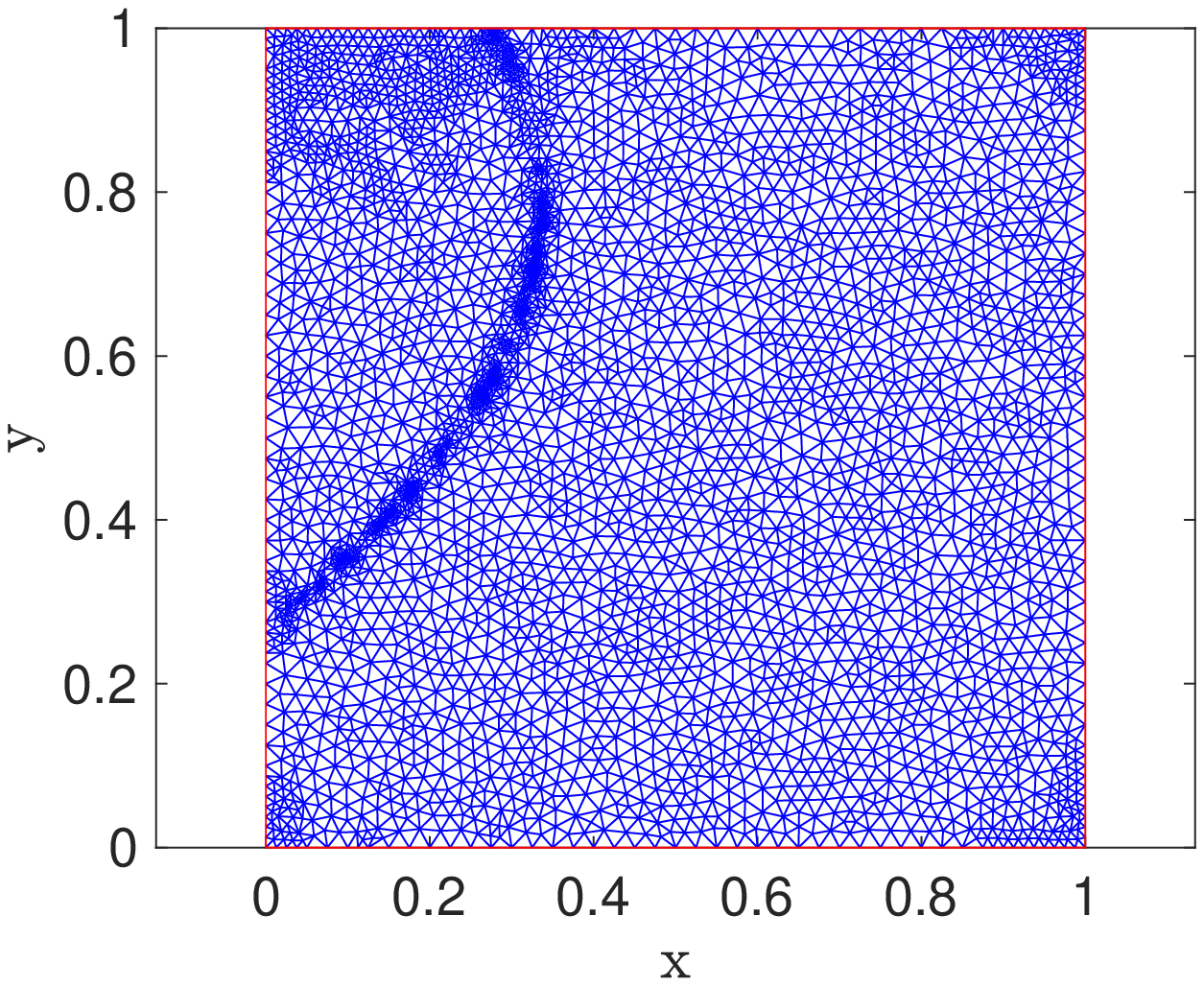}}
	\caption{Samples of the diffusion coefficient $a$ (left) and adaptive triangulations (right).}\label{Fig:AdaptMesh}
\end{figure}
Figure \ref{Fig:AdaptMesh} shows the three samples of the diffusion coefficient and the corresponding adaptive triangulations. It is nicely illustrated, how the element-wise a-posteriori estimation of the error according to Equation~\eqref{EQ:ErrContributionTriangle} enables an identification of the triangles which lie near the discontinuities of the diffusion coefficient and, hence, have a comparatively high error contribution. This results in a local refinement leading to a higher mesh resolution near the discontinuity. 
\end{example}  

\subsection{Numerical experiments for the random elliptic PDE}\label{subsec:NumExEllPDE}
    
In this section, we present numerical examples for the experimental investigation of the existence of a positive parameter $\kappa_a>0$ such that Equation~\eqref{EQ:ConvFEM} holds. Further, in the presented examples, we compare the performance of a standard FE discretization with the FE approach using the local refinement strategy described in Subsection \ref{subsubsec:AFEM}. We consider the domain $\mathcal{D}=(0,1)^2$ and set $f\equiv 10$, $\overline{a}\equiv 0.1$, $\Phi_1 = 0.1\,\exp(\cdot)$, $\Phi_2=|\cdot|$ , $F = \min(|\cdot|, 30)$ in Equation~\eqref{EQ:DiffCoeffDefi}. $W_1$ and $W_2$ are centered GRFs with squared exponential covariance function (see Subsection \ref{subsec:NumExCov}), where we set $\sigma=0.5$ and $r=1$. We use the circulant embedding method (see cf.  \cite{AnalysisOfCirculantEmbeddingMethodsForSamplingStationaryRandomFields} and  \cite{CirculantEmbeddingWithWMCAnalysisForEllipicPDEWithLognormalCoefficients}) to draw samples of the GRFs $W_1,~W_2$ on an equidistant grid with stepsize $h_W=1/200 = 0.005$ and obtain evaluations everywhere on the domain by multilinear interpolation. Due to the high spatial regularity of $W_1$ and $W_2$ and the high correlation length, this stepsize is fine enough to ensure that the approximation error of the GRFs $W_1,~W_2$ is negligible.  The Lévy process $l$ is set to be a Poisson process with intensity parameter $\lambda>0$, which will be specified later. It follows that $l(t)-l(s)\sim Poiss(\lambda(t-s))$ for $0\leq s\leq t$ and we draw samples from the Poisson process by the Uniform Method (see \cite[Section 8.1.2]{LevyProcessesInFinance}). 

The goal of our numerical experiments is to approximate the strong error:
We set $h_\ell=0.25\cdot 2^{-(\ell-1)}$, for $\ell=1,\dots, 8$ and approximate the left hand side of Equation~\eqref{EQ:ConvFEM} by 
\begin{align}\label{EQ:EllPDERMSEDefi}
RMSE^2 := \|u-u_\ell\|_{L^2(\Omega;V)}^2 \approx \frac{1}{M}\sum_{i=1}^M \|u_{ref}^{(i)}-u_\ell^{(i)}\|_V^2,
\end{align} 
for $\ell = 1,\dots,5$, where $(u_{ref}^{(i)}-u_\ell^{(i)}, ~i=1,\dots,M)$ are i.i.d. copies of the random variable $u_{ref}-u_\ell$ and $M\in\mathbb{N}$. We use a reference grid on $\mathcal{D}$ with $801$ grid points for interpolation and prolongation and take $u_{ref}:=u_{8}$ as the pathwise reference solution. The RMSE is estimated for the standard FE method and for the FE method with adaptive refinement as described in Subsection \ref{subsubsec:AFEM}. In order to obtain comparable approximations on each FE level, we compute the adaptive meshes as follows: for $\ell\in\{1,\dots,5\}$, we denote by $n_\ell$ the number of triangles in the non-adaptive triangulation with FE mesh refinement parameter $h_\ell$. The (sample-dependent) adaptive mesh on level $\ell$ is obtained by performing the local refinement procedure described in Subsection \ref{subsubsec:AFEM} until the number of triangles in the adaptive mesh exceeds $n_\ell$. The resulting mesh is then used to compute the adaptive FE approximation on level $\ell$.
  
\subsubsection{Homogeneous Dirichlet boundary conditions}
In our first experiment, we choose homogeneous Dirichlet boundary contitions on $\partial \mathcal{D}$ and set $\lambda=2$. Figure \ref{Fig:GSLPPDEHBCSamples} shows samples of the diffusion coefficient and the corresponding PDE solutions approximated by the FE method.
\begin{figure}[ht]
	\centering
	\subfigure{\includegraphics[scale=0.24]{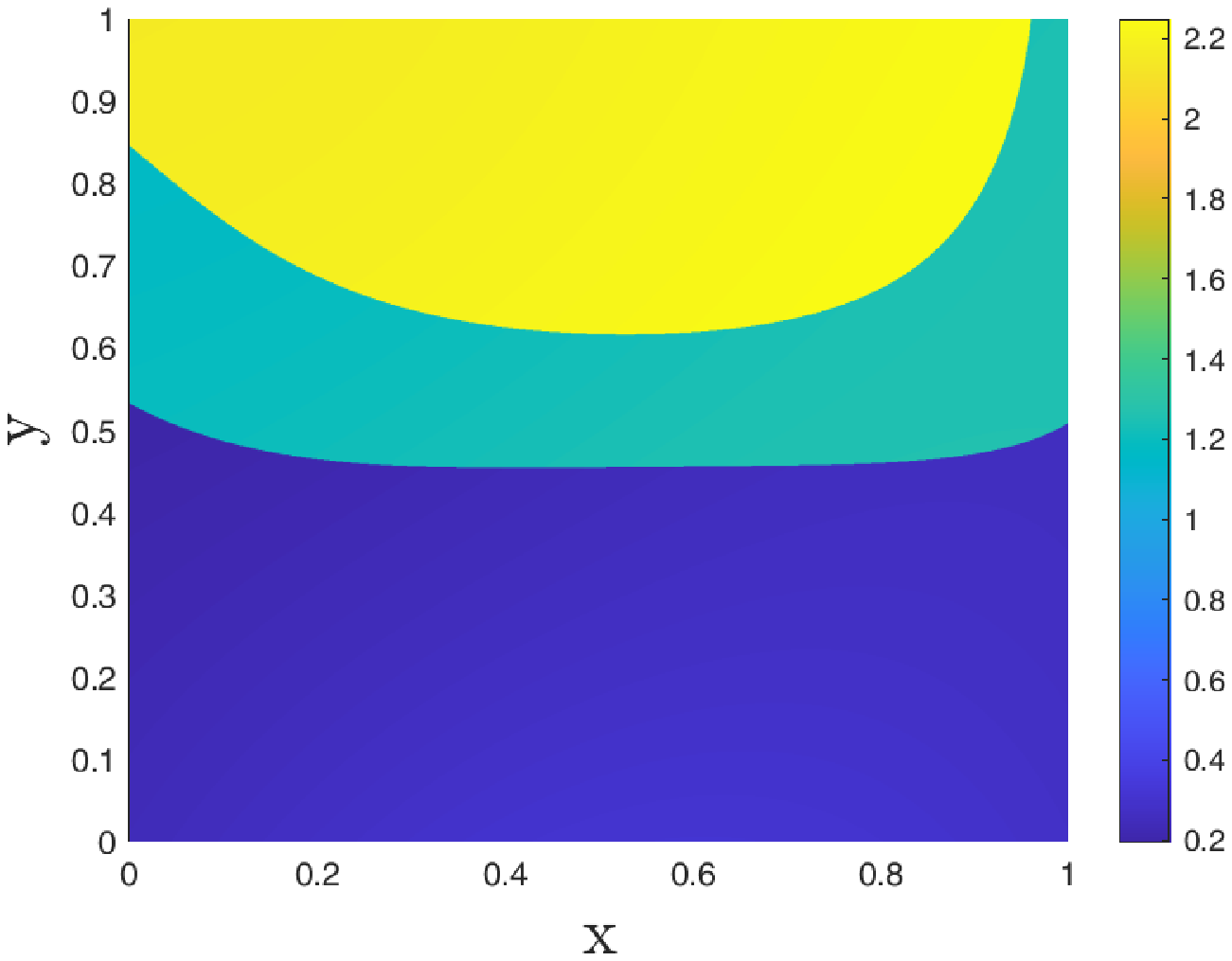}}
		\subfigure{\includegraphics[scale=0.24]{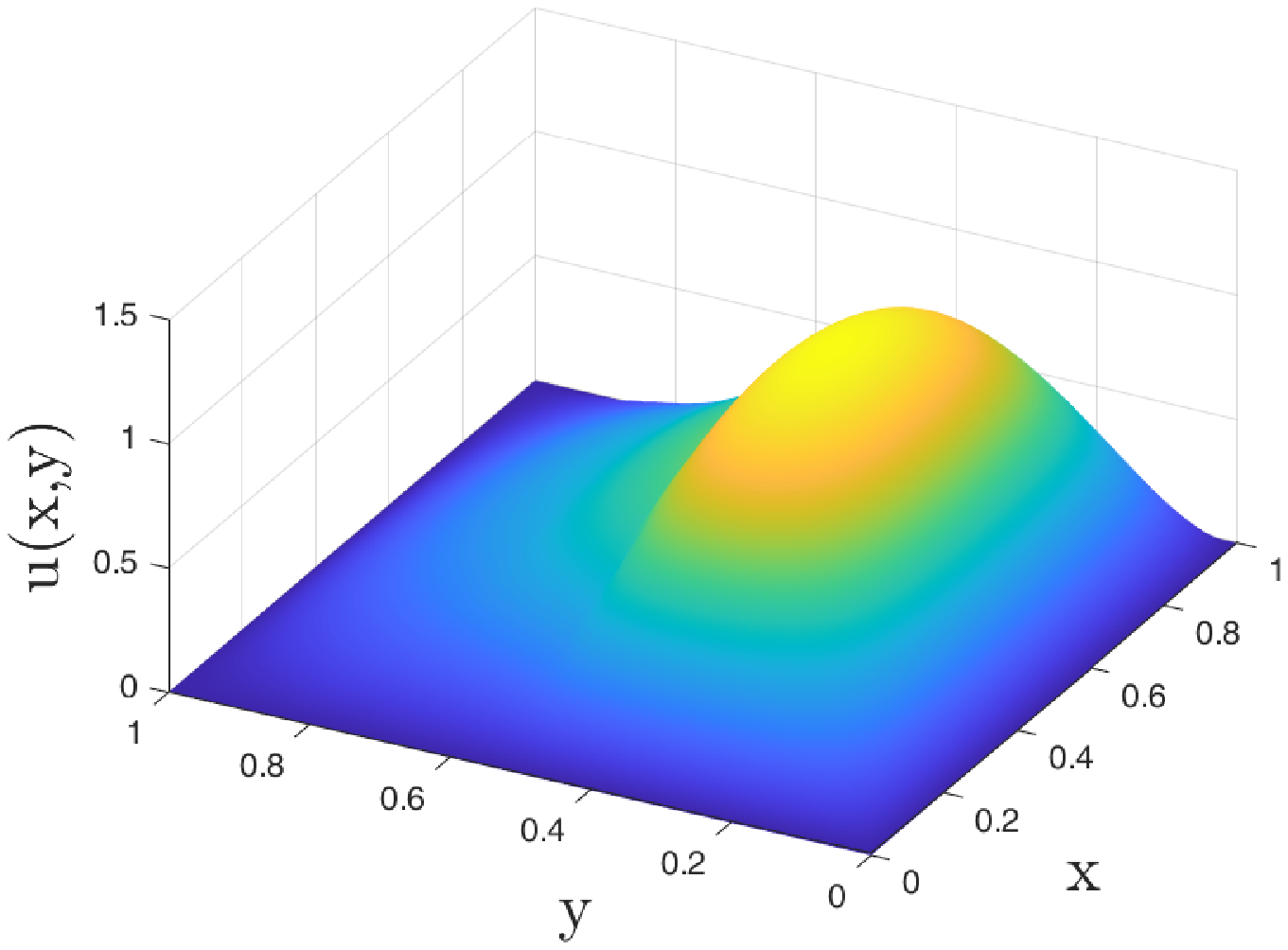}}
		\subfigure{\includegraphics[scale=0.24]{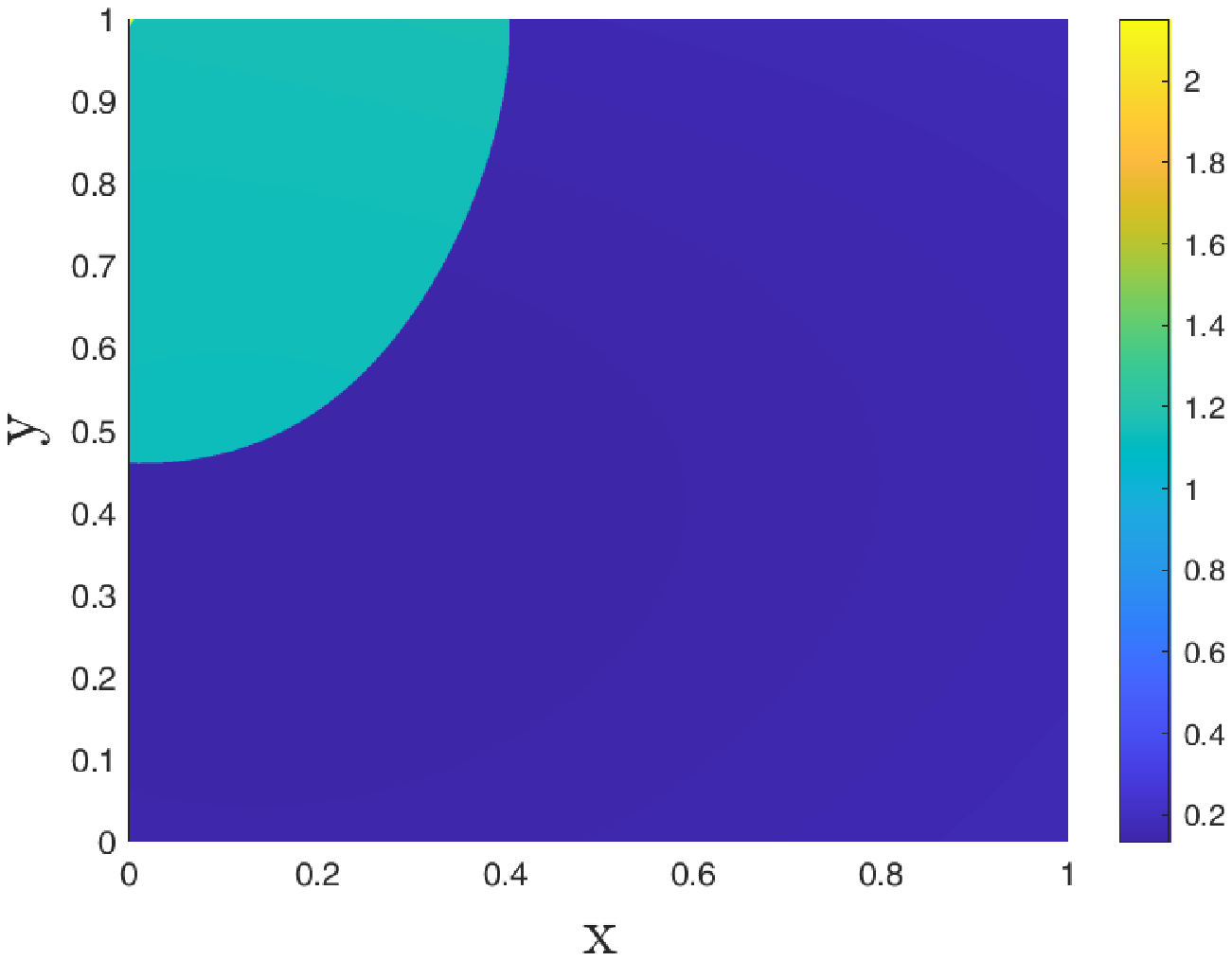}}
		\subfigure{\includegraphics[scale=0.24]{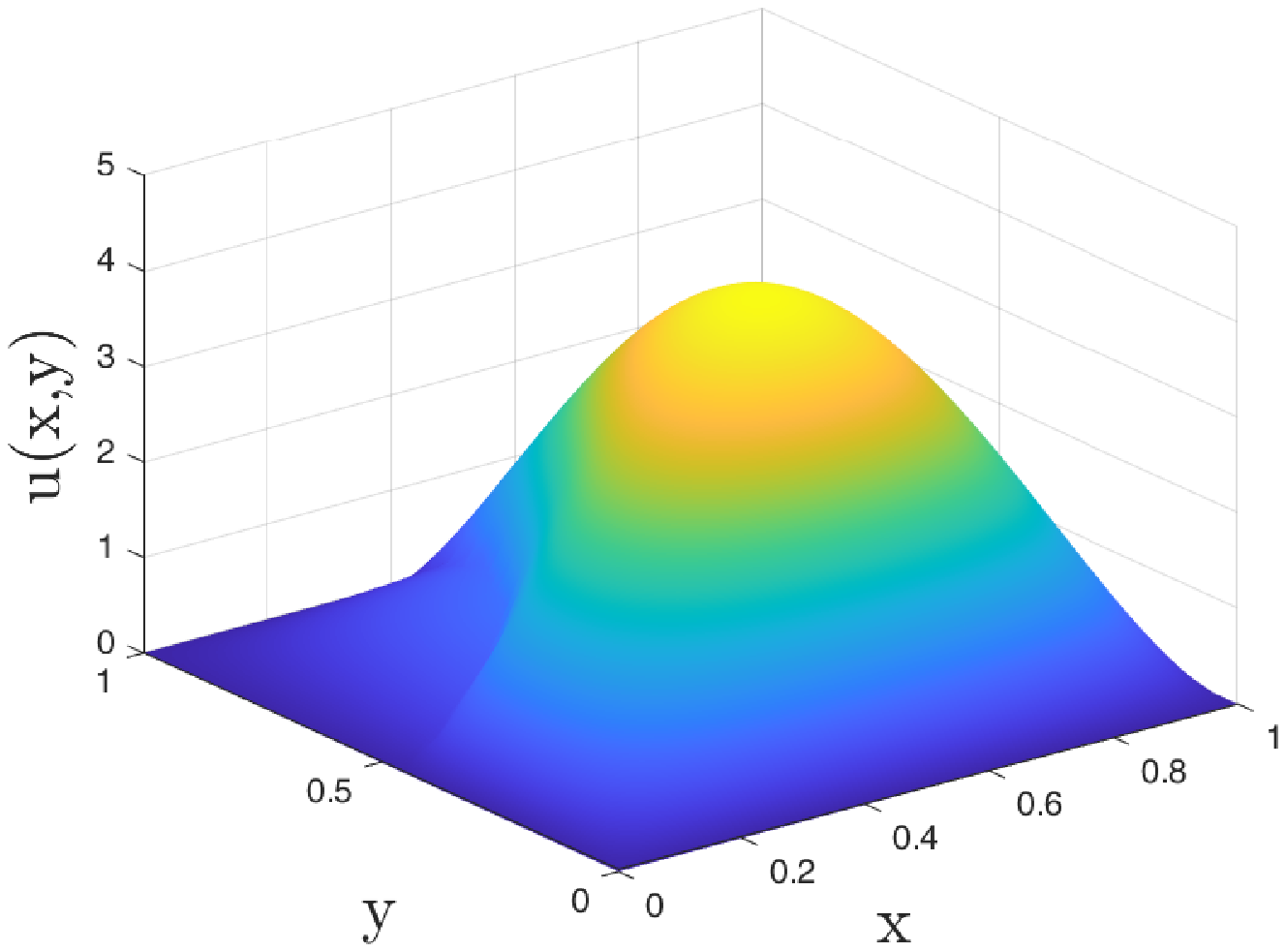}}
	\caption{Samples of the diffusion coefficient with $\lambda=2$ and corresponding FE solution to the elliptic PDE with homogeneous Dirichlet boundary conditions.}\label{Fig:GSLPPDEHBCSamples}
\end{figure}
We use $M=150$ samples so approximate the RMSE according to Equation~\eqref{EQ:EllPDERMSEDefi} with the non-adaptive and the adaptive FE approach. Figure \ref{Fig:ConvFEDBCP2} shows the approximated RMSE plotted against the inverse FE refinement parameter. For the adaptive FE approximations, the sample-dependent mesh refinement parameters on each discretization level are averaged over all $150$ samples.  We see a convergence with rate $\approx 0.65$ for the standard FE method, which is in line with our expectations (see Section \ref{subsec:FEMethod}). Further, we observe that the adaptive refinement strategy yields an improved convergence rate of $\approx 0.85$ and a smaller estimated RMSE on the considered levels.
\begin{figure}[ht]
	\centering
	\subfigure{\includegraphics[scale=0.45]{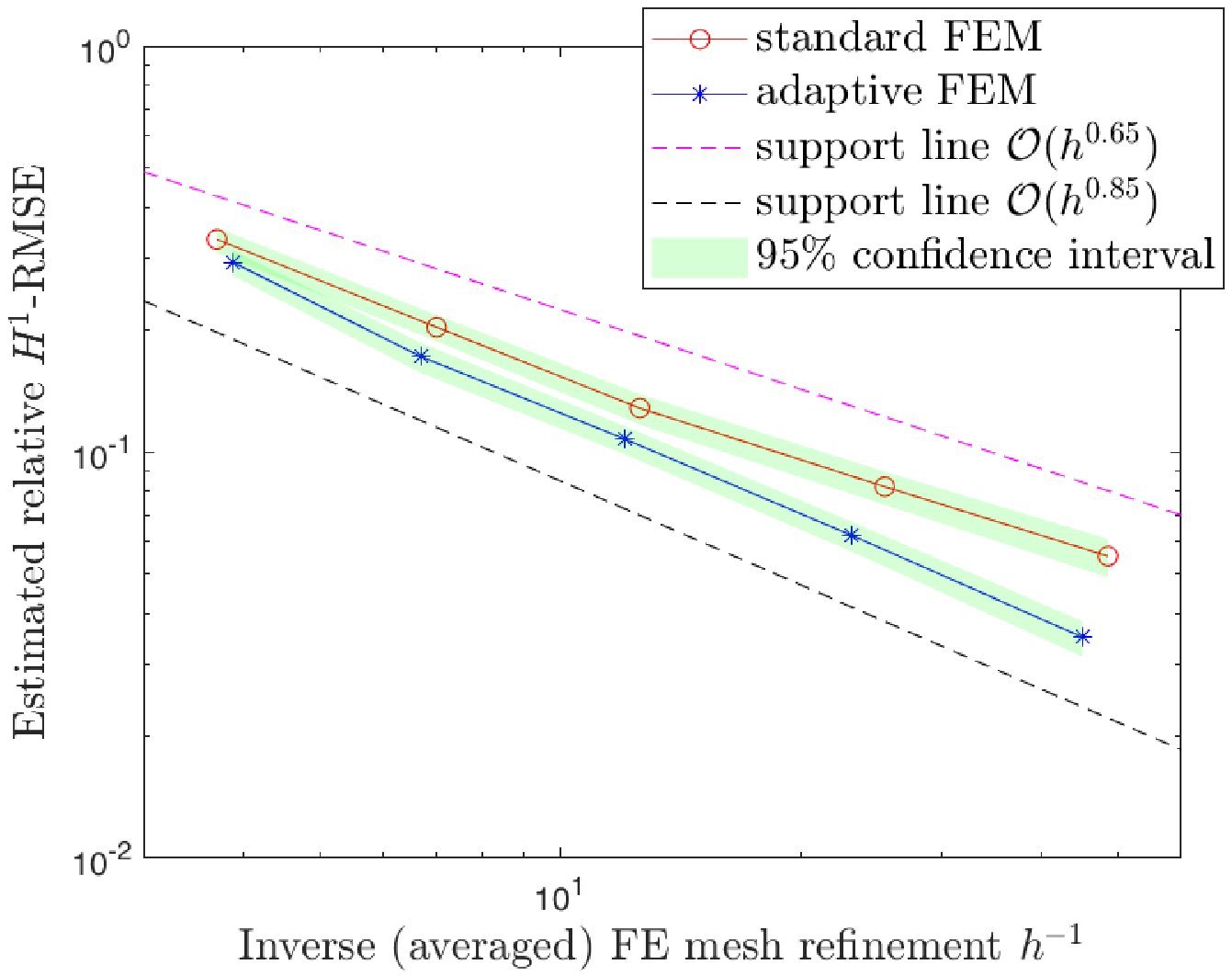}}
		\subfigure{\includegraphics[scale=0.45]{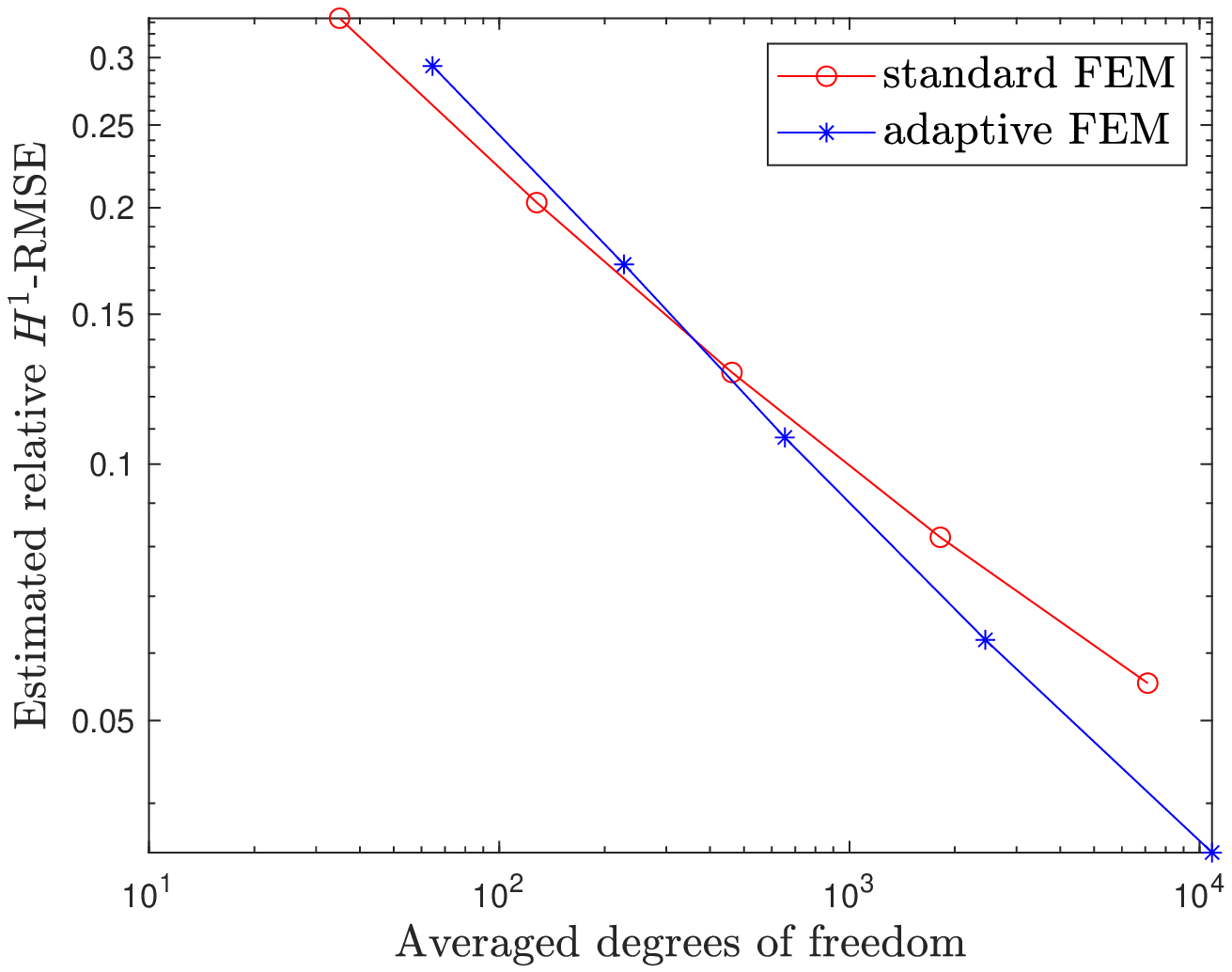}}
	\caption{Convergence of the (standard and adaptive) FE method (left) and degrees-of-freedom to error plot (right) where $l$ is a Poisson(2)-process and we impose homogeneous Dirichlet boundary conditions.}\label{Fig:ConvFEDBCP2}
\end{figure}
The right hand side of Figure \ref{Fig:ConvFEDBCP2} shows the estimated RMSE plotted against the degrees of freedom of the linear FE-system on each discretization level. For the adaptive FE method, the degrees of freedom are averaged over all samples. After a pre-asymptotic behaviour on the first and second level, we see that the adaptive FE method performs more efficient in terms of the degrees of freedom: the adaptive FE method achieves a certain RMSE with less degrees of freedom compared to the standard FE method.

\subsubsection{Mixed Dirichlet-Neumann boundary conditions}
In the second numerical example, we use mixed Dirichlet-Neumann boundary conditions: we split the domain boundary $\partial\mathcal{D}$ by $\Gamma_1=\{0,1\}\times[0,1]$ and $\Gamma_2=(0,1)\times\{0,1\}$ and impose the following pathwise boundary conditions
\begin{align*}
u(\omega,\cdot)=\begin{cases} 0.1  & ~on~ \{0\}\times [0,1] \\ 0.3 &~on~\{1\}\times [0,1]  \end{cases} \text{ and } a\overrightarrow{n}\cdot\nabla u=0 \text{ on } \Gamma_2,
\end{align*}
for $\omega\in\Omega$. Further, we set $\lambda=3$, which results in a higher number of jumps in the diffusion coefficient. Figure \ref{Fig:GSLPPDEMBCSamples} shows samples of the diffusion coefficient and the corresponding PDE solutions approximated by the FE method.
\begin{figure}[ht]
	\centering
	\subfigure{\includegraphics[scale=0.24]{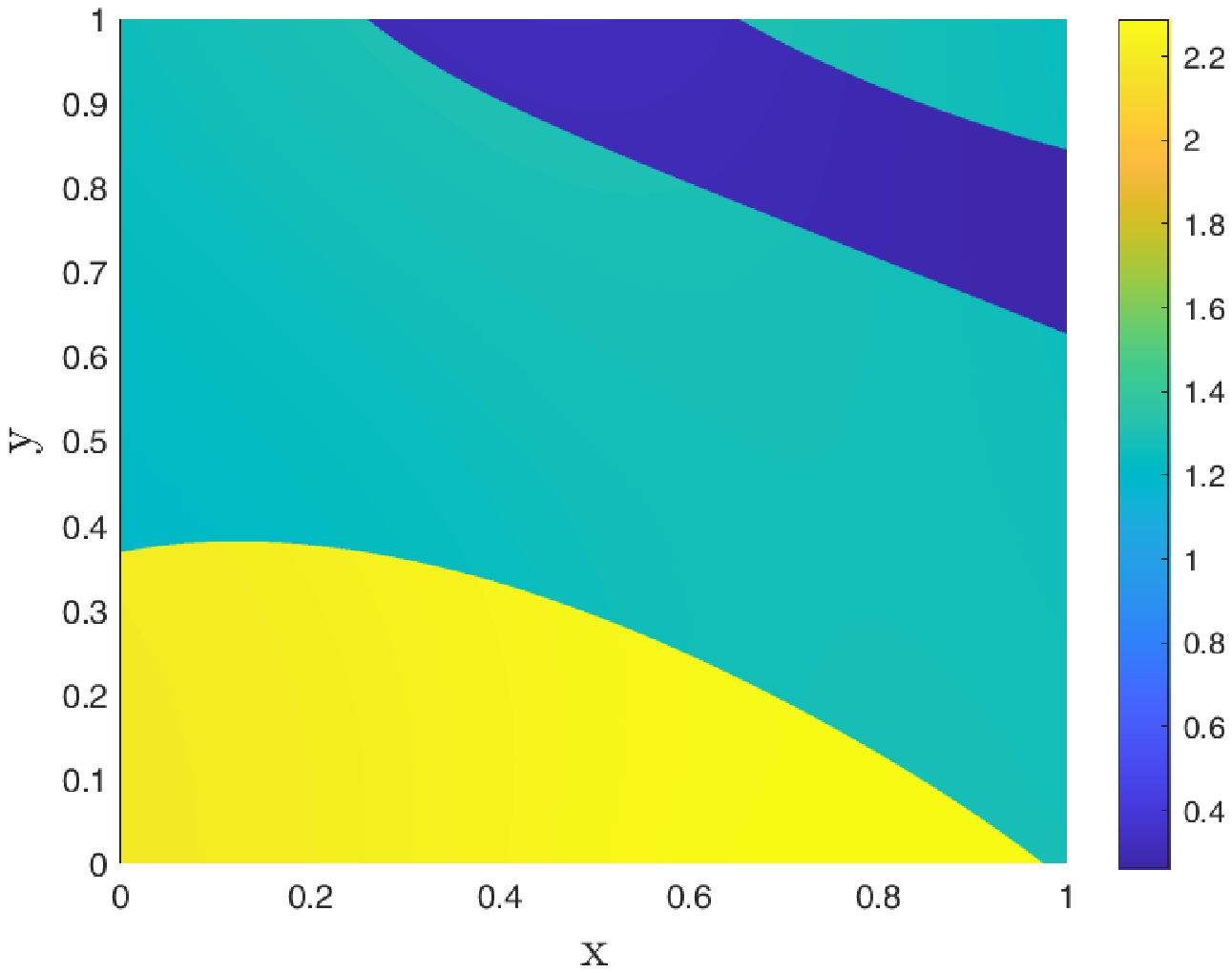}}
		\subfigure{\includegraphics[scale=0.24]{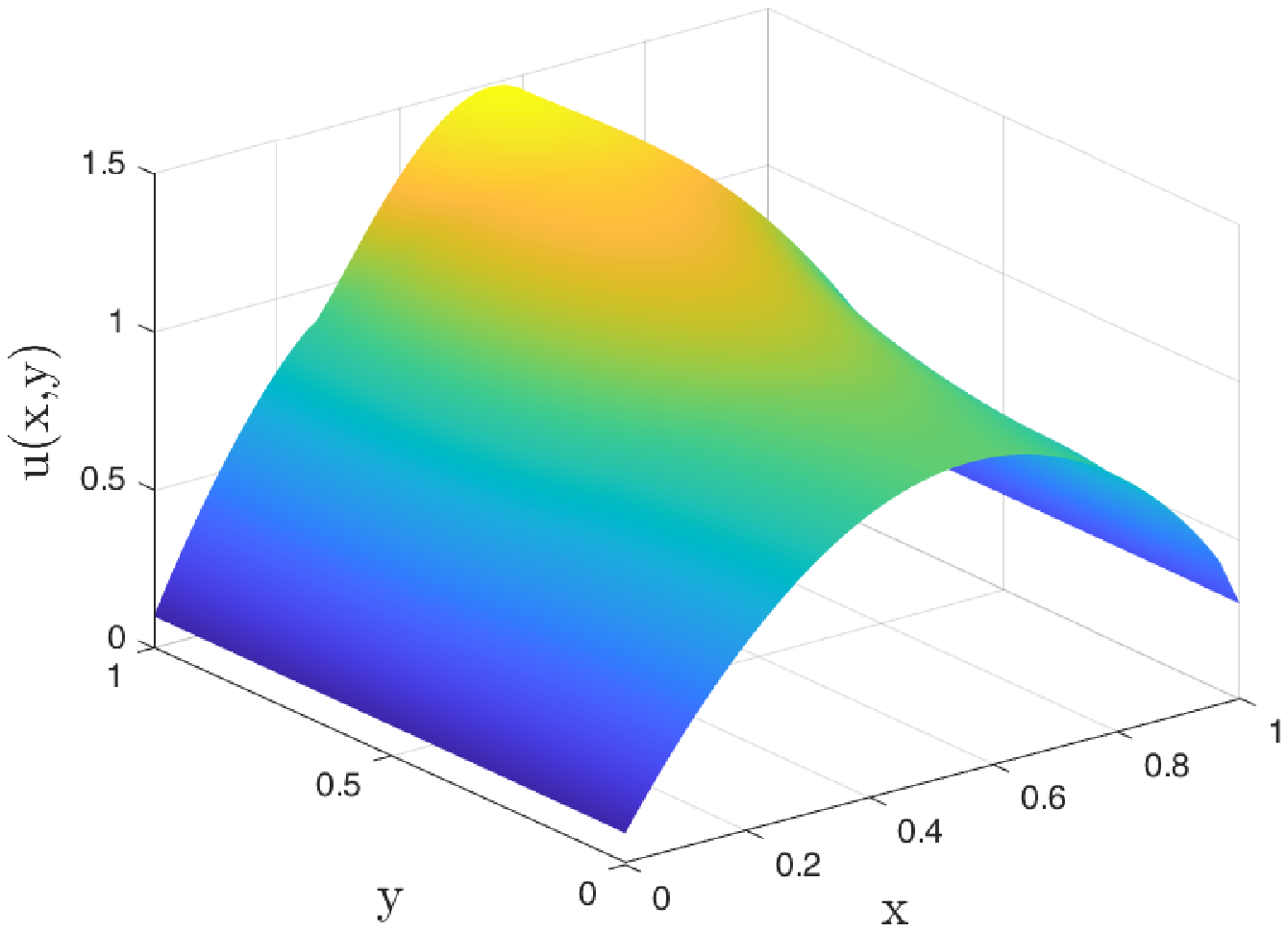}}
		\subfigure{\includegraphics[scale=0.24]{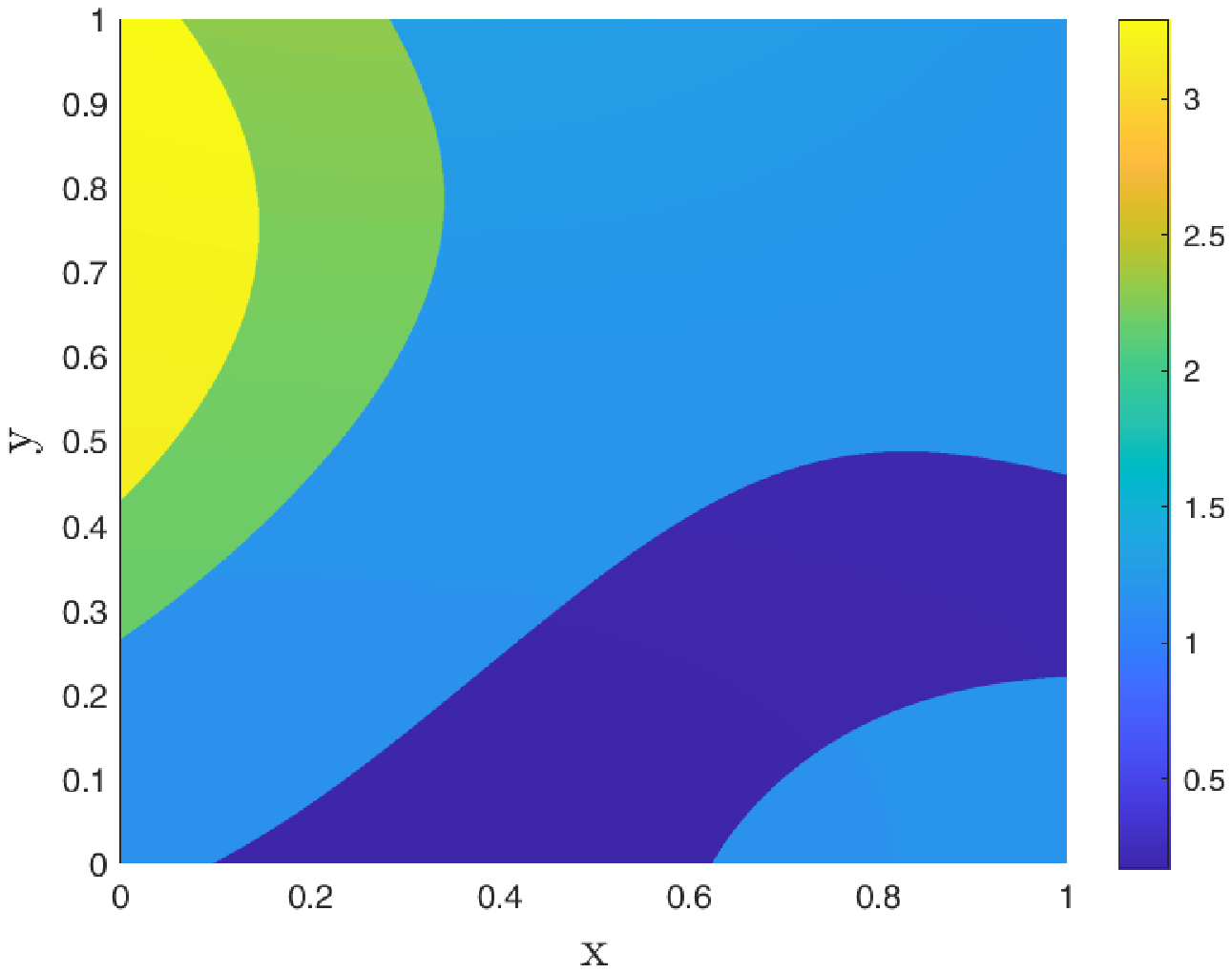}}
		\subfigure{\includegraphics[scale=0.24]{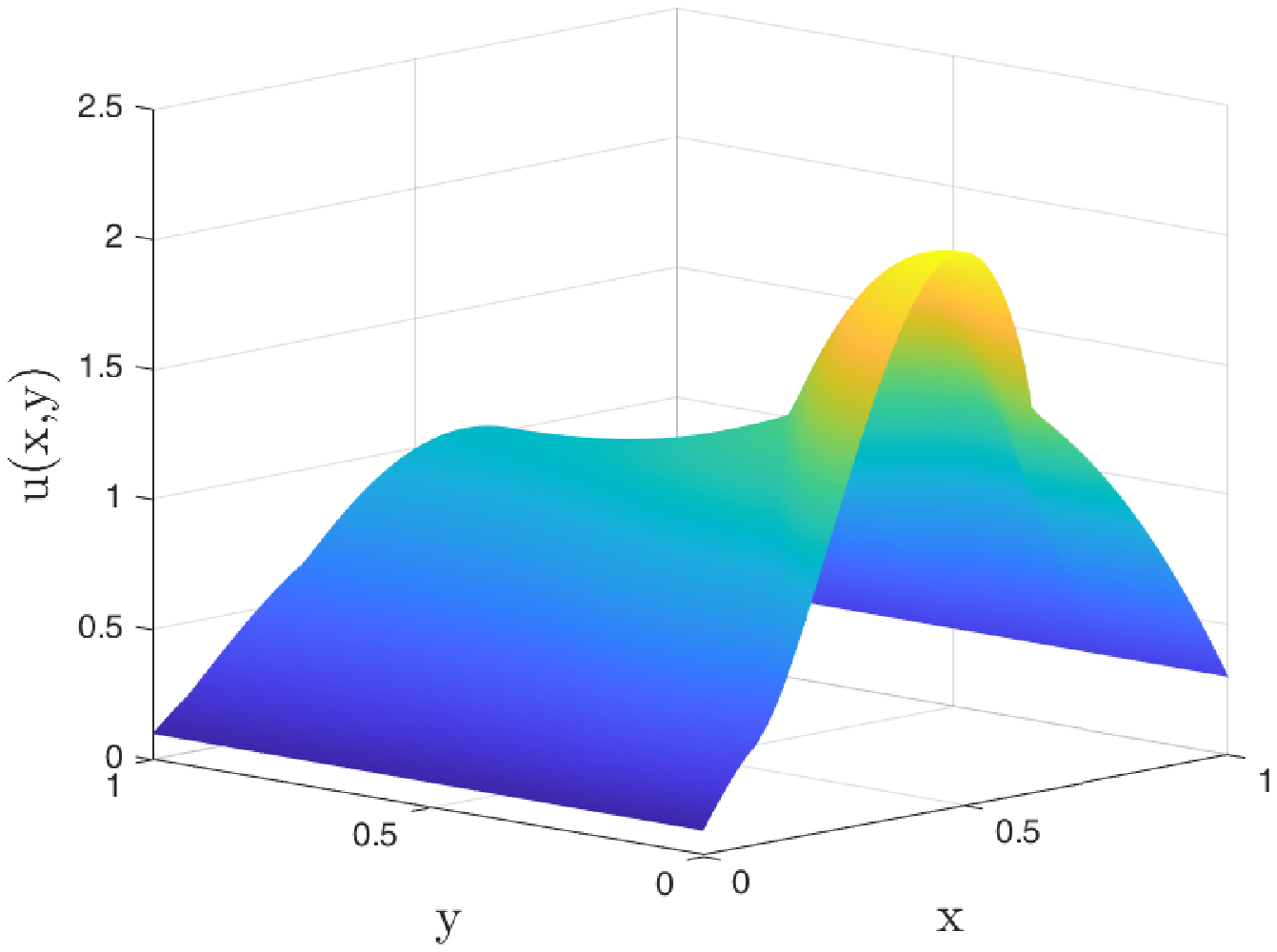}}
	\caption{Samples of the diffusion coefficient with $\lambda=3$ and corresponding FE solution to the elliptic PDE with mixed Dirichlet-Neumann boundary conditions.}\label{Fig:GSLPPDEMBCSamples}
\end{figure}
As in the first experiment, we use $M=150$ samples to approximate the RMSE according to \eqref{EQ:ConvFEM} with the standard FE approach and the adaptive approach. The approximated values are plotted against the  inverse FE refinement parameter. The results are presented in Figure \ref{Fig:ConvFEMBCP3}, which shows a convergence rate of $\approx 0.6$ for the standard FE approach, which is slightly smaller than the observed convergence rate of $\approx 0.65$ in the first numerical example, where we imposed homogeneous Dirichlet boundary conditions and considered a Poisson process with intensity parameter $\lambda = 2$, leading to a smaller expected number of jumps in the diffusion coefficient. Further, Figure \ref{Fig:ConvFEMBCP3} shows a convergence rate of $\approx 0.85$ for the adaptive FE approach and smaller magnitudes of the RMSE compared to the standard FE method. The right plot of Figure \ref{Fig:ConvFEMBCP3} reveals the higher efficiency of the adaptive FE method compared to the standard approach in the sense that the number of degrees of freedom, which are necessary to achieve a certain error, is significantly smaller compared to the standard FE approach. Further, we see that the performance difference of the standard and the adaptive FE method is larger compared to the first example due to the higher number of expected jumps in the diffusion coefficient, which renders the adaptive local refinement strategy even more suitable for this problem. Overall, the results are in line with our expectations. 
\begin{figure}[ht]
	\centering
	\subfigure{\includegraphics[scale=0.45]{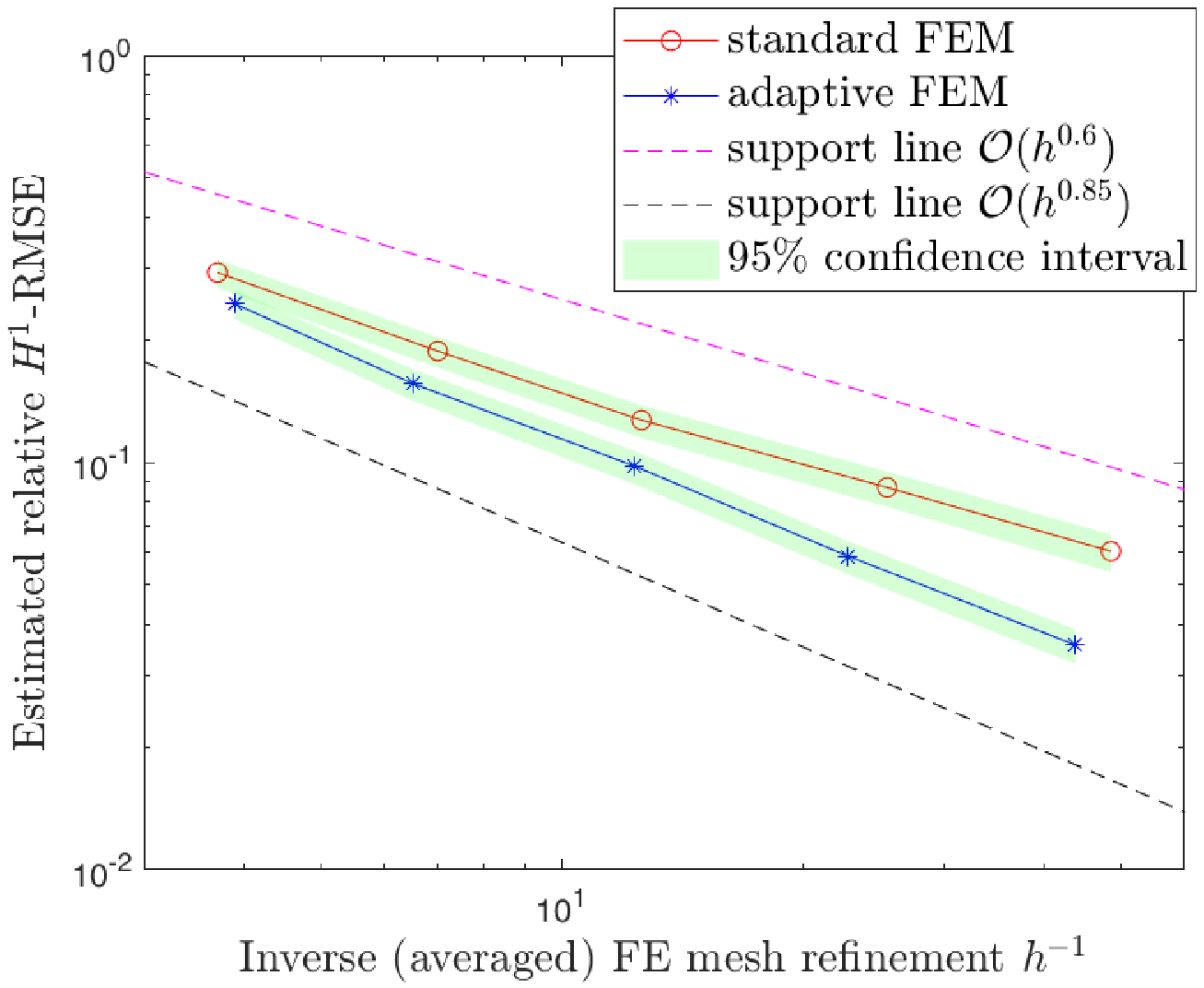}}
		\subfigure{\includegraphics[scale=0.45]{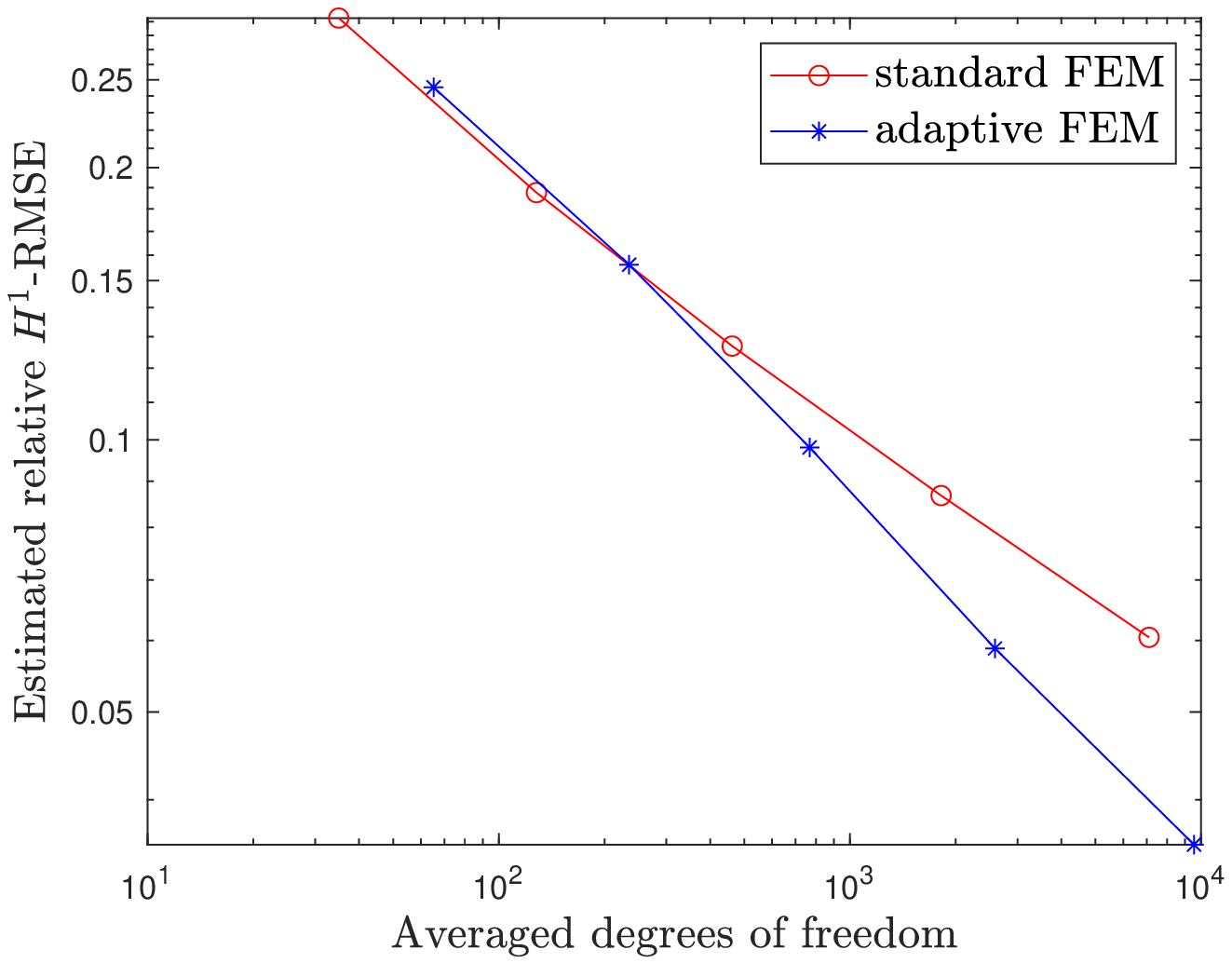}}
	\caption{Convergence of the (standard and adaptive) FE method (left) and degrees-of-freedom to error plot (right) where $l$ is a Poisson(3)-process and we impose mixed Dirichlet-Neumann boundary conditions.}\label{Fig:ConvFEMBCP3}
\end{figure}

\subsubsection{Mixed Dirichlet-Neumann boundary conditions and jump-accentuated coefficients}
In our last experiment we consider mixed Dirichlet-Neumann boundary conditions as in the previous section. The diffusion coefficient is set to be 
\begin{align*}
a(x,y) = 0.01 + 0.01\,\exp(W_1(x,y)) + 30\,l(F(W_2(x,y)),
\end{align*}
for $(x,y)\in\mathcal{D}$, where $l$ is a Poisson(4)-process. This leads to a jump-accentuated coefficient with high contrast. We take $M=500$ samples to estimate the RMSE for the standard FE method and the adaptive approach according to \eqref{EQ:EllPDERMSEDefi}. The results are presented in Figure \ref{Fig:ConvFEMBCP4}.
\begin{figure}[ht]
	\centering
	\subfigure{\includegraphics[scale=0.45]{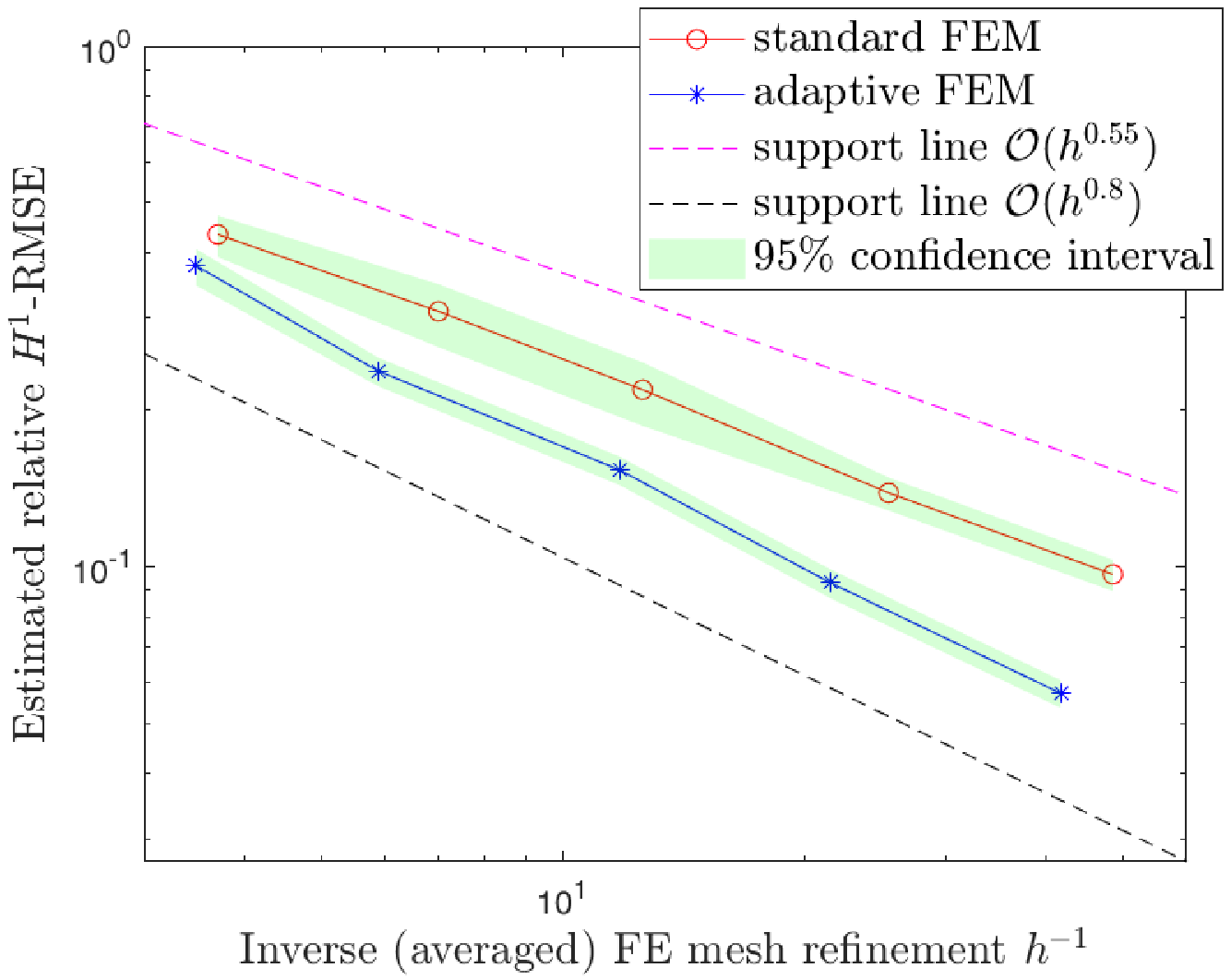}}
		\subfigure{\includegraphics[scale=0.45]{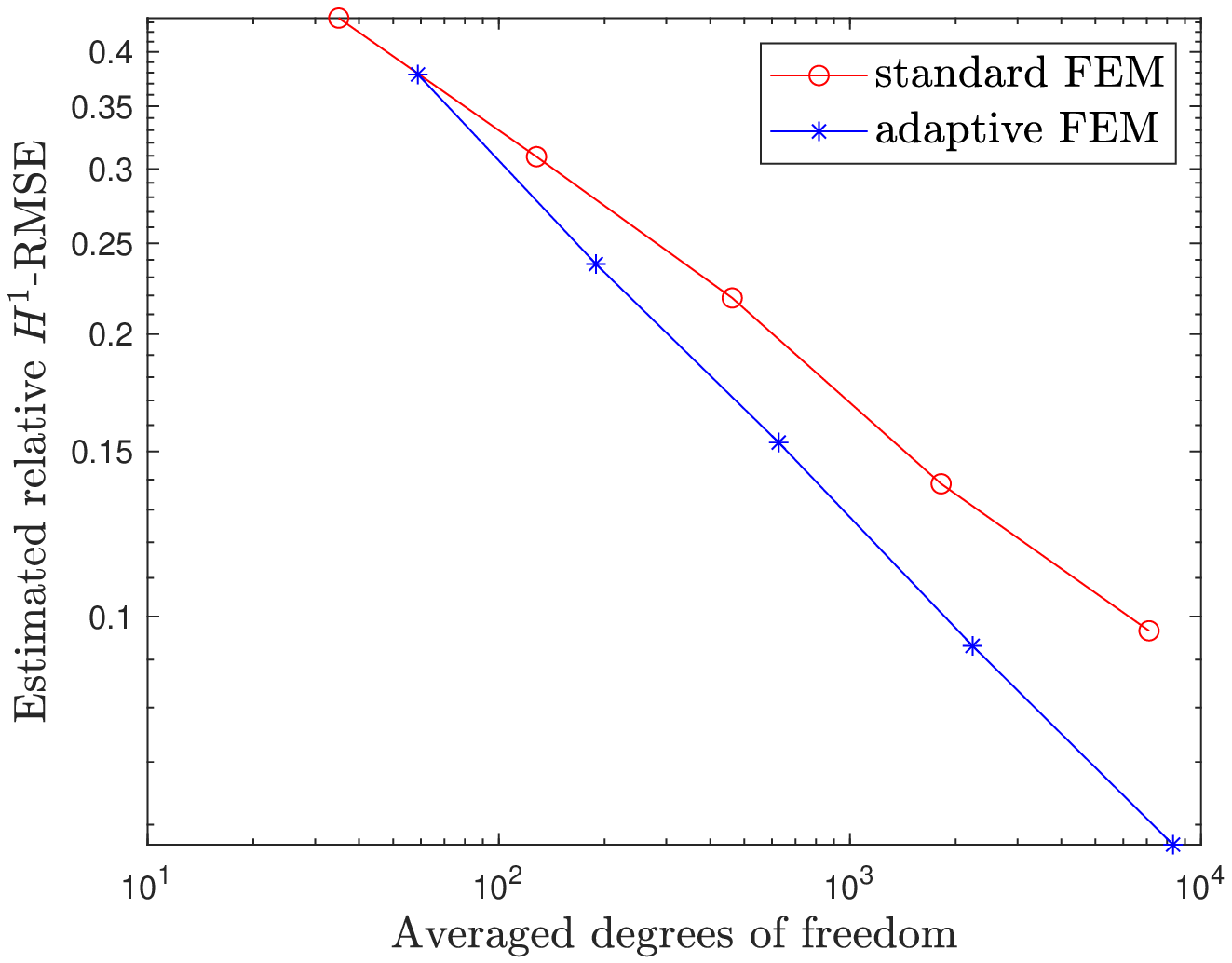}}
	\caption{Convergence of the (standard and adaptive) FE method (left) and degrees-of-freedom to error plot (right) where $l$ is a Poisson(4)-process, we impose mixed Dirichlet-Neumann boundary conditions and use a high-contrast diffusion coefficient.}\label{Fig:ConvFEMBCP4}
\end{figure} 
We obtain a convergence rate of $\approx 0.55$ for the standard FE approach and $\approx 0.8$ for the adaptive FE method, which is slightly slower than the observed rates in the previous experiment. This is in line with our expectations since we have an increased jump intensity in the underlying Poisson process and larger jump heights in the diffusion coefficient, both having a negative influence on the convergence rate of the FE method (see also \cite{SGRFPDE, AStudyOfElliptic, RegularityResultsForLaplaceInterfaceProblemsInTroDimensions}). Figure \ref{Fig:ConvFEMBCP4} also reveals, as expected, that the magnitude of the RMSE is significantly smaller in the adaptive FE approach. The right plot of Figure \ref{Fig:ConvFEMBCP4} demonstrates that the adaptive refinement strategy is able to achieve a certain error with significantly less degrees of freedom in the corresponding linear FE system, than the standard FE approach.

    %%%%%%%%%%%%%%%%%%%%%%%%%%%%%%%%%%%%%%%%%%%%%%%%%%%%%%%%%%%%%%%%%%%%%%%%%%%%%%%%%
	%%%%%%%%%%%%%%%%%%%%%%%%%%%%%%%%%%%%%%%%%%%%%%%%%%%%%%%%%%%%%%%%%%%%%%%%%%%%%%%%%
	%% Acknowledgments
	%%%%%%%%%%%%%%%%%%%%%%%%%%%%%%%%%%%%%%%%%%%%%%%%%%%%%%%%%%%%%%%%%%%%%%%%%%%%%%%%%
	%%%%%%%%%%%%%%%%%%%%%%%%%%%%%%%%%%%%%%%%%%%%%%%%%%%%%%%%%%%%%%%%%%%%%%%%%%%%%%%%%

	\section*{Acknowledgments}
	Funded by Deutsche Forschungsgemeinschaft (DFG, German Research Foundation) under Germany's Excellence Strategy - EXC 2075 - 390740016.
	
	%%%%%%%%%%%%%%%%%%%%%%%%%%%%%%%%%%%%%%%%%%%%%%%%%%%%%%%%%%%%%%%%%%%%%%%%%%%%%%%%%
	%%%%%%%%%%%%%%%%%%%%%%%%%%%%%%%%%%%%%%%%%%%%%%%%%%%%%%%%%%%%%%%%%%%%%%%%%%%%%%%%%
	%% Bibliography
	%%%%%%%%%%%%%%%%%%%%%%%%%%%%%%%%%%%%%%%%%%%%%%%%%%%%%%%%%%%%%%%%%%%%%%%%%%%%%%%%%
	%%%%%%%%%%%%%%%%%%%%%%%%%%%%%%%%%%%%%%%%%%%%%%%%%%%%%%%%%%%%%%%%%%%%%%%%%%%%%%%%%

	\bibliographystyle{siam}
	\bibliography{references}

\end{document}